\documentclass[graybox]{svmult}

\usepackage[english]{babel}  
\usepackage{graphicx}        % standard LaTeX graphics tool
\usepackage[bottom]{footmisc}% places footnotes at page bottom

\usepackage[numbers]{natbib}
\usepackage[intlimits]{amsmath}
\usepackage{amssymb}
\allowdisplaybreaks[1]
\usepackage{amscd}
\usepackage{empheq}
\usepackage{stmaryrd}
\usepackage{fancyhdr}
\let\oldmarginpar\marginpar
\renewcommand\marginpar[1]{\-\oldmarginpar[\raggedleft\footnotesize #1]{\raggedright\footnotesize #1}}
\usepackage[svgnames,dvipsnames,hyperref]{xcolor}
\usepackage[hyphens]{url}
\usepackage[hyperfootnotes=true]{hyperref}
\usepackage{breakurl}
\hypersetup{colorlinks=true, linktoc=page}
\hypersetup{citecolor=RoyalBlue, linkcolor=cyan, urlcolor=PineGreen, filecolor=green}

\smartqed
\newcommand{\ves}[1]{\mathcal{X}_{#1}}
\newcommand{\dfs}[1]{\mathcal{F}^{#1}}
\renewcommand{\vec}[1]{\mathbf#1}
\newcommand{\metric}[2]{\langle#1,#2\rangle}
\newcommand{\ltwo}[4]{\metric{#1}{#2}_{L^2(#4)}}
\newcommand{\bigltwo}[4]{\bigl\langle{#1},{#2}\bigr\rangle_{L^2(#4)}}
\newcommand{\ltwonorm}[3]{|#1|_{L^2(#3)}}
\newcommand{\bigltwonorm}[3]{\bigl|#1\bigr|_{L^2(#3)}}
\newcommand{\dualprod}[2]{\langle#1|#2\rangle}
\newcommand{\bigdualprod}[2]{\big\langle{#1}\big|{#2}\big\rangle}
\newcommand{\pslnot}{\boldsymbol{\Psi}_{\mathrm{SL}}}
\newcommand{\pdlnot}{\boldsymbol{\Psi}_{\mathrm{DL}}}
\newcommand{\psl}[1]{\pslnot}
\newcommand{\pdl}[1]{\pdlnot}
\newcommand{\tpsl}[1]{\widetilde{\boldsymbol{\Psi}}_{\mathrm{SL}}}
\newcommand{\jump}[1]{\llbracket #1\rrbracket}
\newcommand{\mean}[1]{\{#1\}}
\newcommand{\ppoint}{X}
\newcommand{\hhat}[1]{#1}
\newcommand{\bdryform}{\varphi}
\newcommand{\supp}{\operatorname{supp}} 
\newcommand{\traceadapt}{\widetilde{\gamma}}
\newcommand{\mywedge}{{\mathsf{\Lambda}}}
\newcommand{\pcomp}[1]{\mywedge^{#1}}
\newcommand{\hzero}{H}
\begin{document}
\title*{Differential Forms and Boundary Integral Equations for Maxwell-Type Problems}
\author{Stefan Kurz and Bernhard Auchmann}
\institute{Stefan Kurz \at Robert Bosch GmbH, Corporate Sector Research and Advance Engineering, \\D-70465 Stuttgart, Germany, \email{stefan.kurz2@de.bosch.com}
\and Bernhard Auchmann \at CERN, TE/MPE, CH-1211 Geneva 23, Switzerland, \email{bernhard.auchmann@cern.ch}}
\maketitle
\abstract{We present boundary-integral equations for Maxwell-type problems in a differential-form setting. Maxwell-type problems are governed by the differential equation $(\delta\mathrm{d}-k^2)\omega = 0$, where $k\in\mathbb{C}$ holds, subject to some restrictions. This problem class generalizes $\textbf{curl}\,\textbf{curl}$- and $\mathrm{div}\,\textbf{grad}$-types of problems in three dimensions. The goal of the paper is threefold: 1) Establish the Sobolev-space framework in the full generality of differential-form calculus on a smooth manifold of arbitrary dimension and with Lipschitz boundary. 2) Introduce integral transformations and fundamental solutions, and derive a representation formula for Maxwell-type problems. 3) Leverage the power of differential-form calculus to gain insight into properties and inherent symmetries of boundary-integral equations of Maxwell-type.}
\begin{acknowledgement}
The authors would like to thank Ralf Hiptmair for feedback and valuable suggestions.
\end{acknowledgement}
\tableofcontents
\section{Introduction}\label{sec:introduction}
It is the goal of this paper to express the theory of boundary-integral equations for Maxwell-type problems in the language of differential-form calculus. Maxwell-type problems are governed by the differential equation
\begin{equation*}
(\delta\mathrm{d}-k^2)\omega = 0,
\end{equation*}
where $k\in\mathbb{C}$ fulfills either $k=0$ or $0\le\arg k<\pi,k\ne 0$ \cite[eq.\ (9.13)]{mclean}. The exterior derivative $\mathrm{d}$ and coderivative $\delta$ will be defined in Section~\ref{sec:basics}. This problem class generalizes $\textbf{curl}\,\textbf{curl}$- and $\mathrm{div}\,\textbf{grad}$-types of problems in three dimensions. It encompasses electro- and magnetostatics (potential problems), eddy-current and diffusion-type problems, as well as scattering problems.\par
In the authors' view, differential-form calculus features a range of advantages over classical vector analysis, that are particularly interesting in the field of boun\-dary-integral equations. We give four examples: (i) Being independent of dimension, operators of the same class act upon fields on the domain and on the boundary. (ii) For a comprehensive treatment of the subject, only two families of functional spaces are required on the domain and on the boundary, respectively. The two families are related via Hodge duality. (iii) Involved computations with cross-products of normal vectors and tangent vectors are replaced by more elegant tools. (iv) A discretization of the functional spaces in terms of discrete differential forms is readily available and, in fact, an integral part of the differential-form setting. In this context {\sc Hiptmair} writes in \cite[p.\ 239ff.]{df5}: "Suitable finite elements for electromagnetic fields should be introduced and understood as discrete differential forms. ... Finite elements that lack an interpretation as discrete differential forms have to be used with great care.". For establishing spaces of discrete differential forms on two-dimensional surfaces we also point to \cite[Sec.\ 4.1.]{Buffa2007}.\par
The reader will find that, in many ways, the theory and proofs outlined in this paper are reminiscent of vector-analysis literature. This is not surprising, since a major part of our work consisted in translating classical proofs to the more general differential-form setting. In other places, presumably well-known subjects may look strangely unfamiliar. Study of the theory from the viewpoint of differential-form calculus reveals structural layers that are often hidden or obscured by the nature of vector analysis. For examples we point to the definition of generalized integral transforms, the image spaces of Sobolev spaces under the Hodge operator, or the symmetry of Calder\'on projectors under dual transformations. We hope that, with this work, we can help to spark the curiosity for differential-form calculus in the community, and do our share to lay the groundwork for future progress in the field. After all, G.-C. Rota wrote \cite[p.\ 46]{rota}, "Exterior algebra is not meant to prove old facts, it is meant to disclose a new world."\par
From a historical perspective, the idea to generalize Maxwell's equations, using $p$-forms in $n$-dimensional Euclidean space, was first put forward in a seminal paper by {\sc Weyl} in 1952 \cite{weyl1952}. Comparable work for the static case, that is, for potential problems, was accomplished by {\sc Kress} in 1972 \cite{kress1972}. Related work about higher dimensional electromagnetic scattering on Lipschitz domains in $\mathbb{R}^n$ was published by {\sc Jawerth} and {\sc Mitrea} in 1995 \cite{jawerth1995}. Recently, {\sc Pauly} has published a series of papers, where the low frequency asymptotics for generalized Maxwell equations have been examined under rather general assumptions \cite{pauly2006}.\par
In Section~\ref{sec:prelim} we give a concise summary of relevant topics of differential-form calculus. The section also includes contributions on topics such as integral transformations, and fundamental solutions of Helmholtz-type equations. So-called translation isomorphisms are introduced, that carry the differential-form setting in three-dimensional Euclidean space over to the classical vector-analysis setting. Section~\ref{sec:sobolevspacesdf} presents a differential-form based Sobolev-space framework that sets the scene for the discussion of Maxwell-type problems, their solutions, and boundary data. The section builds upon a 2004 work by {\sc Weck}~\cite{weck2004}. Translation isomorphisms are used to establish the link with Sobolev spaces in classical calculus. Section~\ref{sec:represent} is devoted entirely to the representation formula for Maxwell-type problems. The results generalize the Kirchhoff and Stratton-Chu formulae. In Section~\ref{sec:bdryintop} we introduce boundary-integral operators and establish some of the properties that are required to prove the well-posedness of boundary-value problems. Finally, Section~\ref{sec:BIE} studies properties of the Calder\'on projector and reveals a powerful symmetry with respect to dual transformations.\par
In our notation, we seek to strike a balance between readability on the one hand, and the addition of information that helps to interpret the compact differential-form notation on the other hand. If in doubt, we tend to favor the former over the latter, assuming that the generality and elegance of differential-form calculus best serve the readers' interest. For example, operators in Section~\ref{sec:prelim} are defined for forms of arbitrary degrees, and on (Riemannian) manifolds of arbitrary dimension. We therefore do not generally distinguish in our notation between, for example, the Hodge operators acting on forms of various degrees on a domain $\Omega$, and the Hodge operators acting upon the traces of said forms on the boundary $\Gamma$. The metric tensor which applies in the definition of each operator is clear from the context. 
A generalization that we did not adopt is to introduce graded Sobolev spaces on the entire exterior algebra of differential forms. We have opted for spaces of homogeneous degree and highlight the degree in the notation. 
All along the text, the relationship to results of classical vector analysis is established in framed paragraphs, to keep the paradigms separate in the main body of the paper. 
\section{Differential Forms - Preliminaries}\label{sec:prelim}
In this section, we intend to summarize important results of differential-form calculus. For an introduction we recommend, for example, \cite{jaenich_va_eng,Lang1995}. Throughout the paper, $n$ denotes the dimension of the problem domain; the degree of forms is frequently denoted by $p$ and $q$, which are always related by $q=n-p$.\par
Powers of minus one followed by operators, as in $(-1)^{pq}\mathrm{op}_1\,\mathrm{op}_2\,\boldsymbol{\phi}$, are to be read as follows: The degree $p$ refers to the differential form that the sequence of operators acts upon from the left. In this example, $p$ is the degree of the form $\boldsymbol{\phi}$. $n$ is always the dimension of the problem domain, even if operators and forms on the domain boundary are considered; and $q=n-p$ following the above rule.
\subsection{Basic Definitions}\label{sec:basics}
We introduce differential forms on a smooth, orientable Riemannian manifold $(M,\mathrm{g})$ of finite dimension $n$, where $\mathrm{g}$ denotes the metric tensor. We have $\mathbb{R}^3$, or a subset thereof, with Euclidean metric in mind. Throughout this section $V$ denotes a vector space over a field $\mathbb{F}$, where $\mathbb{F}$ may be either $\mathbb{R}$ or $\mathbb{C}$.\par
\marginpar{Simple $p$-vectors, $p$-vectors} A simple $p$-vector may be thought of as an ordered $p$-tuple of vectors that belong to a vector space $V$. The $p$-tuple is interpreted as a $p$-parallelepiped with oriented volume. An elementary permutation in the tuple changes the orientation. A change of orientation is indicated by a change of sign of the simple $p$-vector. More precisely, a simple $p$-vector is an equivalence class of ordered $p$-tuples of vectors that (i) span the same subspace of $V$; (ii) span $p$-parallelepipeds of identical oriented volume. $p$-vectors are linear combinations of simple $p$-vectors. They form a vector space $\mywedge^pV$ of dimension $\binom{n}{p}$, $0\leq p\leq n$. Up to dimension $n=3$ all $p$-vectors are simple $p$-vectors. We find $\mywedge^1V=V$, $\mywedge^p V=\emptyset$ for $p>n$ and for $p<0$, and we set $\mywedge^0V=\mathbb{F}$.\par 
Alternatively, $p$-vectors are defined in \cite{greub2} via an isomorphism that identifies $\mywedge^p V$ with the vector space of skew-symmetric tensors of rank $p$ over $V$.\par
\marginpar{Basis of $\mywedge^p V$, multi-indices}Let $(\vec{e}_i \,|\, 1\leq i\leq n)$ denote an ordered basis of $V$. We pick in each $\mywedge^p V$ an ordered basis
\[
(\vec{e}_J\,|\,J\in\mathcal{J}^n_p),
\]
where $J=j_1j_2\dots j_p$ is a multiindex,
\[
\mathcal{J}^n_p=\{J=j_1j_2\dots j_p\,|\,1\le j_1<j_2<\dots<j_p\le n\},
\]
and $\vec{e}_J$ is the equivalence class that contains the $p$-tupel $(\vec{e}_{j_1},\vec{e}_{j_2},\ldots,\vec{e}_{j_p})$.\par
\marginpar{Exterior product, exterior algebra}The exterior product, or wedge product, is a bilinear mapping 
\[
\wedge: \mywedge^k V\times \mywedge^\ell V \rightarrow \mywedge^{k+\ell}V: (\vec{v},\vec{w})\mapsto \vec{v}\wedge \vec{w},
\] defined by the following properties:
\begin{alignat*}{3}
\text{(i)}&\quad&& \text{$\wedge$ is associative, $(\vec{u}\wedge\vec{v})\wedge\vec{w} = \vec{u}\wedge(\vec{v}\wedge\vec{w})$,\quad $\vec{u}\in\mywedge^j V$;}\\
\text{(ii)}&\quad&&\text{$\wedge$ is graded anticommutative, $\vec{v}\wedge\vec{w}=(-1)^{k\ell}\vec{w}\wedge\vec{v}$ for} \\
&\quad&&\text{$\vec{v}\in\mywedge^k V$ and $\vec{w}\in\mywedge^\ell V$;}\\
\text{(iii)}&\quad&&1\wedge\vec{v}=\vec{v}\;\text{ for all }\vec{v}\in \mywedge^k V.
\end{alignat*}
To compute the exterior product we first relate the basis vectors of $V$ to those of $\mywedge^p V$. Let $K=k_1k_2\ldots k_p$ be an arbitrary $p$-index, and $\sigma(K)$ a permutation of $K$. Then we define
\begin{equation}\label{extprodaux}
\vec{e}_{k_1}\wedge\cdots\wedge\vec{e}_{k_p}=\begin{cases}+\vec{e}_{\sigma(K)}&\sigma(K)\in\mathcal{J}^n_p,\quad\sigma\text{ even},\\
-\vec{e}_{\sigma(K)}&\sigma(K)\in\mathcal{J}^n_p,\quad\sigma\text{ odd},\\
0 & \text{otherwise}.
\end{cases}
\end{equation} 
Next we define for basis vectors $\vec{e}_I\in\mywedge^kV$, $\vec{e}_J\in\mywedge^\ell V$
\begin{alignat*}{2}
\vec{e}_I\wedge\vec{e}_J&=(\vec{e}_{i_1}\wedge\dots\wedge\vec{e}_{i_k})\wedge(\vec{e}_{j_1}\wedge\dots\wedge\vec{e}_{j_\ell})\\
&=\vec{e}_{i_1}\wedge\dots\wedge\vec{e}_{i_k}\wedge\vec{e}_{j_1}\wedge\dots\wedge\vec{e}_{j_\ell},
\end{alignat*}
where $I\in\mathcal{J}^n_k$, $J\in\mathcal{J}^n_\ell$. The right hand side is defined by \eqref{extprodaux}, or zero for $k+\ell>n$, respectively. Finally, the exterior product extends by linearity to the entire spaces. We say that $\mywedge^p V$ is the $p$-th exterior power of $V$. The direct sum $\mywedge V= \mywedge ^0V\oplus \mywedge ^1V\oplus\dots \oplus\mywedge ^{n-1}V\oplus \mywedge ^n V$ is again a vector space. The pair $(\mywedge V,\wedge)$ has the structure of a graded algebra. It is called the exterior algebra over $V$.\par
\marginpar{Tangent and cotangent vectors}Recall that a tangent vector on a manifold can be interpreted as a directional-derivative operator. Coordinates $(x^1,...,x^n)$ on a patch $U\subset M$ induce a canonical coordinate basis $(\partial_{x^1},...,\partial_{x^n})$ of the tangent space $T_\ppoint M$. Here $TM$ is the tangent bundle over $M$, and $T_\ppoint M$ is a fibre in point $\ppoint\in U$. Cotangent vectors, on the other hand, are elements of the dual space $T^*_\ppoint M$, the cotangent space. In particular, the differential of a scalar function $\lambda$ on $M$ taken at a point $\ppoint$ is a cotangent vector. The action of a cotangent vector $(\mathrm{d}\lambda)_\ppoint\in T^*_\ppoint M$ on a vector $\vec{v}\in T_\ppoint M$ equals the directional derivative of $\lambda$ in direction of $\vec{v}$, $(\mathrm{d}\lambda)_\ppoint (\vec{v})=\vec{v} (\lambda)$. The canonical coordinate basis of the cotangent space in a point $\ppoint\in U$ reads $(\mathrm{d}x^1,...,\mathrm{d}x^n)$, and we see that $\mathrm{d}x^i(\partial_{x^j})=\partial_{x^j}(x^i)=\delta^i_j$, with $\delta^i_j$ the Kronecker delta.\par
\marginpar{Smooth $p$-vector fields and differential $p$-forms} A vector field $\vec{v}$ is a section of $TM$. The space of smooth vector fields (component functions are $C^\infty$) is denoted $\ves{1}(M)$. The space of smooth differential 1-forms is denoted $\dfs{1}(M)$; its elements are smooth sections of the cotangent bundle. Note that the definitions of $\ves{1}(M)$ and $\dfs{1}(M)$ require smoothness of the manifold $M$.\par
A $p$-vector field is a section of the $p$-th exterior power of the tangent bundle $\mywedge^p TM$, whose fibres are the spaces $\mywedge^p T_\ppoint M$. The space of smooth $p$-vector fields is denoted $\ves{p}(M)$, and the space of smooth differential $p$-forms $\dfs{p}(M)$.\par
\marginpar{Basis representation}Coordinate bases of $\ves{p}(M)$ and $\dfs{p}(M)$ are given in $U\subset M$ by $(\partial_{x^J} \,| \,J\in\mathcal{J}^n_p)$ and $(\mathrm{d}{x^J}\, |\, J\in\mathcal{J}^n_p)$, respectively. Hence the basis representation
\[
\boldsymbol{\omega}=\sum_{J\in\mathcal{J}^n_p}\omega_J\,\mathrm{d}x^J\in\dfs{p}(M),
\]
where $\omega_J\in C^\infty(M)$ are the component functions.\par
In the sequel, we will encounter function spaces of differential forms. A generic function space $\mathcal{H}$ of $p$-forms defined on $M$ will be denoted $\mathcal{H}\pcomp{p}(M)$. For instance, $\dfs{p}=C^\infty\pcomp{p}(M)$.\par
\marginpar{Duality product}We denote by 
\[
\dualprod{\,\cdot\,}{\,\cdot\,}_\ppoint: \mywedge ^p T^*_\ppoint M\times \mywedge ^p T_\ppoint M\rightarrow \mathbb{F}
\]
the algebraic duality product at a point $\ppoint\in M$.\par
\marginpar{Exterior and interior product/contraction}The exterior product above extends naturally to $p$-forms and $p$-vector fields. An alternative notation is given by
\[
\mathrm{j}:\dfs{\ell}(M)\times\dfs{k}(M)\rightarrow\dfs{\ell+k}(M):(\omega,\eta)\mapsto \mathrm{j}_\eta\omega=\eta\wedge\omega,
\]
and, analogously,
\[
\mathrm{j}:\ves{\ell}(M)\times\ves{k}(M)\rightarrow\ves{\ell+k}(M):(\vec{v},\vec{w})\mapsto \mathrm{j}_\vec{w} \vec{v}=\vec{w}\wedge \vec{v}.
\]
The interior products or contractions
\[
\mathrm{i}:\ves{k}(M)\times\dfs{\ell}(M)\rightarrow\dfs{\ell-k}(M):(\vec{v},\omega)\mapsto\left\{\begin{array}{lll} \mathrm{i}_\vec{v}\omega, \quad&\text{for}\quad &\ell\ge k,\\ 0,\quad&\text{for}\quad&\ell<k,\end{array}\right.
\]
and
\[
\mathrm{i}:\dfs{\ell}(M)\times\ves{k}(M)\rightarrow\ves{k-\ell}(M):(\omega,\vec{v})\mapsto \left\{\begin{array}{lll}\mathrm{i}_ \omega\vec{v}, \quad&\text{for}\quad &\ell\le k,\\0,\quad&\text{for}\quad&\ell > k\end{array}\right.
\]
are defined by dualities
\[
\dualprod{\mathrm{i}_\vec{v}\omega}{\vec{w}}_\ppoint=\dualprod{\omega}{\mathrm{j}_\vec{v} \vec{w}}_\ppoint,\quad\dualprod{\phi}{\mathrm{i}_\omega \vec{v}}_\ppoint =\dualprod{\mathrm{j}_\omega\phi}{\vec{v}}_\ppoint,
\]
for all $\vec{w}\in\ves{\ell-k}(M)$, $\phi\in\dfs{k-\ell}(M)$, and $\ppoint\in M$. The use of the notation $\mathrm{i}_\vec{v}$, $\mathrm{j}_\vec{v}$ as a shorthand for $\mathrm{i}(\vec{v},\cdot)$, $\mathrm{j}(\vec{v},\cdot)$ is a mere matter of convenience.\par
General properties of interior and exterior products are
\begin{alignat*}{3}
\text{(i)}&\quad&&\mathrm{i}_\vec{v}\mathrm{i}_\vec{w}=(-1)^{ij}\mathrm{i}_\vec{w}\mathrm{i}_\vec{v},\quad&&\mathrm{j}_\vec{v}\mathrm{j}_\vec{w}=(-1)^{ij}\mathrm{j}_\vec{w}\mathrm{j}_\vec{v},\\
\text{(ii)}&\quad&&\mathrm{i}_\vec{u}\mathrm{i}_\vec{u} = 0,\quad&&\mathrm{j}_\vec{u}\mathrm{j}_\vec{u}=0, \\
\text{(iii)}&\quad&&\mathrm{i}_\vec{u}(\omega\wedge\eta)=\mathrm{i}_\vec{u}\omega\wedge\eta + &&(-1)^{k}\omega\wedge\mathrm{i}_\vec{u}\eta,
\end{alignat*}
where $\vec{v}\in\ves{i}(M)$, $\vec{w}\in\ves{j}(M)$, $\omega\in\dfs{k}(M)$, $\eta\in\dfs{\ell}(M)$, and $ \vec{u}\in\ves{1}(M)$. Equivalent dual properties to (i) and (ii) hold for differential forms.\par
\marginpar{Pushforward, pullback}Consider the differential
\[ 
\mathrm{D}\varphi:T_\ppoint M\rightarrow T_{\varphi(\ppoint)}N:\vec{w}\mapsto\mathrm{D}\varphi(\vec{w})
\]
of a smooth map $\varphi:M\rightarrow N$ from manifold $M$ to manifold $N$. It naturally extends to $p$-vectors $\vec{v},\vec{w}\in\mywedge^pT_\ppoint M$ by the exterior compound
\[
\mathrm{D}\varphi(\vec{v}\wedge \vec{w})=\mathrm{D}\varphi(\vec{v})\wedge\mathrm{D}\varphi(\vec{w}).
\]
In case $\varphi$ is a diffeomorphism, $\mathrm{D}\varphi$ induces a pushforward operator
\[ 
\varphi_*:\ves{p}(M)\rightarrow\ves{p}(N):\vec{v}\mapsto\varphi_*\vec{v},
\]
by defining $\varphi_*\vec{v}_\ppoint= \mathrm{D}\varphi(\vec{v}_\ppoint)$ in all points $\ppoint\in M$. The pullback operator
\[ 
\varphi^*:\dfs{p}(N)\rightarrow\dfs{p}(M):\omega\mapsto\varphi^*\omega
\]
is defined by duality at every point $\ppoint\in M$,
\[
\dualprod{\varphi^*\omega}{\vec{v}}_\ppoint=\dualprod{\omega}{\mathrm{D}\varphi(\vec{v}_\ppoint)}_{\varphi\ppoint},
\]
where $\omega\in\dfs{p}(N)$, $\vec{v}\in\ves{p}(M)$. Note that $\varphi$ is not required to be a diffeomorphism in this case.\par
\marginpar{Exterior derivative}The exterior derivative is a linear map
\[
\mathrm{d}:\dfs{p}(M)\rightarrow\dfs{p+1}(M):\omega\mapsto\mathrm{d}\omega
\]
defined by the following properties:
\begin{alignat*}{3}
\text{(i)}&\quad&& \text{For $\gamma\in C^\infty(M)$, $\mathrm{d}\gamma$ is the differential of $\gamma$;}\\
\text{(ii)}&\quad&&\mathrm{d}\mathrm{d} = 0, \;\text{(first Poincar\'e lemma);}\\
\text{(iii)}&\quad&&\mathrm{d}(\omega\wedge\eta)=\mathrm{d}\omega\wedge\mathrm{\eta}+(-1)^{p}\omega\wedge\mathrm{d}\eta,\;\text{for}\;\omega\in\dfs{p}(M).
\end{alignat*}
The exterior derivative commutes with the pullback, i.e., $\mathrm{d}\varphi^*=\varphi^*\mathrm{d}$ holds. Its action on $\vec{v}=\vec{v}_1\wedge\ldots\wedge\vec{v}_{p+1}\in\mywedge^{p+1}T_\ppoint M$ is given by
\[
\dualprod{\mathrm{d}\boldsymbol{\omega}}{\vec{v}}_\ppoint=\sum_{i=1}^{p+1}(-1)^{i-1}\dualprod{\nabla_{\vec{v}_i}\boldsymbol{\omega}}{\vec{v}_1\wedge\ldots\wedge\hat{\vec{v}}_i\wedge\ldots\wedge\vec{v}_{p+1}}_\ppoint,
\]
where $\nabla_{\vec{v}_i}\boldsymbol{\omega}$ denotes the directional derivative of $\boldsymbol{\omega}$ at $\ppoint$ in direction $\vec{v}_i$, and $\hat{\vec{v}}_i$ is omitted.\par
\marginpar{Stokes' theorem}The integration of differential $p$-forms can be reduced to multivariate integration of scalar functions in coordinates. Stokes' theorem on an embedded $p$-dimensional compact submanifold $S$, $\varphi:S\rightarrow M$, reads for $\omega\in \dfs{p-1}(M)$
\[
\int_{\varphi(S)}\mathrm{d}\omega=\int_{\partial \varphi(S)}\omega,
\]
where $\partial$ is the boundary operator. It shows that the following operator equations are equivalent:
\[
\mathrm{d}\mathrm{d}=0\quad\text{and}\quad\partial\partial=0.
\]
Moreover, with the above properties of the exterior derivative, we find the integra\-tion-by-parts rule
\[
\int_{\varphi(S)}\mathrm{d}\omega\wedge\mathrm{\eta}+(-1)^{k}\int_{\varphi(S)}\omega\wedge\mathrm{d}\eta=\int_{\partial \varphi(S)}\omega\wedge\eta,
\]
for $\omega\in\dfs{k}(M)$, and $\eta\in\dfs{p-k-1}(M)$.\par
\marginpar{Orientation, volume form, and volume vector-field} An orientation of $M$ is given by a consistent orientation of all tangent spaces of $M$. An orientation of $T_\ppoint M$ is given by a fixed ordering of the vectors of a basis. Hence, a non-vanishing $n$-vector field $\vec{m}\in\ves{n}(M)$ defines an orientation on $M$, and $-\vec{m}$ represents the opposite orientation. Likewise, a volume form $\mu\in\dfs{n}(M)$ with $\dualprod{\mu}{\vec{m}}_\ppoint>0$ everywhere represents the same orientation of $M$. Henceforth we require without loss of generality that $\dualprod{\mu}{\vec{m}}_\ppoint=1$ holds everywhere.\par
\marginpar{Poincar\'e isomorphism} The Poincar\'e isomorphism \cite[p.\ 151]{greub2} is defined using the above volume form and volume vector-field (remember $q=n-p$)
\[
\mathrm{p}:\dfs{q}(M)\rightarrow\ves{p}(M):\eta\mapsto \mathrm{i}_\eta \vec{m}.
\]
The inverse is given by
\[
\mathrm{p}^{-1}:\ves{p}(M)\rightarrow\dfs{q}(M):\vec{v}\mapsto (-1)^{pq}\mathrm{i}_\vec{v} \mu.
\]
An important property of the Poincar\'e isomorphism is
\begin{eqnarray}\label{poincare}\dualprod{\omega}{\mathrm{p}\eta}\mu=(-1)^{pq}\omega\wedge\eta,\end{eqnarray}
where $\omega\in\dfs{p}(M)$, $\eta\in\dfs{q}(M)$, and the point-wise duality product defines a smooth scalar function on $M$.\par
\marginpar{Riesz isomorphism, inner product, norm}We prefer defining the Riemannian metric tensor as a Riesz isomorphism between smooth vector fields and 1-forms, $\mathrm{g}:\ves{1}(M)\rightarrow\dfs{1}(M)$. The definition is extended to $p$-vector fields and $p$-forms
\[
\mathrm{g}:\ves{p}(M)\rightarrow\dfs{p}(M)
\]
by the exterior compound, $\mathrm{g}(\vec{v}\wedge \vec{w})=\mathrm{g}\vec{v}\wedge\mathrm{g}\vec{w}$. The Riesz isomorphism defines a point-wise inner product 
\[
\mywedge^pT_\ppoint M\times \mywedge ^pT_\ppoint M\rightarrow\mathbb{F}: (\vec{v},\vec{w})\mapsto\metric{\vec{v}}{\vec{w}}_\ppoint=\dualprod{\mathrm{g}\vec{v}}{\overline{\vec{w}}}_\ppoint 
\]
in point $\ppoint\in M$, and, analogously,
\[
\mywedge ^pT^*_\ppoint M\times \mywedge ^pT^*_\ppoint M\rightarrow\mathbb{F}: (\omega,\eta)\mapsto\metric{\omega}{\eta}_\ppoint=\dualprod{ {\omega}}{\mathrm{g}^{-1}\overline{\eta}}_\ppoint,
\]
where the bar denotes complex conjugation of the component functions in the case $\mathbb{F}=\mathbb{C}$. The metric tensor is required to be smooth in the sense that on all coordinate patches $U$ the functions $g_{ij}:U\rightarrow\mathbb{F}$, defined by $\ppoint\mapsto\metric{\partial_{x^i}}{\partial_{x^j}}_\ppoint$, are smooth.\par
The inner products induce the corresponding point-wise norms on $\mywedge^pT_\ppoint M$ and $\mywedge ^pT^*_\ppoint M$
\[
|\vec{v}|_\ppoint=\sqrt{\metric{\vec{v}}{\vec{v}}}_\ppoint\quad\text{and}\quad|\omega|_\ppoint=\sqrt{\metric{\omega}{\omega}}_\ppoint.
\]
In what follows, duality product $\dualprod{\cdot}{\cdot}$, inner product $\metric{\cdot}{\cdot}$, and norm $|\cdot|$, acting on smooth $p$-forms and $p$-vector fields, denote point-wise evaluations that define smooth scalar functions on $M$.\par
\marginpar{Unit volume form, Hodge operator}We select a unit volume form, i.e., $|\mu|=1$ everywhere. This connects volume form and metric, which would otherwise be unrelated. The Hodge operator can then be defined by
\[
*:\dfs{p}(M)\rightarrow\dfs{q}(M):\phi\mapsto(-1)^{pq}\mathrm{p}^{-1}\mathrm{g}^{-1}\phi,
\]
and we see that $*\phi=\mathrm{i}(\mathrm{g}^{-1}\phi,\mu)$. The Hodge operator is an isomorphism. It fulfills $**=(-1)^{pq}$, and its inverse is given by $*^{-1}=(-1)^{pq}*=(-1)^{pq}\mathrm{g}\,\mathrm{p}$. The standard textbook definition
\begin{eqnarray}
\label{hodge}\dualprod{\omega}{\mathrm{g}^{-1}\phi}\mu=\omega\wedge*\phi,
\end{eqnarray}
where $\phi\in\dfs{p}(M)$, can be linked to the explicit definition by setting $\eta=*\phi$ in \eqref{poincare}.\par
\marginpar{$\mathcal{D}\pcomp{p}(M)$} Let $\mathcal{D}\pcomp{p}(M)\subset\dfs{p}(M)$ denote the space of test forms, that are smooth, compactly supported $p$-forms on $M\setminus\partial M$ \cite{jaenich_va_eng}. For $\omega,\eta\in\mathcal{D}\pcomp{p}(M)$ we define the inner product
\begin{equation*}
\ltwo{\boldsymbol{\omega}}{\boldsymbol{\eta}}{p}{M}=\int_M\metric{\omega}{\eta}\boldsymbol{\mu}=\int_{M}\boldsymbol{\omega}\wedge *\overline{\boldsymbol{\eta}},
\end{equation*}
and the norm $\ltwonorm{\omega}{p}{M}^2=\ltwo{\boldsymbol{\omega}}{\boldsymbol{\omega}}{p}{M}$. \marginpar{$L^2\pcomp{p}(M)$}Completion of $\mathcal{D}\pcomp{p}(M)$ with respect to this norm yields the Hilbert space $L^2\pcomp{p}(M)$ of square integrable $p$-forms.\par
\marginpar{Coderivative}At this point we introduce the coderivative\,\footnote{When comparing with \cite{weck2004} please note that the coderivative is defined there with an additional minus sign.}
\begin{equation}\label{defcoderivative}
\delta:\dfs{p}(M)\rightarrow\dfs{p-1}(M):\omega\mapsto\delta\omega=(-1)^p*^{-1}\mathrm{d}*\omega,
\end{equation}
and note the property
\[
\delta\delta=0.
\]
The coderivative is defined to be the adjoint operator of the exterior derivative in the $L^2$ inner product,
\[
\ltwo{\mathrm{d}\boldsymbol{\omega}}{\boldsymbol{\eta}}{p+1}{M}=\ltwo{\boldsymbol{\omega}}{\delta\boldsymbol{\eta}}{p}{M},
\]
where the intersection of the support of $\omega\in\dfs{p}(M)$ and $\eta\in\dfs{p+1}(M)$ must be compact in $M\setminus\partial M$.\par
\marginpar{Laplace-Beltrami operator}Eventually, we define the Laplace-Beltrami operator
\begin{equation}\label{deflaplacebeltrami}
\boldsymbol{\Delta}:\dfs{p}(M)\rightarrow\dfs{p}(M):\omega\mapsto\boldsymbol{\Delta}\omega=(\mathrm{d}\delta + \delta\mathrm{d})\omega.
\end{equation}
\begin{framed}
We can now link up with the notation of classical vector analysis. The scalar and vector fields of 3D Euclidean vector analysis are recovered by the translation isomorphisms,
\begin{subequations}\label{deftranslation}
\begin{alignat}{3}
\Upsilon^0:\quad&\mathrm{Id}_{\dfs{0}(M)},\label{translation0}\\
\Upsilon^1:\quad&\dfs{1}(M)\rightarrow\ves{1}(M)&\;:\quad&&{\omega}&\mapsto \mathrm{g}^{-1}{\omega},\label{translation1}\\
\Upsilon^2:\quad&\dfs{2}(M)\rightarrow\ves{1}(M)&\;:\quad&&{\omega}&\mapsto \mathrm{p}{\omega},\label{translation2}\\
\Upsilon^3:\quad&\dfs{3}(M)\rightarrow\dfs{0}(M)&\;:\quad&&{\omega}&\mapsto *{\omega}.\label{translation3}
\end{alignat}
\end{subequations}
Note that the usual definition $\Upsilon^2:\omega\mapsto\mathrm{g}^{-1}*\omega$ coincides with the above definition since for $n=3$ we find $*^{-1}=*=\mathrm{g}\mathrm{p}$ and therefore $\mathrm{g}^{-1}*=\mathrm{p}$.\par
The differential operators of vector analysis are derived from the exterior derivative by
\begin{alignat*}{1}
\textbf{grad} &= \Upsilon^1\;\mathrm{d}\;(\Upsilon^0)^{-1},\\
\textbf{curl} &= \Upsilon^2\;\mathrm{d}\;(\Upsilon^1)^{-1},\\
\mathrm{div} &= \Upsilon^3\;\mathrm{d}\;(\Upsilon^2)^{-1}.
\end{alignat*}
The identities $\textbf{curl}\,\textbf{grad}=\vec{0}$ and $\mathrm{div}\,\textbf{curl}=0$ follow directly from the first Poincar\'e lemma $\mathrm{d}\mathrm{d}=0$.\par
Integration of a 1-form $\alpha\in\dfs{1}(M)$ along an oriented curve $C$ that is embedded in $M$ by $\varphi :C\rightarrow M$ is related to the integration of the tangential component of a vector:
\[
\int_{\varphi(C)}\alpha=\int_{\varphi(C)} \vec{a}\cdot \vec{t} \,\mathrm{d}s,
\]
where $\vec{a}=\Upsilon^1\alpha$, $\mathrm{d}s$ is the metric arc-length differential, and $\vec{t}$ is the unit tangent vector field to $C$. We have used the dot-product notation for the inner product.\par
Similarly,
\[
\int_{\varphi(A)}\beta=\int_{\varphi(A)} \vec{b}\cdot \vec{n} \,\mathrm{d}a,
\]
where $A$ is an inner-oriented surface, embedded in $M$, $\vec{b}=\Upsilon^2\beta$, $\mathrm{d}a$ is the metric area differential, and $\vec{n}$ is the unit normal vector field to $A$. If $(\vec{v}_1,\vec{v}_2)$ defines an orientation of $T_\ppoint A$ in $\ppoint\in A$, $\vec{v}_1,\vec{v}_2\in T_\ppoint A$, then $\vec{n}_\ppoint$ is chosen such that $(\varphi_*\vec{v}_1,\varphi_*\vec{v}_2,\vec{n}_\ppoint)$ is consistent with the orientation of $T_\ppoint M$.\par
Eventually we find 
\[
\int_{\varphi(V)}\omega = \int_{\varphi(V)} w\,\mathrm{d} v,
\]
where $w=\Upsilon^3\omega$, $\mathrm{d}v$ is the metric volume differential, and $V$ is an inner-oriented domain that is embedded in $M$, $\varphi(V)\subset{M}$. We note that $\mathrm{d}v=\mu$ in \eqref{hodge}.\par
Stokes' theorem with $\omega\in \dfs{p-1}(M)$, and the embedded $p$-dimensional submanifold $S$, $\varphi:S\rightarrow M$, 
\[
\int_{\varphi(S)}\mathrm{d}\omega=\int_{\partial \varphi(S)}\omega
\]
encompasses the classical theorems for $n=3$, $p=1,2,3$,
\begin{alignat*}{3}
&\int_{\varphi(C)}\textbf{grad}\;f\cdot \vec{t}\; \mathrm{d}s&&=\left.f\right|_{\partial \varphi(C)},&&\quad\text{fundamental theorem of calculus},\\
&\int_{\varphi(A)}\textbf{curl}\;\vec{v}\cdot \vec{n}\;\mathrm{d}a&&=\int_{\partial \varphi(A)}\vec{v}\cdot \vec{t}\;\mathrm{d}s,&&\quad\text{Stokes' theorem},\\
&\int_{\varphi(V)}\mathrm{div}\;\vec{w}\;\mathrm{d}v&&=\int_{\partial \varphi(V)}\vec{w}\cdot \vec{n}\;\mathrm{d}a,&&\quad\text{Gauss' theorem}.
\end{alignat*}
The coderivative corresponds to classical operators:
\begin{alignat*}{1}
 \Upsilon^2\;\delta\;(\Upsilon^3)^{-1}&=-\textbf{grad} ,\\
\Upsilon^1\;\delta\;(\Upsilon^2)^{-1}&= \textbf{curl} ,\\
\Upsilon^0\;\delta\;(\Upsilon^1)^{-1}&=-\mathrm{div} .
\end{alignat*}
Eventually, we recover the scalar and vector Laplace operators from
\begin{alignat*}{3}
\Upsilon^0\;\boldsymbol{\Delta}\;(\Upsilon^0)^{-1}&=-\mathrm{div}\;\textbf{grad}&&=-\Delta_\mathrm{s} ,\\
\Upsilon^1\;\boldsymbol{\Delta}\;(\Upsilon^1)^{-1}&= -\textbf{grad}\;\mathrm{div}+\textbf{curl}\;\textbf{curl}&&=-\boldsymbol{\Delta}_\mathrm{v} ,\\
\Upsilon^2\;\boldsymbol{\Delta}\;(\Upsilon^2)^{-1}&=\textbf{curl}\;\textbf{curl} -\textbf{grad}\;\mathrm{div}&&=-\boldsymbol{\Delta}_\mathrm{v},\\
\Upsilon^3\;\boldsymbol{\Delta}\;(\Upsilon^3)^{-1}&=-\mathrm{div}\;\textbf{grad}&&=-\Delta_\mathrm{s} ,
\end{alignat*}
where $\Delta_\mathrm{s}$ and $\Delta_\mathrm{v}$ denote the classical scalar and vector Laplace operators, respectively.
\end{framed}
\subsection{Integral Transformations}\label{sec:inttrafo}
Motivated by the notion of electromagnetic fields and sources, we distinguish bet\-ween the source manifold $M$ and the observation manifold $N$ of an integral transformation. Forms in the domain of the transformation are defined on the source manifold; the image under the transformation is defined on the observation manifold. The transformation kernels are expressed in terms of double forms. In this section, source- and observation manifolds are both modeled by the Euclidean space $M=N=\mathbb{R}^n$, though. The source manifold will be chosen differently in Section~\ref{sec:SLDL}, as the boundary of a domain $\Omega\subset M$. At this point we will take full advantage of the differential-form based framework.\par
The parallel transport of a vector $\vec{v}\in T_\ppoint M$ from $\ppoint$ to $\ppoint^\prime$ is denoted by $\Gamma_\ppoint ^{\ppoint^\prime}\vec{v}=\vec{v}^\prime\in T_{\ppoint^\prime}M$. Note that in Euclidean space parallel transport is path independent. Therefore, a vector $\vec{v}$ anchored at $\ppoint$ can be uniquely extended throughout the entire manifold $M$. The resulting vector field is called covariantly constant, and the space of covariantly constant vector fields is denoted by $\ves{1}^{\text{const}}(M)$. The parallel transport $\Gamma_\ppoint^{\ppoint^\prime}$ carries over naturally to $p$-vectors and $p$-covectors, hence $\ves{p}^{\text{const}}(M)$ and $\dfs{\text{const},p}(M)$. These spaces are closed with respect to the Euclidean metric and the Hodge operator in the sense that $\dfs{\text{const},p}(M)=\mathrm{g}\bigl(\ves{p}^{\text{const}}(M)\bigr)$ and $\dfs{\text{const},q}(M)=*\bigl(\dfs{\text{const},p}(M)\bigr)$, respectively. Eventually, Cartesian coordinates refer to orthonormal covariantly constant basis fields.\par
Let $f:M\times M\rightarrow\mathbb{F}:(\ppoint,\ppoint^{\prime})\mapsto f(\ppoint,\ppoint')$ denote a scalar kernel function. With the help of this kernel, we define the integral transformation $\mathcal{T}$ for scalar functions $\sigma$ by
\begin{alignat*}{1}
(\mathcal{T}\sigma)(\ppoint^{\prime})&=\int_M f(\ppoint,\ppoint^\prime)\sigma(\ppoint)\,\mu=\int_M f(\ppoint,\ppoint ^\prime)*\sigma(\ppoint)\\
&=\ltwo{\sigma(\ppoint)}{\overline{f}(\ppoint,\ppoint^{\prime})}{0}{M},
\end{alignat*}
provided the integral exists. $\mu$ denotes the unit volume form on $M$. The generalization to $p$-forms $\omega\in\dfs{p}(M)$ is obtained by requiring that
\begin{align}
\dualprod{\mathcal{T}\boldsymbol{\omega}}{\vec{v}}_{\ppoint^\prime}&=\int_M f(\ppoint,\ppoint^{\prime})\,\dualprod{\boldsymbol{\omega}}{\vec{v}}_{\ppoint}\,\mu
=\int_M f(\ppoint,\ppoint^{\prime})\,\boldsymbol{\omega}(\ppoint)\wedge*\mathrm{g}\vec{v},\label{deftransform1}
\end{align} 
for $\vec{v}\in\ves{p}^{\text{const}}(M)$, which amounts to integration by Cartesian components. For the second equality we used \eqref{hodge}. Equation \eqref{deftransform1} provides a preliminary definition for the integral transformation of $p$-forms. For the definition of layer potentials and boundary integral operators we wish to have an equivalent explicit definition, which is provided next.\par
\marginpar{Double forms}At this point it is useful to introduce double forms, a concept that goes back to \textsc{de Rham} \cite[pp.\ 30-33]{rham}. Applications in the context of electromagnetic Green kernels have been reported in \cite{tucker2008,df32}.\par
A double form of bi-degree $(p,r)$ over the product manifold $M\times N$ may be regarded for each point in $M$ as a $p$-covector valued $r$-form on $N$, or for each point in $N$ as a $r$-covector valued $p$-form on $M$ \cite[Sec.\ 3]{tucker2008}. Given a pair of points $(\ppoint,\ppoint^{\prime})\in M\times N$ on the source and observation manifold, respectively, a double form evaluates to a tensor product of covectors. The covector in the observation point $\ppoint '$ is marked by a prime. In the presence of double forms, we need to distinguish between operators acting on the source and on the observation side of the tensor product. Unary and binary operators acting on the observation side are either evaluated in $\ppoint^\prime$ or marked with a prime.\par
We call double forms of bi-degree $(p,p)$ double $p$-forms. In electrical engineering literature, the vector proxies of double $p$-forms are often called dyadics \cite{Tai1994}.\par
\marginpar{Identity double $p$-form}The identity double $p$-form is defined for $(\ppoint,\ppoint^\prime)\in M\times M$ by
\begin{equation}\label{defip}
\boxed{\dualprod{\boldsymbol{I}_p(\ppoint,\ppoint^{\prime})}{\Gamma_\ppoint^{\ppoint^\prime}\vec{v}}_{\ppoint^\prime}=(\mathrm{g}\,\vec{v})_\ppoint,}
\end{equation}
where $\vec{v}\in\mywedge^pT_\ppoint(M)$. It naturally extends to $\vec{v}\in\ves{p}(M)$.\par
\marginpar{Integral transformations of $p$-forms}The integral transformation of $p$-forms $\omega\in\dfs{p}(M)$ is defined by
\begin{empheq}[box=\fbox]{alignat=1}
\label{deftransform2}
(\mathcal{T}\boldsymbol{\omega})(\ppoint^\prime)=\ltwo{\boldsymbol{\omega}(\ppoint)}{\overline{f}(\ppoint,\ppoint^{\prime})I_p(\ppoint,\ppoint^{\prime})}{p}{M}.
\end{empheq}
The conjugation of $f$ compensates for the conjugation which is inherent to the $L^2$ inner product. The double $p$-form $\overline{f}(\ppoint,\ppoint^{\prime})\boldsymbol{I}_p(\ppoint,\ppoint^{\prime})$ is the extension of the scalar kernel $\overline{f}(\ppoint,\ppoint^{\prime})$ to $p$-forms.\par
To show that \eqref{deftransform1} and \eqref{deftransform2} are equivalent we rewrite \eqref{deftransform1} with \eqref{defip},
\[
\dualprod{\mathcal{T}\boldsymbol{\omega}}{\vec{v}}_{\ppoint^\prime}=
\int_M f(\ppoint,\ppoint^{\prime})\boldsymbol{\omega}(\ppoint)\wedge*\dualprod{\boldsymbol{I}_p(\ppoint,\ppoint^{\prime})}{\vec{v}}_{\ppoint^{\prime}},
\]
where we took into account $\vec{v}\in\ves{p}^{\text{const}}(M)$. By linearity, we obtain \eqref{deftransform2}, and vice versa.\par
A basis representation of the identity double $p$-form on a coordinate patch $U\subset M$ is given by 
\begin{eqnarray}\label{coordid}
{\boldsymbol{I}_p(\ppoint,\ppoint^{\prime})=\det(g_{JK})\bigl(\Gamma^{\ppoint^{\prime}}_\ppoint\mathrm{d}x^J\bigr)\otimes \mathrm{d}x^K,}
\end{eqnarray}
where we adopt the summation convention for the multi-indices $J,K\in\mathcal{J}^n_p$. $(g_{JK})$ denotes the $p\times p$ submatrix of $(g_{jk})$. Equation \eqref{coordid} can be seen by picking a basis vector $\vec{v}=\partial_{x^I}\in\mywedge^pT_\ppoint M$, $p>0$, $\ppoint\in U\subset M$ in \eqref{defip}. With $\boldsymbol{I}_p(\ppoint,\ppoint^{\prime})$ according to \eqref{coordid} we obtain from \eqref{defip}
\[
\dualprod{\boldsymbol{I}_p(\ppoint,\ppoint^{\prime})}{\Gamma^{\ppoint^{\prime}}_\ppoint\partial_{x^I}}_{\ppoint^\prime}=\mathrm{det}(g_{JK})\delta^J_I\mathrm{d}x^K=\mathrm{det}(g_{IK})\mathrm{d}x^K=(\mathrm{g}\,\partial_{x^I})_\ppoint.
\] 
By convention, we set $I_0(\ppoint,\ppoint^\prime)=1$, and use $I_p$ and $f$ as shorthands for $I_p(\ppoint,\ppoint^\prime)$ and $f(\ppoint,\ppoint^\prime)$.\par
In Cartesian coordinates, $n=3$, we obtain the simple expressions
\begin{align*}
\boldsymbol{I}_0&=1,\\
\boldsymbol{I}_1&=\mathrm{d}x^{\prime}\otimes \mathrm{d}x+\mathrm{d}y^{\prime}\otimes \mathrm{d}y+\mathrm{d}z^{\prime}\otimes \mathrm{d}z,\\
\boldsymbol{I}_2&=\mathrm{d}x^{\prime}\!\wedge\!\mathrm{d}y^{\prime}\otimes \mathrm{d}x\!\wedge\!\mathrm{d}y+\mathrm{d}y^{\prime}\!\wedge\!\mathrm{d}z^{\prime}\otimes \mathrm{d}y\!\wedge\!\mathrm{d}z+\mathrm{d}z^{\prime}\!\wedge\!\mathrm{d}x^{\prime}\otimes \mathrm{d}z\!\wedge\!\mathrm{d}x,\\
\boldsymbol{I}_3&=\mathrm{d}x^{\prime}\!\wedge\!\mathrm{d}y^{\prime}\!\wedge\!\mathrm{d}z^{\prime}\otimes \mathrm{d}x\!\wedge\!\mathrm{d}y\!\wedge\!\mathrm{d}z.
\end{align*}
\begin{lemma}%
The identity double $p$-form fulfills the following properties
\begin{subequations}\label{ipprop}
\begin{alignat}{3}
\text{(i)}&&\quad*\boldsymbol{I}_p&\,=\;&&(*^{-1})^{\prime}\boldsymbol{I}_q,\label{ipprop1}\\
\text{(ii)}&&\quad\mathrm{d}\boldsymbol{I}_p&\,=\;&&\mathrm{d}^{\prime}\boldsymbol{I}_p=0.\label{ipprop2}\\
\intertext{For convolution-type scalar kernels $f(\ppoint,\ppoint^\prime)$, i.e., for scalar kernels $f:M\times M\rightarrow\mathbb{F}$ that depend only on the Euclidean distance $d(\ppoint,\ppoint^\prime)$ between $\ppoint$ and $\ppoint^\prime$ it also holds that}
\text{(iii)}&&\quad\mathrm{d}(f\boldsymbol{I}_p)&\,=\;&&\delta^{\prime}(f\boldsymbol{I}_{p+1}),\label{ipprop4}\\
\text{(iv)}&& \quad\delta(f\boldsymbol{I}_p)&\,=\;&&\mathrm{d}^{\prime}(f\boldsymbol{I}_{p-1}),\label{ipprop6}\\
\text{(v)}&& \quad\boldsymbol{\Delta}(f\boldsymbol{I}_p)&\,=\;&&\boldsymbol{\Delta}^{\prime}(f\boldsymbol{I}_p).\label{ipprop5}
\end{alignat}
\end{subequations}
\end{lemma}
We give an example for (i) in Cartesian coordinates, for $n=3$, $p=1$:
\begin{align*}
*\boldsymbol{I}_1&=*(\mathrm{d}x^{\prime}\otimes \mathrm{d}x+\mathrm{d}y^{\prime}\otimes \mathrm{d}y+\mathrm{d}z^{\prime}\otimes \mathrm{d}z)\\
&=\mathrm{d}x^{\prime}\otimes \mathrm{d}y\!\wedge\!\mathrm{d}z+\mathrm{d}y^{\prime}\otimes \mathrm{d}z\!\wedge\!\mathrm{d}x+\mathrm{d}z^{\prime}\otimes \mathrm{d}x\!\wedge\!\mathrm{d}y\\
&=(*^{-1})^{\prime}(\mathrm{d}y^{\prime}\!\wedge\!\mathrm{d}z^{\prime}\otimes \mathrm{d}y\!\wedge\!\mathrm{d}z+\mathrm{d}z^{\prime}\!\wedge\!\mathrm{d}x^{\prime}\otimes \mathrm{d}z\!\wedge\!\mathrm{d}x+\mathrm{d}x^{\prime}\!\wedge\!\mathrm{d}y^{\prime}\otimes \mathrm{d}x\!\wedge\!\mathrm{d}y)\\
&=(*^{-1})^{\prime}I_2.
\end{align*}
\begin{proof}
(i) is shown by observing that for $\vec{v}\in\mywedge^pT_\ppoint M$
\begin{align*}
\dualprod{*\boldsymbol{I}_p}{\Gamma_\ppoint^{\ppoint^\prime}\vec{v}}_{\ppoint^\prime}&=(*\mathrm{g}\vec{v})_\ppoint
=(-1)^{pq}(\mathrm{g}*\vec{v})_\ppoint\\
&=(-1)^{pq}\dualprod{\boldsymbol{I}_q}{*^{\prime}\Gamma_\ppoint^{\ppoint^\prime}\vec{v}}_{\ppoint^\prime}
=(-1)^{pq}\dualprod{*^{\prime}\boldsymbol{I}_q}{\Gamma_\ppoint^{\ppoint^\prime}\vec{v}}_{\ppoint^\prime},
\end{align*}
where we used \eqref{defip}. The Hodge operator on $p$-forms is dual to the Hodge operator on $p$-vectors. More precisely, for a $p$-vector $\vec{v}$ and a $q$-vector $\vec{w}$ we find $\dualprod{*\mathrm{g}\vec{v}}{\vec{w}}=\dualprod{\mathrm{g}\vec{v}}{*\vec{w}}=\dualprod{\mathrm{g}*\vec{v}}{**\vec{w}}=(-1)^{pq}\dualprod{\mathrm{g}*\vec{v}}{\vec{w}}$, since the Hodge map is an isometry. (ii) holds, because $\boldsymbol{I}_p$ is covariantly constant. This can be seen from the fact that its Cartesian component functions are constant. To show (iii-v), let $f$ denote a convolution-type kernel, so that for fixed $(\ppoint,\ppoint^{\prime})$ we find
\begin{equation}\label{ipprop3}
(\mathrm{d}^{\prime}f)_{\ppoint^{\prime}}=-\Gamma^{\ppoint^{\prime}}_{\ppoint}(\mathrm{d}f)_{\ppoint}.
\end{equation}
Let $\vec{v}\in\mywedge^p T_\ppoint M$ and consider
\begin{align*}
\bigdualprod{\mathrm{d}(f\boldsymbol{I}_p)}{\Gamma_\ppoint^{\ppoint^\prime}\vec{v}}_{\ppoint^\prime}&=(\mathrm{d}f\wedge \mathrm{g}\vec{v})_\ppoint
=\bigdualprod{\boldsymbol{I}_{p+1}}{\Gamma_\ppoint^{\ppoint^\prime}(\mathrm{g}^{-1}\mathrm{d}f\wedge\vec{v})}_{\ppoint^\prime}\\
&=-\bigdualprod{\boldsymbol{I}_{p+1}}{(\mathrm{g}^{-1}\mathrm{d})^{\prime}f\wedge\Gamma_\ppoint^{\ppoint^\prime}\vec{v}}_{\ppoint^\prime}
=-\bigdualprod{\mathrm{i}^\prime\bigl((\mathrm{g}^{-1}\mathrm{d})^{\prime}f,\boldsymbol{I}_{p+1}\bigr)}{\Gamma_\ppoint^{\ppoint^\prime}\vec{v}}_{\ppoint^\prime},
\end{align*}
where we used \eqref{defip}, \eqref{ipprop2}, and \eqref{ipprop3}. Therefore,
\begin{align*}
\mathrm{d}(f\boldsymbol{I}_p)&=-\mathrm{i}^\prime\bigl((\mathrm{g}^{-1}\mathrm{d})^{\prime}f,\boldsymbol{I}_{p+1}\bigr)
=(-1)^{p+1}(*^{-1})^{\prime}(\mathrm{d}^{\prime}f\wedge\,*^{\prime}\boldsymbol{I}_{p+1})\\
&=(-1)^{p+1}(*^{-1}\,\mathrm{d}\,*)^{\prime}(f\boldsymbol{I}_{p+1})=\delta^{\prime}(f\boldsymbol{I}_{p+1}),
\end{align*}
where we took into account \eqref{defcoderivative}, \eqref{ipprop2}, and that for a 1-form $\boldsymbol{\phi}=\mathrm{g}\vec{u}$ and $p$-form $\boldsymbol{\omega}$ we have $\mathrm{i}(\vec{u},\boldsymbol{\omega})=(-1)^{p-1}*^{-1}(\boldsymbol{\phi}\wedge\,*\boldsymbol{\omega})$. (iv) results from (iii), by exchanging unprimed with primed points, because $f$ and $\boldsymbol{I}_p$ are symmetric in $\ppoint$ and $\ppoint^{\prime}$. Eventually, (v) is a consequence of (iii) and (iv).
\qed\end{proof}
\subsection{Fundamental Solution of the Helmholtz Equation}
In this section, $M$ is defined in the same way as in Section~\ref{sec:inttrafo}. That is, source- and observation manifolds are both modeled by the Euclidean space $M=\mathbb{R}^n$.\par
Let $\delta(\ppoint,\ppoint^{\prime})$ denote the $n$-dimen\-sional Dirac delta distribution of source and observation point. Using this kernel, the transformation \eqref{deftransform2} reduces to the identity map from source- to observation-manifold, $\mathcal{T}\boldsymbol{\omega}=\mathrm{Id}\,\boldsymbol{\omega}=\boldsymbol{\omega}$, as can be seen from \eqref{deftransform1}.\par
\marginpar{Fundamental solution}Moreover, let $g_n(\ppoint,\ppoint^\prime)$ denote the standard fundamental solution of convolution type of the scalar Helmholtz equation,
\begin{equation}\label{green0}
(\Delta-k^2)g_n(\ppoint,\ppoint^{\prime})=\delta(\ppoint,\ppoint^{\prime}).
\end{equation}
With boundary-integral equations in mind, we restrict ourselves to dimension $n\ge 2$. We find, e.g., in \cite[Ch.\ 2.2, 2.3]{kythe1996}\footnote{Be aware that in \cite{kythe1996} the factor $r$ is missing in the denominator of the first case, and that $\nabla^2=\Delta_\mathrm{s}=-\Delta$, where $\Delta$ is the Laplace-Beltrami operator according to \eqref{deflaplacebeltrami}.}, \cite[Ch.\ 8, 9]{mclean}
\begin{equation}\label{defgreen}
g_n(\ppoint,\ppoint^{\prime})=\begin{cases}
\frac{i}{4}\left(\frac{k}{2\pi r}\right)^{(n/2)-1}H^{(1)}_{(n/2)-1}(kr) &\quad 0\le\arg k<\pi,k\ne 0,\\[0.5ex]
\frac{1}{2\pi}\ln \frac{r_0}{r},\quad r_0>0&\quad k=0,\,n=2,\\[0.5ex]
\frac{1}{(n-2)S_n(1)}r^{2-n}&\quad k=0,\,n\ge 3,
\end{cases}
\end{equation}
where $r=d(\ppoint,\ppoint^{\prime})$ is the Euclidean distance between $\ppoint$ and $\ppoint^\prime$, $H^{(1)}$ denotes the Hankel function of the first kind, and $S_n(1)$ is the Euclidean measure of the unit sphere in $n$ dimensions. For $n=3$
\[
g_3(\ppoint,\ppoint^{\prime})=\frac{\exp(ikr)}{4\pi r},\quad 0\le\arg k<\pi.
\]
\marginpar{Green operator, Green kernel}The Green operator $\mathcal{G}$ related to the Helmholtz operator $\boldsymbol{\Delta}-k^2$ for $p$-forms is the integral transformation \eqref{deftransform2} with the scalar kernel $\overline{f}=\overline{g}_n$. The extension of this kernel to $p$-forms, the Green kernel, is denoted by
\begin{equation}\label{deffundamentalp}
\boxed{\boldsymbol{G}_p=\overline{g}_n\,\boldsymbol{I}_p.}
\end{equation}
\begin{lemma}The following properties of the Green operator and the Green kernel hold:\\
(i) The Green operator provides a right inverse of the Helmholtz operator,
\begin{equation}\label{greenprop2}
{(\boldsymbol{\Delta}-k^2)\,\mathcal{G}=\mathrm{Id}.}
\end{equation}
(ii) Locally, the expression takes the form
\begin{equation}\label{greenprop1}
{(\boldsymbol{\Delta}-\overline{k}^2)\boldsymbol{G}_p=\delta(\ppoint,\ppoint^{\prime})\boldsymbol{I}_p,}
\end{equation}
i.e.\ $\boldsymbol{G}_p$ is a fundamental solution of the adjoint Helmholtz operator on $p$-forms.
\end{lemma}
\begin{proof}
We use the notation $\boldsymbol{\Delta}_p$ for the Laplace-Beltrami operator in order to emphasize that the operator is acting on $p$-forms. For flat Euclidean space and Cartesian coordinates we find for multi-indices $J\in\mathcal{J}^n_p$ and by using the summation convention,
\[
\boldsymbol{\Delta}_p\boldsymbol{\omega}=(\boldsymbol{\Delta}_0\omega_J)dx^J.
\]
That is, the Laplace-Beltrami operator acting on a $p$-form can be computed in Cartesian coordinates by applying scalar Laplace-Beltrami operators to its component functions. This is a particular case of the so-called Weitzenb{\"o}ck identities \cite{kotiuga2007}. From this result we can infer that in Euclidean space, for $\vec{v}\in\ves{p}^{\text{const}}(M)$,
\begin{equation}\label{weitzenboeck1}
\boldsymbol{\Delta}_0\dualprod{\boldsymbol{\omega}}{\vec{v}}=\dualprod{\boldsymbol{\Delta}_p\boldsymbol{\omega}}{\vec{v}}.
\end{equation}
From \eqref{deftransform1}, \eqref{green0}, and \eqref{weitzenboeck1} we derive
\begin{align*}
\dualprod{(\boldsymbol{\Delta}_p-k^2)\mathcal{G}\boldsymbol{\omega}}{\vec{v}}&=(\boldsymbol{\Delta}_0-k^2)\dualprod{\mathcal{G}\boldsymbol{\omega}}{\vec{v}}\\
&=\int_M(\boldsymbol{\Delta}_0-k^2)^{\prime}g_n(\ppoint,\ppoint^{\prime})\dualprod{\boldsymbol{\omega}}{\vec{v}}\,\mu\\
&=\int_M\delta(\ppoint,\ppoint^{\prime})\dualprod{\boldsymbol{\omega}}{\vec{v}}\,\mu=\dualprod{\mathrm{Id}\,\boldsymbol{\omega}}{\vec{v}},
\end{align*}
which proves (i). By using \eqref{deftransform2}, \eqref{ipprop5}, and \eqref{greenprop2} we prove (ii):
\begin{align*}
\int_M\boldsymbol{\omega}\wedge*\bigl((\boldsymbol{\Delta}-\overline{k}^2)\boldsymbol{G}_p\bigr)&=
\int_M\boldsymbol{\omega}\wedge*\bigl((\boldsymbol{\Delta}-\overline{k}^2)^{\prime}\boldsymbol{G}_p\bigr)\\
&=(\boldsymbol{\Delta}-\overline{k}^2)^{\prime}\int_M\boldsymbol{\omega}\wedge*\boldsymbol{G}_p\\
&=(\boldsymbol{\Delta}-\overline{k}^2)\,\overline{\mathcal{G}}\boldsymbol{\omega}=\mathrm{Id}\,\boldsymbol{\omega}
=\int_M\boldsymbol{\omega}\wedge*\delta(\ppoint,\ppoint^{\prime})\boldsymbol{I}_p.
\end{align*}
\qed\end{proof}
\subsection{Single-Layer and Double-Layer Potentials}\label{sec:SLDL}
In this section $\Omega$ denotes a bounded open subset of Euclidean space $M$ with smooth boundary $\partial\Omega$. We shall regard the boundary $\Gamma$ as a smooth manifold embedded into $\overline{\Omega}$ by $\iota:\Gamma\rightarrow\overline{\Omega}:\iota\Gamma=\partial\Omega$. We introduce integral transformations from the boundary $\Gamma$ to the domain $\Omega$, where $\Omega$ plays the role of the observation manifold, and $\Gamma$ that of the source manifold. The theory is extended to the case of Lipschitz boundaries in subsequent sections.
\marginpar{Tangential trace}The tangential trace for smooth forms is defined by pullback,
\begin{empheq}[box=\fbox]{equation}\label{tangtrac}
\mathrm{t}:\dfs{p}(\Omega)\rightarrow\dfs{p}(\Gamma): \boldsymbol{\omega}\mapsto\iota^*\boldsymbol{\omega},
\end{empheq}
\marginpar{Normal trace}and the normal trace is given by
\begin{empheq}[box=\fbox]{equation}\label{normtrac}
\mathrm{n}:\dfs{p}(\Omega)\rightarrow\dfs{p-1}(\Gamma): \boldsymbol{\omega}\mapsto\hhat{*}^{-1}\,\mathrm{t}\,*\boldsymbol{\omega}.
\end{empheq}
In contrast to the tangential trace, the normal trace is a metric-dependent concept.\par
\marginpar{Single- and double-layer potential}We define the single- and double-layer potentials to be the integral transformations from $\Gamma$ to $\Omega$
\begin{subequations}\label{deflayerpotentials}
\begin{empheq}[box=\fbox]{alignat=5}
\psl{p}&:\dfs{p}(\Gamma)\rightarrow\dfs{p}(\Omega)&&:\hhat{\boldsymbol{\omega}}\mapsto\ltwo{\hhat{\boldsymbol{\omega}}}{\mathrm{t}\,\boldsymbol{G}_p}{p}{\Gamma},\label{defsl}\\[0.2\baselineskip]
\pdl{p}&:\dfs{p}(\Gamma)\rightarrow\dfs{p}(\Omega)&&:\hhat{\boldsymbol{\omega}}\mapsto-*\mathrm{d}\,\psl{q-1}(\hhat{*}^{-1}\hhat{\boldsymbol{\omega}}).\label{defdl}
\end{empheq}
\end{subequations}
Recall that the $L^2$ inner product in \eqref{defsl} is $p$-form valued, because it contains a double $p$-form as second factor.
\begin{remark}
An alternative definition of the double-layer potential reads
\begin{equation*}
\pdl{p}:\dfs{p}(\Gamma)\rightarrow\dfs{p}(\Omega):\hhat{\boldsymbol{\omega}}\mapsto\delta^{\prime}\ltwo{\boldsymbol{\omega}}{\mathrm{n}\,\boldsymbol{G}_{p+1}}{p}{\Gamma}.
\end{equation*}
The equivalence can be seen from
\begin{align}\label{dldef2}
\delta^{\prime}\ltwo{\boldsymbol{\omega}}{\mathrm{n}\,\boldsymbol{G}_{p+1}}{p}{\Gamma}
&=(-1)^{p(q-1)}\delta^{\prime}\ltwo{\boldsymbol{\omega}}{\hhat{*}\hhat{*}\mathrm{n}\,\boldsymbol{G}_{p+1}}{p}{\Gamma}\nonumber\\
&=(-1)^{q-1}(\delta*)^{\prime}\ltwo{\hhat{*}^{-1}\boldsymbol{\omega}}{\mathrm{t}\,\boldsymbol{G}_{q-1}}{q-1}{\Gamma}\nonumber\\
&=-(*\mathrm{d})^{\prime}\ltwo{\hhat{*}^{-1}\boldsymbol{\omega}}{\mathrm{t}\,\boldsymbol{G}_{q-1}}{q-1}{\Gamma},
\end{align}
where we used $\hhat{*}\,\mathrm{n}=\mathrm{t}\,*$ and \eqref{ipprop1}, as well as that for $p$-forms there holds $\delta\,*=(-1)^{p+1}*\mathrm{d}$. 
\end{remark}
\begin{lemma}\label{lemma:lpotprop}
The single- and double-layer potentials exhibit the following properties in $\Omega$:
\begin{subequations}\label{lpotprop}
\begin{alignat}{3}
\text{(i)}\quad&&\delta\pslnot-\pslnot\hhat{\delta}&=0,\label{lpotprop1}\\
\text{(ii)}\quad&&(\Delta-k^2)\pslnot&=0,\label{lpotprop2}\\
\text{(iii)}\quad&&*\pdlnot&=(-1)^{p+1}\mathrm{d}\,\pslnot\hhat{*},\label{lpotprop3}\\
\text{(iv)}\quad&&\delta\pdlnot&=0,\label{lpotprop5}\\
\text{(v)}\quad&&\mathrm{d}\,\pdlnot-\pdlnot\hhat{\mathrm{d}}&=k^2*^{-1}\pslnot\hhat{*},\label{lpotprop4}\\
\text{(vi)}\quad&&*\,\mathrm{d}\,\pdlnot&=(-1)^p\mathrm{d}\,\pslnot\hhat{*}\,\hhat{\mathrm{d}}+k^2\pslnot\hhat{*},\label{lpotprop7}\\
\text{(vii)}\quad&&(\delta\mathrm{d}-k^2)\pdlnot&=0.\label{lpotprop6}
\end{alignat}
\end{subequations}
\end{lemma}
\begin{remark}
\begin{enumerate}
\item[]
\item Property (v) shows that the double-layer potential respects the de Rham sequence for $k=0$.
\item Properties (iii) and (vi) will be used in the proof of the jump relations in Section~\ref{sec:interfac}.
\end{enumerate}
\end{remark}
\begin{proof}
(i) is shown by
\begin{align*}
\delta\psl{p}\hhat{\boldsymbol{\omega}}
&=\delta^{\prime}\ltwo{\hhat{\boldsymbol{\omega}}}{\mathrm{t}\,\boldsymbol{G}_p}{p}{\Gamma}
=\ltwo{\hhat{\boldsymbol{\omega}}}{\mathrm{t}\,\mathrm{d}\,\boldsymbol{G}_{p-1}}{p}{\Gamma} \nonumber\\
&=\ltwo{\hhat{\boldsymbol{\omega}}}{\hhat{\mathrm{d}}\,\mathrm{t}\,\boldsymbol{G}_{p-1}}{p}{\Gamma}
=\ltwo{\hhat{\delta}\,\hhat{\boldsymbol{\omega}}}{\mathrm{t}\,\boldsymbol{G}_{p-1}}{p-1}{\Gamma}\\
&=\psl{{p-1}}\hhat{\delta}\omega,
\end{align*}
and (ii) is seen from
\begin{align*}
(\Delta-k^2)\psl{p}\hhat{\boldsymbol{\omega}}
&=(\Delta-k^2)^{\prime}\ltwo{\hhat{\boldsymbol{\omega}}}{\mathrm{t}\,\boldsymbol{G}_p}{p}{\Gamma}\nonumber\\
&=\ltwo{\hhat{\boldsymbol{\omega}}}{\mathrm{t}(\Delta-\overline{k}^2)\boldsymbol{G}_p}{p}{\Gamma}\\
&=\ltwo{\hhat{\boldsymbol{\omega}}}{\delta(\ppoint,\ppoint^{\prime})\,\mathrm{t}\,\boldsymbol{I}_p}{p}{\Gamma}\nonumber
=0,
\end{align*}
where we used \eqref{ipprop4} and \eqref{ipprop5}. (iii) and (iv) are direct consequences of the definitions \eqref{deflayerpotentials}. (vi) is obtained from combining (iii) and (v). (v) 
is shown by
\begin{align*}
\!\!\!\!\!\!\!\nonumber\mathrm{d}\pdl{p}\
&=-\mathrm{d}*\mathrm{d}\,\psl{q-1}\hhat{*}^{-1}
=(-1)^{q-1}*\delta\mathrm{d}\,\psl{q-1}\hhat{*}^{-1}\\
\nonumber&=(-1)^{q-1}*\bigl((\Delta-k^2)-\mathrm{d}\delta+k^2\bigr)\psl{q-1}\hhat{*}^{-1}\\
&=(-1)^q*\mathrm{d}\,\psl{q-2}\hhat{\delta}\,\hhat{*}^{-1}
+(-1)^{q-1}k^2*\psl{q-1}\hhat{*}^{-1}\\
\nonumber&=-*\mathrm{d}\,\psl{q-2}\hhat{*}^{-1}\hhat{\mathrm{d}}
+(-1)^{(p+1)(q-1)}k^2*\psl{q-1}\hhat{*}\\
&=\pdl{p+1}\hhat{\mathrm{d}}+k^2*^{-1}\psl{q-1}\hhat{*},
\end{align*}
and (vii) is seen from
\begin{align*}
\!\!\!\!\!\!\!\nonumber\delta\mathrm{d}\,\pdl{p}
&=\delta k^2*^{-1}\psl{q-1}\hhat{*}\\
&=(-1)^{p+1}k^2*^{-1}\mathrm{d}\,\psl{q-1}\hhat{*}\\
\nonumber&=-k^2*\mathrm{d}\,\psl{q-1}\hhat{*}^{-1}
=k^2\pdl{p}.
\end{align*}
\qed\end{proof}
\section{Sobolev Spaces of Differential Forms}\label{sec:sobolevspacesdf}
Throughout this section we adopt the following notation. $\Omega$ denotes a bounded open subset of a smooth orientable Riemannian manifold $M$ with Lipschitz boundary $\partial\Omega$. We require that $\Omega$ is homeomorphic to an open ball. The theory can be extended to the topologically non-trivial case, compare \cite{hiptmair2007}, but we do not delve into this. We shall regard the boundary $\Gamma$ as a Lipschitz manifold embedded into $\overline{\Omega}$ by $\iota:\Gamma\rightarrow\overline{\Omega}:\iota\Gamma=\partial\Omega$.\par
\subsection{Sobolev Spaces on the Domain}
\marginpar{$L^2\pcomp{p}(\Omega)$, \\weak derivative, \\weak coderivative}Let $\mathcal{D}\pcomp{p}(\overline{\Omega})$ be the restriction of $\mathcal{D}\pcomp{p}(M)$ to ${\Omega}$, and
\[
\mathcal{D}\pcomp{p}(\Omega)=\{\phi\in\mathcal{D}\pcomp{p}(M)|\supp{\phi}\subset\Omega\},
\]
respectively. The inner product of $p$-forms on $\Omega$ is defined by
\begin{equation*}
\ltwo{\boldsymbol{\omega}}{\boldsymbol{\eta}}{p}{\Omega}=\int_{\Omega}\boldsymbol{\omega}\wedge *\overline{\boldsymbol{\eta}},
\end{equation*}
and $\ltwonorm{\cdot}{p}{\Omega}$ denotes the corresponding norm. $L^2\pcomp{p}(\Omega)$ is the completion of $\mathcal{D}\pcomp{p}(\overline{\Omega})$ with respect to this norm. We call $\eta\in L^2\pcomp{p+1}(\Omega)$ the weak derivative of $\omega\in L^2\pcomp{p}(\Omega)$ iff
\[
\ltwo{\eta}{\phi}{p+1}{\Omega}=\ltwo{\omega}{\delta\phi}{p}{\Omega},
\]
for all $\phi\in\mathcal{D}\pcomp{p+1}(\Omega)$. $\pi\in L^2\pcomp{p-1}(\Omega)$ is called the weak coderivative of $\omega\in L^2\pcomp{p}(\Omega)$ iff
\[
\ltwo{\pi}{\psi}{p-1}{\Omega}=\ltwo{\omega}{\mathrm{d}\psi}{p}{\Omega},
\]
for all $\psi\in\mathcal{D}\pcomp{p-1}(\Omega)$.
\marginpar{$\hzero \pcomp{p}(\mathrm{d},\Omega)$}We can now define Sobolev spaces of differential $p$-forms on $\Omega$. Let 
\[\hzero \pcomp{p}(\Omega)=L^2\pcomp{p}(\Omega).
\] 
Define
\[
\boxed{\hzero \pcomp{p}(\mathrm{d},\Omega)=\{\boldsymbol{\omega}\in \hzero \pcomp{p}(\Omega)\,|\,\mathrm{d}\boldsymbol{\omega}\in \hzero \pcomp{p+1}(\Omega)\},}
\]
where $\mathrm{d}$ has to be understood in the weak sense.\par\noindent
We denote by 
\[
|\boldsymbol{\omega}|_{\hzero \pcomp{p}(\mathrm{d},\Omega)}=\ltwonorm{\boldsymbol{\omega}}{p}{\Omega}+\ltwonorm{\mathrm{d}\boldsymbol{\omega}}{p+1}{\Omega}
\]
the graph norm associated with $\hzero \pcomp{p}(\mathrm{d},\Omega)$. The following sequence is exact\,\textsuperscript{\ref{footnote:sequence}}
\begin{equation}\label{exactsequence1}
\begin{CD}
\cdots @>\mathrm{d}>>
\hzero \pcomp{p}  (\mathrm{d},\Omega) @>\mathrm{d}>>
\hzero \pcomp{p+1}(\mathrm{d},\Omega) @>\mathrm{d}>>
\cdots
\end{CD}\quad.
\end{equation}
The notations are to be read as follows: The Sobolev space $H^k\pcomp{p}(\Omega)$ contains $p$-forms over $\Omega$ that have component functions in $H^k(\Omega)$, $k\ge 0$. For $H^k\pcomp{p}(\mathrm{D},\Omega)$, the forms, as well as the derivation $\mathrm{D}$ of the forms in the weak sense, are in the respective $H^k\pcomp{p}(\Omega)$ space. The index $k=0$ is omitted.\par
\marginpar{$\hzero \pcomp{p}(\delta,\Omega)$}The Hodge star operator naturally extends to $L^2\pcomp{p}(\Omega)$ and we define the image spaces 
\[
\hzero \pcomp{p}(\delta,\Omega)=*\,\hzero \pcomp{q}(\mathrm{d},\Omega),\] 
with the norm 
\[
|\boldsymbol{\omega}|_{\hzero \pcomp{p}(\delta,\Omega)}=|*\boldsymbol{\omega}|_{\hzero \pcomp{q}(\mathrm{d},\Omega)}.
\]
 It is easy to see that these spaces are characterized by
\[\boxed{
\hzero \pcomp{p}(\delta,\Omega)=\{\boldsymbol{\omega}\in \hzero \pcomp{p}(\Omega)\,|\,\delta\boldsymbol{\omega}\in \hzero \pcomp{p-1}(\Omega)\},}
\]
where $\delta$ has to be understood in the weak sense. The sequence of these Sobolev spaces is, again, exact,
\begin{equation*}
\begin{CD}
\cdots @>\delta>>
\hzero \pcomp{p}  (\delta,\Omega) @>\delta>>
\hzero \pcomp{p-1}(\delta,\Omega) @>\delta>>
\cdots
\end{CD}\quad.
\end{equation*}
Note that the maps
\begin{empheq}[box=\fbox]{alignat*=2}
*:\;&\hzero \pcomp{p}(\mathrm{d},\Omega)&&\rightarrow \hzero \pcomp{q}(\delta,\Omega),\\
*:\;&\hzero \pcomp{p}(\delta,\Omega)&&\rightarrow \hzero \pcomp{q}(\mathrm{d},\Omega)
\end{empheq}
are isometric isomorphisms.\par
Eventually, for use in Section~\ref{sec:represent}, we introduce the space
\[
\hzero \pcomp{p}(\delta\mathrm{d},\Omega)=\{\boldsymbol{\omega}\in \hzero \pcomp{p}(\mathrm{d},\Omega)\,|\,\delta\mathrm{d}\boldsymbol{\omega}\in \hzero \pcomp{p}(\Omega)\},
\]
equipped with the graph norm, as well as the subspaces of closed and coclosed $p$-forms,
\begin{alignat*}{3}
\hzero \pcomp{p}(\mathrm{d}0,\Omega)&=\{\boldsymbol{\omega}\in \hzero \pcomp{p}(\mathrm{d},\Omega)&&|\mathrm{d}\boldsymbol{\omega}=0\}&&\subset \hzero \pcomp{p}(\mathrm{d},\Omega),\\
\hzero \pcomp{p}(\delta 0,\Omega)&=\{\boldsymbol{\omega}\in \hzero \pcomp{p}(\delta,\Omega)&&|\delta\boldsymbol{\omega}=0\}&&\subset \hzero \pcomp{p}(\delta,\Omega),\\
\hzero \pcomp{p}(\delta\mathrm{d} 0,\Omega)&=\{\boldsymbol{\omega}\in \hzero \pcomp{p}(\delta\mathrm{d},\Omega)&&|\delta\mathrm{d}\boldsymbol{\omega}=0\}&&\subset \hzero \pcomp{p}(\delta\mathrm{d},\Omega).
\end{alignat*}
\begin{framed}
For $n=3$, $p=0,...,n$, $\hzero \pcomp{p}(\mathrm{d},\Omega)$ encompasses a number of well-known spaces,
\begin{align*}
H^1(\Omega)&=\Upsilon^0 \hzero \pcomp{0}(\mathrm{d},\Omega),\\
\mathbf{H}(\textbf{curl},\Omega) &=\Upsilon^1 \hzero \pcomp{1}(\mathrm{d},\Omega),\\
\mathbf{H}(\mathrm{div},\Omega) &= \Upsilon^2 \hzero \pcomp{2}(\mathrm{d},\Omega),\\
L^2(\Omega) &= \Upsilon^3 \hzero \pcomp{3}(\mathrm{d},\Omega).
\end{align*}
Note that the translation isomorphisms \eqref{deftranslation} naturally extend to the Sobolev space setting. We can therefore redefine the differential operators of classical vector analysis 
\begin{alignat*}{4}
\textbf{grad} &= \Upsilon^1\;\mathrm{d}\;(\Upsilon^0)^{-1}&&:H^1(\Omega)\rightarrow\mathbf{H}(\textbf{curl},\Omega),\\
\textbf{curl} &= \Upsilon^2\;\mathrm{d}\;(\Upsilon^1)^{-1}&&:\mathbf{H}(\textbf{curl},\Omega)\rightarrow\mathbf{H}(\mathrm{div},\Omega),\\
\mathrm{div} &= \Upsilon^3\;\mathrm{d}\;(\Upsilon^2)^{-1}&&:\mathbf{H}(\mathrm{div},\Omega)\rightarrow L^2(\Omega).
\end{alignat*}
The sequence \eqref{exactsequence1} translates to
\[
\begin{CD}
H^1(\Omega) @>\textbf{grad}>>
\mathbf{H}(\textbf{curl},\Omega) @>\textbf{curl}>>
\mathbf{H}(\mathrm{div},\Omega) @>\mathrm{div}>>
L^2(\Omega)
\end{CD}.
\]
\end{framed}
In the next section we will proceed to introduce traces of the above Sobolev spaces. In the argument we will need $H^1\pcomp{p}(\Omega)$, the space of $p$-forms $\omega$ with component functions $\omega_I$ in $H^1(\Omega)$, 
\[
H^1\pcomp{p}(\Omega)=\{\boldsymbol{\omega}\in L^2\pcomp{p}(\Omega)\,|\,\omega_I\in H^1(\Omega),I\in\mathcal{J}^n_p\}.
\]
Different coordinates on $\Omega$ are related by smooth transition maps, so that this definition is independent from their choice, and the norms 
\[
|\boldsymbol{\omega}|^2_{H^1\pcomp{p}(\Omega)}=\sum_{I\in\mathcal{J}^n_p}|\omega_I|^2_{H^1(\Omega)}
\]
are equivalent. Note that $\hzero \pcomp{0}(\mathrm{d},\Omega)=H^1\pcomp{0}(\Omega)$. Moreover, we will encounter
\[
H^1\pcomp{p}(\mathrm{d},\Omega)=\{\boldsymbol{\omega}\in H^1\pcomp{p}(\Omega)\,|\,\mathrm{d}\boldsymbol{\omega}\in H^1\pcomp{p+1}(\Omega)\},
\]
and
\[
H^1\pcomp{p}(\delta,\Omega)=\{\boldsymbol{\omega}\in H^1\pcomp{p}(\Omega)\,|\,\delta\boldsymbol{\omega}\in H^1\pcomp{p-1}(\Omega)\},
\]
equipped with their respective graph norms.
\subsection{Sobolev Spaces on the Boundary}\label{sec:dfboundary}
\marginpar{$L^2\pcomp{p}(\Gamma)$}The Lipschitz manifold $\Gamma$ has well-defined tangent spaces almost everywhere. Their orientation is chosen to be consistent with the orientation of $\Omega$. The induced Riemannian metric allows for the definition of a Hodge operator. It is shown in \cite[Remark 1]{weck2004} that the $L^2\pcomp{p}(\Gamma)$ spaces are well defined, with the inner product
\begin{equation*}
\ltwo{\boldsymbol{\omega}}{\boldsymbol{\eta}}{p}{\Gamma}=\int_{\Gamma}\boldsymbol{\omega}\wedge \hhat{*}\overline{\boldsymbol{\eta}}
\end{equation*}
for $\omega,\eta\in L^2\pcomp{p}(\Gamma)$.
\marginpar{$H_{\parallel}^{\pm1/2}\pcomp{p}(\Gamma)$}Let $\hat{\omega}=\mathrm{t}\,\omega$ denote the tangential trace according to \eqref{tangtrac}. It is shown in \cite[Sec.\ 2]{weck2004} that the tangential trace can be extended to $H^1\pcomp{p}(\Omega)$. This extension may be defined even in the case of a Lipschitz boundary. We then define
\[
\boxed{H_{\parallel}^{1/2}\pcomp{p}(\Gamma)=\mathrm{t}\,H^1\pcomp{p}(\Omega)}
\]
equipped with the norm
\[
|\hat{\boldsymbol{\omega}}|_{H_{\parallel}^{1/2}\pcomp{p}(\Gamma)}=\inf_{\hat{\boldsymbol{\omega}}=\mathrm{t}\,\boldsymbol{\omega}}\{|\boldsymbol{\omega}|_{H^1\pcomp{p}(\Omega)}\}.
\]
The infimum is taken over all $\boldsymbol{\omega}\in H^1\pcomp{p}(\Omega)$ satisfying $\hat{\boldsymbol{\omega}}=\mathrm{t}\,\boldsymbol{\omega}$. Denote $H_{\parallel}^{-1/2}\pcomp{p}(\Gamma)$ the topological dual of $H_{\parallel}^{1/2}\pcomp{p}(\Gamma)$, with $L^2\pcomp{p}(\Gamma)$ as pivot space.\par
\marginpar{$H_{\perp}^{\pm 1/2}\pcomp{p}(\Gamma)$}It is shown in \cite[Lemma 5]{weck2004} that the Hodge operator $\hhat{*}$ may be restricted to $H_{\parallel}^{1/2}\pcomp{p}(\Gamma)$ and extended by continuity to $H_{\parallel}^{-1/2}\pcomp{p}(\Gamma)$. This allows us to define the image spaces
\[
\boxed{H_{\perp}^{\pm 1/2}\pcomp{p}(\Gamma)=\hhat{*}\,H_{\parallel}^{\pm 1/2}\pcomp{q-1}(\Gamma),}
\]
with the norm 
\[
|\hat{\boldsymbol{\omega}}|_{H_{\perp}^{1/2}\pcomp{p}(\Gamma)}=|\hhat{*}\,\hat{\boldsymbol{\omega}}|_{H_{\parallel}^{1/2}\pcomp{q-1}(\Gamma)}.
\]
Note that the maps
\begin{empheq}[box=\fbox]{alignat*=2}
\hhat{*}:\;&H_{\parallel}^{\pm 1/2}\pcomp{p}(\Gamma)&&\rightarrow H_{\perp}^{\pm 1/2}\pcomp{q-1}(\Gamma),\\
\hhat{*}:\;&H_{\perp}^{\pm 1/2}\pcomp{p}(\Gamma)&&\rightarrow H_{\parallel}^{\pm 1/2}\pcomp{q-1}(\Gamma)
\end{empheq}
are isometric isomorphisms, and the following inclusions hold \cite[Thm.\ 2]{weck2004},
\begin{alignat*}{3}
H_{\parallel}^{1/2}\pcomp{p}(\Gamma)&\subset L^2\pcomp{p}(\Gamma)&&\subset H_{\parallel}^{-1/2}\pcomp{p}(\Gamma),\\
H_{\perp}^{1/2}\pcomp{p}(\Gamma)&\subset L^2\pcomp{p}(\Gamma)&&\subset H_{\perp}^{-1/2}\pcomp{p}(\Gamma).
\end{alignat*}
\begin{remark}
\begin{enumerate}
\item[]
\item For smooth boundaries, the spaces $H_{\parallel}^{1/2}\pcomp{p}(\Gamma)$ and $H_{\perp}^{1/2}\pcomp{p}(\Gamma)$ are identical.
\item A fully intrinsic characterization of the spaces $H_{\parallel}^{1/2}\pcomp{p}(\Gamma)$ and $H_{\perp}^{1/2}\pcomp{p}(\Gamma)$ is available in the case of polyhedral domains. This has been shown in \cite{theor47a,theor47b} for $n=3$, $p=1$. In this case, the boundary consists of a finite number of smooth faces, and on each face the space $H_{\parallel}^{1/2}\pcomp{1}(\Gamma)$ coincides with 
\[
H^{1/2}\pcomp{1}(\Gamma)=\{\hat{\boldsymbol{\omega}}\in L^2\pcomp{1}(\Gamma)\,|\,\hat{\omega}_I\in H^{1/2}(\Gamma),\,I\in\mathcal{J}^{n-1}_1\}.
\]
The faces meet in edges, where forms in $H_{\parallel}^{1/2}\pcomp{1}(\Gamma)$ exhibit a weak tangential and forms in $H_{\perp}^{1/2}\pcomp{1}(\Gamma)$ a weak normal continuity (related to boundedness of the functionals $\mathcal{N}_{ij}^\parallel$ and $\mathcal{N}_{ij}^\bot$ defined in \cite[Prop.\ 4.3.]{theor47a}). It is expected that this approach could be generalized on polyhedral domains straightforwardly to other values of $n$ and $p$.
\end{enumerate}
\end{remark}
\marginpar{$H_{\perp}^{-1/2}\pcomp{p}(\hhat{\mathrm{d}},\Gamma)$}Since $H^1\pcomp{p}(\mathrm{d},\Omega)$ is a subspace of $H^1\pcomp{p}(\Omega)$, the space $Y=\hhat{*}\,\mathrm{t}\,H^1\pcomp{q-2}(\mathrm{d},\Omega)$ is a well defined subspace of $H_{\perp}^{1/2}\pcomp{p+1}(\Gamma)$, and we equip it with the norm
\[
|\hat{\boldsymbol{\omega}}|_Y=\inf_{\hat{\boldsymbol{\omega}}=\hhat{*}\,\mathrm{t}\,\boldsymbol{\omega}}\{|\boldsymbol{\omega}|_{H^1\pcomp{q-2}(\mathrm{d},\Omega)}\}.
\]
The infimum is taken over all $\boldsymbol{\omega}\in H^1\pcomp{q-2}(\mathrm{d},\Omega)$ satisfying $\hat{\boldsymbol{\omega}}=\hhat{*}\,\mathrm{t}\,\boldsymbol{\omega}$. We are now in the position to define the exterior derivative on the boundary in the weak sense
\[
\hhat{\mathrm{d}}:H_{\perp}^{-1/2}\pcomp{p}(\Gamma)\rightarrow Y^{\prime}:\hat{\boldsymbol{\omega}}\mapsto\hhat{\mathrm{d}}\,\hat{\boldsymbol{\omega}}
\]
by \cite[Lemma 3]{weck2004}
\begin{gather}
\dualprod{\hhat{\mathrm{d}}\,\hat{\boldsymbol{\omega}}}{\hat{\boldsymbol{\eta}}}_Y=(-1)^p\dualprod{\hat{\boldsymbol{\omega}}}{\hhat{*}^{-1}\,\mathrm{t}\,\mathrm{d}\,\boldsymbol{\eta}}_{H_{\perp}^{1/2}\pcomp{p}(\Gamma)},\nonumber\\
\boldsymbol{\eta}\in H^1\pcomp{q-2}(\mathrm{d},\Omega),\quad\hat{\boldsymbol{\eta}}=\hhat{*}\,\mathrm{t}\,\boldsymbol{\eta}\in Y.\label{extderivative1}
\end{gather}
The duality pairing between $Y^{\prime}$ and $Y$ is denoted by $\dualprod{\cdot}{\cdot}_Y$, and between $H_{\perp}^{\mp 1/2}\pcomp{p}(\Gamma)$ by $\dualprod{\cdot}{\cdot}_{H_{\perp}^{1/2}\pcomp{p}(\Gamma)}$, respectively. Note that $H_{\perp}^{-1/2}\pcomp{p+1}(\Gamma)$ is a subspace of $Y^{\prime}$. This motivates the definition
\[\boxed{
H_{\perp}^{-1/2}\pcomp{p}(\hhat{\mathrm{d}},\Gamma)=\{\hat{\boldsymbol{\omega}}\in H_{\perp}^{-1/2}\pcomp{p}(\Gamma)\,|\,\hhat{\mathrm{d}}\,\hat{\boldsymbol{\omega}}\in H_{\perp}^{-1/2}\pcomp{p+1}(\Gamma)\},
}\]
equipped with the graph norm.\par
We may now restrict the definition of the exterior derivative to
\begin{equation*}
\hhat{\mathrm{d}}:H_{\perp}^{-1/2}\pcomp{p}(\hhat{\mathrm{d}},\Gamma)\rightarrow H_{\perp}^{-1/2}\pcomp{p+1}(\Gamma):\hat{\boldsymbol{\omega}}\mapsto\hhat{\mathrm{d}}\,\hat{\boldsymbol{\omega}}.
\end{equation*}
From \eqref{extderivative1} we infer that $\hhat{\mathrm{d}}\hhat{\mathrm{d}}=0$, so that the image of $\hhat{\mathrm{d}}$ lies even in the smaller space $H_{\perp}^{-1/2}\pcomp{p+1}(\hhat{\mathrm{d}},\Gamma)$, and we can refine the definition to read
\begin{equation*}
\boxed{\hhat{\mathrm{d}}:H_{\perp}^{-1/2}\pcomp{p}(\hhat{\mathrm{d}},\Gamma)\rightarrow H_{\perp}^{-1/2}\pcomp{p+1}(\hhat{\mathrm{d}},\Gamma):\hat{\boldsymbol{\omega}}\mapsto\hhat{\mathrm{d}}\,\hat{\boldsymbol{\omega}}.}
\end{equation*}
This final definition lends itself to the definition of the de Rham complex on $\Gamma$.\par
It is shown in \cite[Thm.\ 3]{weck2004} that the extension of the tangential trace $\mathrm{t}$ from $H^1\pcomp{p}(\Omega)$ to
\[
\boxed{\mathrm{t}:\hzero \pcomp{p}(\mathrm{d},\Omega)\rightarrow H_{\perp}^{-1/2}\pcomp{p}(\hhat{\mathrm{d}},\Gamma)}
\]
is well defined, linear and continuous. Moreover \cite[Thm.\ 4]{weck2004}, it is surjective and hence admits a continuous right inverse $\mathrm{t}^{-1}$. The latter may be chosen such that its range lies in $H^1\pcomp{p}(\mathrm{d},\Omega)$.
The definitions and properties of $\hhat{\mathrm{d}}$ and $\hhat{\mathrm{t}}$ that we gave so far are summarized by the following exact sequence diagram:\,\footnote{Since $\Omega$ was required to be homeomorphic to an open ball, the sequences are exact, up to one-dimensional cohomology groups $\mathcal{H}^0(\Omega)$, $\mathcal{H}^0(\Gamma)$, and $\mathcal{H}^{n-1}(\Gamma)$.\label{footnote:sequence}}
\begin{equation}\label{exactsequence3}
\begin{CD}
\cdots @>\mathrm{d}>>
\hzero \pcomp{p}  (\mathrm{d},\Omega) @>\mathrm{d}>>
\hzero \pcomp{p+1}(\mathrm{d},\Omega) @>\mathrm{d}>>
\cdots\\
@. @V\mathrm{t}VV @V\mathrm{t}VV @.\\ 
\cdots @>\hhat{\mathrm{d}}>>
H_{\perp}^{-1/2}\pcomp{p}  (\hhat{\mathrm{d}},\Gamma) @>\hhat{\mathrm{d}}>>
H_{\perp}^{-1/2}\pcomp{p+1}(\hhat{\mathrm{d}},\Gamma) @>\hhat{\mathrm{d}}>>
\cdots
\end{CD}
\end{equation}
\begin{remark}
\begin{enumerate}
\item[]
\item The diagram \eqref{exactsequence3} commutes, which can be proven by density arguments \cite[Remark 2]{weck2004},
\begin{equation}\label{commutetd}
\mathrm{t}\mathrm{d}=\hhat{\mathrm{d}}\mathrm{t}.
\end{equation}
\item From the standard trace theorem for scalar functions in $H^1(\Omega)$ we conclude that $H_{\perp}^{-1/2}\pcomp{0}(\hhat{\mathrm{d}},\Gamma)=H^{1/2}(\Gamma)=H_{\parallel}^{1/2}\pcomp{0}(\Gamma)$.
\end{enumerate}\end{remark}
\marginpar{$H_{\parallel}^{-1/2}\pcomp{p}(\hhat{\delta},\Gamma)$}There is an $L^2$-adjoint version of the diagram \eqref{exactsequence3}. To obtain it, define the spaces 
\[
H_{\parallel}^{-1/2}\pcomp{p}(\hhat{\delta},\Gamma)=\hhat{*}\,H_{\perp}^{-1/2}\pcomp{q-1}(\hhat{\mathrm{d}},\Gamma),
\]
that are characterized by
\[
\boxed{H_{\parallel}^{-1/2}\pcomp{p}(\hhat{\delta},\Gamma)=\{\hat{\boldsymbol{\omega}}\in H_{\parallel}^{-1/2}\pcomp{p}(\Gamma)\,|\,\hhat{\delta}\,\hat{\boldsymbol{\omega}}\in H_{\parallel}^{-1/2}\pcomp{p-1}(\Gamma)\},}
\]
where $\hhat{\delta}=(-1)^p\hhat{*}^{-1}\,\hhat{\mathrm{d}}\,\hhat{*}$ is the coderivative on $p$-forms. The norm is given by 
\[
|\hat{\boldsymbol{\omega}}|_{H_{\parallel}^{-1/2}\pcomp{p}(\hhat{\delta},\Gamma)}=|\hhat{*}\,\hat{\boldsymbol{\omega}}|_{H_{\perp}^{-1/2}\pcomp{q-1}(\hhat{\mathrm{d}},\Gamma)}.
\]
Finally, we extend the definition of the normal trace for smooth forms in \eqref{normtrac} by
\begin{equation}\label{defnormaltrace}
\boxed{\mathrm{n}:\hzero \pcomp{p}(\delta,\Omega)\rightarrow H_{\parallel}^{-1/2}\pcomp{p-1}(\hhat{\delta},\Gamma):{\boldsymbol{\omega}}\mapsto\hat{\omega}=\hhat{*}^{-1}\,\mathrm{t}\,*\boldsymbol{\omega},}
\end{equation}
which inherits its properties from the tangential trace map \cite[Thm.\ 7]{weck2004}. It is well defined, linear, surjective, and admits a right inverse $\mathrm{n}^{-1}$ the range of which lies in $H^1\pcomp{p}(\delta,\Omega)$.\par
With these definitions we obtain the $L^2$-adjoint version of the diagram \eqref{exactsequence3}, that is again an exact sequence:
\begin{equation}\label{exactsequence4}
\begin{CD}
\cdots @>\delta>>
\hzero \pcomp{p}  (\delta,\Omega) @>\delta>>
\hzero \pcomp{p-1}(\delta,\Omega) @>\delta>>
\cdots\\
@. @V\mathrm{n}VV @V-\mathrm{n}VV @.\\ 
\cdots @>\hhat{\delta}>>
H_{\parallel}^{-1/2}\pcomp{p-1}(\hhat{\delta},\Gamma) @>\hhat{\delta}>>
H_{\parallel}^{-1/2}\pcomp{p-2}(\hhat{\delta},\Gamma) @>\hhat{\delta}>>
\cdots
\end{CD}
\end{equation}
\begin{remark}
\begin{enumerate}
\item[]
\item The diagram \eqref{exactsequence4} commutes, since from \eqref{defcoderivative}, \eqref{commutetd} and \eqref{defnormaltrace}
\begin{equation}\label{commutendelta}
\mathrm{n}\delta=-\hhat{\delta}\mathrm{n}.
\end{equation}
\item From the definitions it follows that $H_{\parallel}^{-1/2}\pcomp{0}(\hhat{\delta},\Gamma)=H_{\parallel}^{-1/2}\pcomp{0}(\Gamma)=H^{-1/2}(\Gamma)$.
\item Up to sign, the diagram \eqref{exactsequence4} can be thought of as the image of the diagram \eqref{exactsequence3} under Hodge star operators.
\item There exists a Hodge decomposition \cite[Thm.\ 11]{weck2004} of the spaces $H_{\perp}^{-1/2}\pcomp{p}(\hhat{\mathrm{d}},\Gamma)$ and $H_{\parallel}^{-1/2}\pcomp{p}(\hhat{\delta},\Gamma)$, respectively, which generalizes the results obtained in \cite{theor47c,theor47b}.
\end{enumerate}\end{remark}\par
\marginpar{Sesquilinear form}It has been shown in \cite[Thm.\ 8]{weck2004} that the $L^2$ inner product on the boundary $\Gamma$ can be extended to a sesquilinear form
\[
b(\cdot,\cdot)\,:\,H_{\parallel}^{-1/2}\pcomp{p}(\hhat{\delta},\Gamma)\times H_{\perp}^{-1/2}\pcomp{p}(\hhat{\mathrm{d}},\Gamma)\rightarrow\mathbb{F}.
\]
For $(\boldsymbol{\omega},\boldsymbol{\eta})\in \hzero \pcomp{p+1}(\delta,\Omega)\times \hzero \pcomp{p}(\mathrm{d},\Omega)$ the integration by parts formula
\begin{equation}\label{partialint}
b(\mathrm{n}\boldsymbol{\omega},\mathrm{t}\boldsymbol{\eta})=
\ltwo{\boldsymbol{\omega}}{\mathrm{d}\boldsymbol{\eta}}{p+1}{\Omega}-
\ltwo{\delta\boldsymbol{\omega}}{\boldsymbol{\eta}}{p}{\Omega}
\end{equation}
holds. In fact, in the light of the surjectivity of the traces and the existence of suitable right inverses, \eqref{partialint} provides the definition of $b(\cdot,\cdot)$. A comparison with \eqref{partialintl2} confirms that $b(\mathrm{n}\boldsymbol{\omega},\mathrm{t}\boldsymbol{\eta})$ reduces to $\ltwo{\mathrm{n}\boldsymbol{\omega}}{\mathrm{t}\boldsymbol{\eta}}{p}{\Gamma}$ in the smooth case.
\begin{framed}
Let us now relate the results \eqref{exactsequence3} and \eqref{exactsequence4} to the language of classical vector analysis in $n=3$ dimensions. We denote the outer normal to $\Omega$ by $\vec{n}$, defined almost everywhere on $\partial\Omega$. Note that classical calculus defines the traces on $\partial\Omega=\iota\Gamma$, whereas with differential-forms we defined them on $\Gamma$ by pullback. According to the definitions in \cite{buffa2003b,buffa2002}, we have to consider the following trace spaces and operators,
\begin{alignat*}{4}
\gamma:\quad&H^1(\Omega)&&\rightarrow H^{1/2}(\partial\Omega):\quad&&u &&\mapsto u\big|_{\partial\Omega},\\
\gamma_n:\quad&\mathbf{H}(\mathrm{div},\Omega)&&\rightarrow H^{-1/2}(\partial\Omega):\quad&&\vec{u}&&\mapsto\vec{u}\cdot\vec{n}\big|_{\partial\Omega},\\
\gamma_{\tau}:\quad&\mathbf{H}(\textbf{curl},\Omega)&&\rightarrow \mathbf{H}^{-1/2}(\mathrm{div}_{\Gamma},\partial\Omega):\quad&&\vec{u}&&\mapsto\vec{u}\times\vec{n}\big|_{\partial\Omega},\\
\pi_{\tau}:\quad&\mathbf{H}(\textbf{curl},\Omega)&&\rightarrow \mathbf{H}^{-1/2}(\mathrm{curl}_{\Gamma},\partial\Omega):\quad&&\vec{u}&&\mapsto\vec{u}-\gamma_n(\vec{u})\vec{n}\big|_{\partial\Omega}.
\end{alignat*}
All trace operators are defined for smooth test functions in the first place and then extended to the respective function spaces. They are all well defined, linear, continuous, surjective, and admit continuous right inverses. Be aware that in \cite[Def.\ 1]{buffa2003b} we find $\gamma_{\tau}:\vec{u}\mapsto\vec{n}\times\vec{u}\big|_{\partial\Omega}$, while in \cite[Def.\ 2.1]{buffa2002}, \cite[Def.\  $\gamma_\mathbf{t}^\times$]{hiptmair2003}, \cite[Def.\  $\gamma_\times$]{hiptmair2007} we have $\gamma_{\tau}:\vec{u}\mapsto\vec{u}\times\vec{n}\big|_{\partial\Omega}$. We follow the latter definition.\par\noindent
We define translation isomorphisms on the boundary,
\begin{subequations}\label{deftranslation2d}
\begin{alignat}{4}
\hat{\Upsilon}^0:\quad&H_{\perp}^{-1/2}\pcomp{0}(\hhat{\mathrm{d}},\Gamma)&&\rightarrow H^{1/2}(\partial\Omega):\quad&&\hhat{\boldsymbol{\omega}}&&\mapsto \phantom{-}\gamma\,\mathrm{t}^{-1}\,\hhat{\boldsymbol{\omega}},\\
\hat{\Upsilon}^1:\quad&H_{\perp}^{-1/2}\pcomp{1}(\hhat{\mathrm{d}},\Gamma)&&\rightarrow\mathbf{H}^{-1/2}(\mathrm{curl}_{\Gamma},\partial\Omega):\quad&&\hhat{\boldsymbol{\omega}}&&\mapsto \phantom{-}\pi_{\tau}\,\Upsilon^1\,\mathrm{t}^{-1}\,\hhat{\boldsymbol{\omega}},\\
\hat{\Upsilon}^{\tilde 1}:\quad&H_{\perp}^{-1/2}\pcomp{1}(\hhat{\mathrm{d}},\Gamma)&&\rightarrow\mathbf{H}^{-1/2}(\mathrm{div}_{\Gamma},\partial\Omega):\quad&&\hhat{\boldsymbol{\omega}}&&\mapsto -\gamma_{\tau}\,\Upsilon^1\,\mathrm{t}^{-1}\,\hhat{\boldsymbol{\omega}},\\
\hat{\Upsilon}^2:\quad&H_{\perp}^{-1/2}\pcomp{2}(\hhat{\mathrm{d}},\Gamma)&&\rightarrow H^{-1/2}(\partial\Omega):\quad&&\hhat{\boldsymbol{\omega}}&&\mapsto \phantom{-}\gamma_n\,\Upsilon^2\,\mathrm{t}^{-1}\,\hhat{\boldsymbol{\omega}}.
\end{alignat}
\end{subequations}
\begin{remark} The translation isomorphisms are extensions of the following maps, which are valid for smooth forms and smooth boundaries: $\hat{\Upsilon}^0:\hhat{\boldsymbol{\omega}}\mapsto\iota_*\hhat{\boldsymbol{\omega}}$, $\hat{\Upsilon}^1:\hhat{\boldsymbol{\omega}}\mapsto\iota_*\hhat{\mathrm{g}}^{-1}\hhat{\boldsymbol{\omega}}$, $\hat{\Upsilon}^{\tilde{1}}=\hat{\Upsilon}^1\hhat{*}:\hhat{\boldsymbol{\omega}}\mapsto\iota_*\hhat{\mathrm{p}}\,\hhat{\boldsymbol{\omega}}$, $\hat{\Upsilon}^2:\hhat{\boldsymbol{\omega}}\mapsto\iota_*\hhat{*}\,\hhat{\boldsymbol{\omega}}$.
These translation isomorphisms coincide with the definitions in \cite{df1} and \cite[p.\ 246]{df5}.\end{remark}
We are now in the position to define the surface differential operators of vector analysis on $\Gamma$, in terms of the exterior derivative $\hhat{\mathrm{d}}$,
\begin{alignat*}{4}
\textbf{grad}_{\Gamma} &= && \hat{\Upsilon}^1\;\hhat{\mathrm{d}}\;(\hat{\Upsilon}^0)^{-1}&&:H^{1/2}(\partial\Omega)&&\rightarrow\mathbf{H}^{-1/2}(\mathrm{curl}_{\Gamma},\partial\Omega),\\
\mathrm{curl}_{\Gamma} &= && \hat{\Upsilon}^2\;\hhat{\mathrm{d}}\;(\hat{\Upsilon}^1)^{-1}&&:\mathbf{H}^{-1/2}(\mathrm{curl}_{\Gamma},\partial\Omega)&&\rightarrow H^{-1/2}(\partial\Omega),\\
\textbf{curl}_{\Gamma} &= -&& \hat{\Upsilon}^{\tilde{1}}\;\hhat{\mathrm{d}}\;(\hat{\Upsilon}^0)^{-1}&&:H^{1/2}(\partial\Omega)&&\rightarrow\mathbf{H}^{-1/2}(\mathrm{div}_{\Gamma},\partial\Omega),\\
\mathrm{div}_{\Gamma} &= -&& \hat{\Upsilon}^2\;\hhat{\mathrm{d}}\;(\hat{\Upsilon}^{\tilde{1}})^{-1}&&:\mathbf{H}^{-1/2}(\mathrm{div}_{\Gamma},\partial\Omega)&&\rightarrow H^{-1/2}(\partial\Omega).
\end{alignat*}
From the general theory we know that the following sequences are exact, provided the boundary $\Gamma$ is homeomorphic to a sphere:
\[
\begin{CD}
H^{1/2}(\partial\Omega) @>\textbf{grad}_{\Gamma}>>
\mathbf{H}^{-1/2}(\mathrm{curl}_{\Gamma},\partial\Omega) @>\mathrm{curl}_{\Gamma}>>
H_*^{-1/2}(\partial\Omega)\\
@| @VV-\mathsf{R}V @|\\
H^{1/2}(\partial\Omega) @>\textbf{curl}_{\Gamma}>>
\mathbf{H}^{-1/2}(\mathrm{div}_{\Gamma},\partial\Omega) @>\mathrm{div}_{\Gamma}>>
H_*^{-1/2}(\partial\Omega)
\end{CD}
\]
We used $H_*^{-1/2}(\partial\Omega)=\{u\in H^{-1/2}(\partial\Omega)\,|\,\dualprod{u}{1}_{H^{1/2}(\partial\Omega)}=0\}$ to eliminate the cohomology group $\mathcal{H}^2(\partial\Omega)$. $\dualprod{\cdot}{\cdot}_{H^{1/2}(\partial\Omega)}$ denotes the duality pairing between the $H^{\mp 1/2}(\partial\Omega)$ spaces. The rotation operator
\begin{equation}\label{defrotation}
\mathsf{R}:\mathbf{H}^{-1/2}(\mathrm{curl}_{\Gamma},\partial\Omega)\rightarrow\mathbf{H}^{-1/2}(\mathrm{div}_{\Gamma},\partial\Omega):\gamma_{\tau}=-\mathsf{R}\,\pi_{\tau}
\end{equation}
renders the diagram commutative. Note that $\mathsf{R}\,\hat{\Upsilon}^1=\hat{\Upsilon}^{\tilde{1}}$. The rotation operator extends the vectorial operator $\vec{n}\times\cdot$, compare \cite[Def.\ $r$]{buffa2002}, \cite[Def.\ $R$]{hiptmair2003}. In \cite{hiptmair2007}, however, we find $\mathsf{R}=-\vec{n}\times\cdot$.
\end{framed}
\section{Representation Formula}\label{sec:represent}
The goal of this section is to collect, unify, and generalize existing knowledge about scalar and vectorial representation formulae in electromagnetics. We refer to the definition of $\Omega$ in Section~\ref{sec:sobolevspacesdf}, where we now consider the Euclidean space $M=\mathbb{R}^n$ with $n\ge 2$. We denote by $\Omega^{\mathrm{c}}$ the complement of $\overline{\Omega}$ in $M$, such that $\partial\Omega=-\partial\Omega^{\mathrm{c}}$, and $\Omega\cup\partial\Omega\cup\Omega^{\mathrm{c}}=M$. On $\Omega^{\mathrm{c}}$ we denote $L^2_{\mathrm{loc}}\pcomp{p}(\Omega^{\mathrm{c}})$ the space of $p$-forms that are square-integrable on each compact subdomain of $\Omega^{\mathrm{c}}$, and define the derived Sobolev spaces accordingly.\par
In Section~\ref{sec:helm} we introduce Maxwell-type problems and the space of solutions thereof. Section~\ref{sec:radiation} proceeds to define radiation and decay conditions. In Section~\ref{sec:secrep} we prove the status of a representation formula that represents elements of the solution space in terms of integral transformations of their boundary data. Section~\ref{sec:interfac} proves the so-called jump relations of the respective terms of the representation formula. The properties and proofs carry over to classical calculus by means of the translation isomorphisms \eqref{deftranslation} and \eqref{deftranslation2d}.
\pagebreak % <--- !!!
\subsection{Maxwell-Type Problems, Solution Spaces, and Trace Operators}\label{sec:helm}
\marginpar{Maxwell-type equation}We call Maxwell-type equation the second-order equation
\begin{equation}\label{defproblem}
\boxed{(\delta\mathrm{d}-k^2)\boldsymbol{\omega}=0\quad\text{in}\quad\Omega\cup\Omega^{\mathrm{c}},}
\end{equation}
compare Fig.\ \ref{fig:bemdf_setting}, for a $p$-form $\omega$, $0\le p < n$. We require that the constant $k\in\mathbb{C}$ fulfills either $k=0$ or $0\le\arg k<\pi,k\ne 0$.\par
We could have chosen the Helmholtz-type equation $(\delta\mathrm{d}+\mathrm{d}\delta-k^2)\boldsymbol{\omega}=0$ as starting point as well. Since we will work with the gauge condition
\begin{align}
\label{defgauge}\delta\omega&=0\quad\text{in}\quad\Omega\cup\Omega^\mathrm{c}
\end{align}
both equations are equivalent. For a more general account see \cite{jawerth1995}.\par
Domains $\Omega$, $\Omega^\mathrm{c}$, $\Omega^R$ and manifolds $\Gamma$, $\Gamma^R$ are defined in Fig.\ \ref{fig:bemdf_setting}. The terms "interior" and "exterior" refer to the domains $\Omega$ and $\Omega^\mathrm{c}$, respectively.\par
Every form that is eligible to be a solution of a boundary-value problem of Maxwell type must be an element of the space
\begin{gather}
 Y^p(\Omega\cup\Omega^\mathrm{c})=\{\boldsymbol{\omega}\in \hzero _{\mathrm{loc}}\pcomp{p}(\delta\mathrm{d},\Omega\cup\Omega^\mathrm{c})\cap \hzero _{\mathrm{loc}}\pcomp{p}(\delta0,\Omega\cup\Omega^\mathrm{c})\nonumber\\
 \,|\,(\delta\mathrm{d}-k^2)\boldsymbol{\omega}=0\}.\label{defsolutionspace}
\end{gather}
The definition implies \eqref{defproblem} as well as \eqref{defgauge}.
\marginpar{Maxwell solutions $X^p(\Omega\cup\Omega^\mathrm{c})$}We call Maxwell solutions those elements of $Y^p(\Omega\cup\Omega^\mathrm{c})$ that fulfill a radiation or decay condition in $\Omega^\mathrm{c}$, see Section~\ref{sec:radiation}. The space of Maxwell solutions is denoted by $X^p(\Omega\cup\Omega^\mathrm{c})\subset Y^p(\Omega\cup\Omega^\mathrm{c})$, and the restriction of this space to the domain of the interior problem reads $X^p(\Omega)=Y^p(\Omega)$.
\begin{remark}
\begin{enumerate}
\item[]
\item For $\boldsymbol{\omega}\in \hzero _{\mathrm{loc}}\pcomp{p}(\delta\mathrm{d},\Omega\cup\Omega^\mathrm{c})\cap \hzero _{\mathrm{loc}}\pcomp{p}(\delta,\Omega\cup\Omega^\mathrm{c})$, the operators in \eqref{defproblem} and \eqref{defgauge} are well defined. This motivates the above definition of a solution space for Maxwell-type problems.
\item The gauge condition \eqref{defgauge} for $p>0$ is a consequence of \eqref{defproblem} for $k\ne 0$, while for $k=0$ it is an additional requirement that we impose. Additional gauge conditions on the periods of the trace $\mathrm{t}*\omega$ would apply on topologically non-trivial domains.
\end{enumerate}\end{remark}
\begin{figure}
\sidecaption
\includegraphics[width=0.6\textwidth]{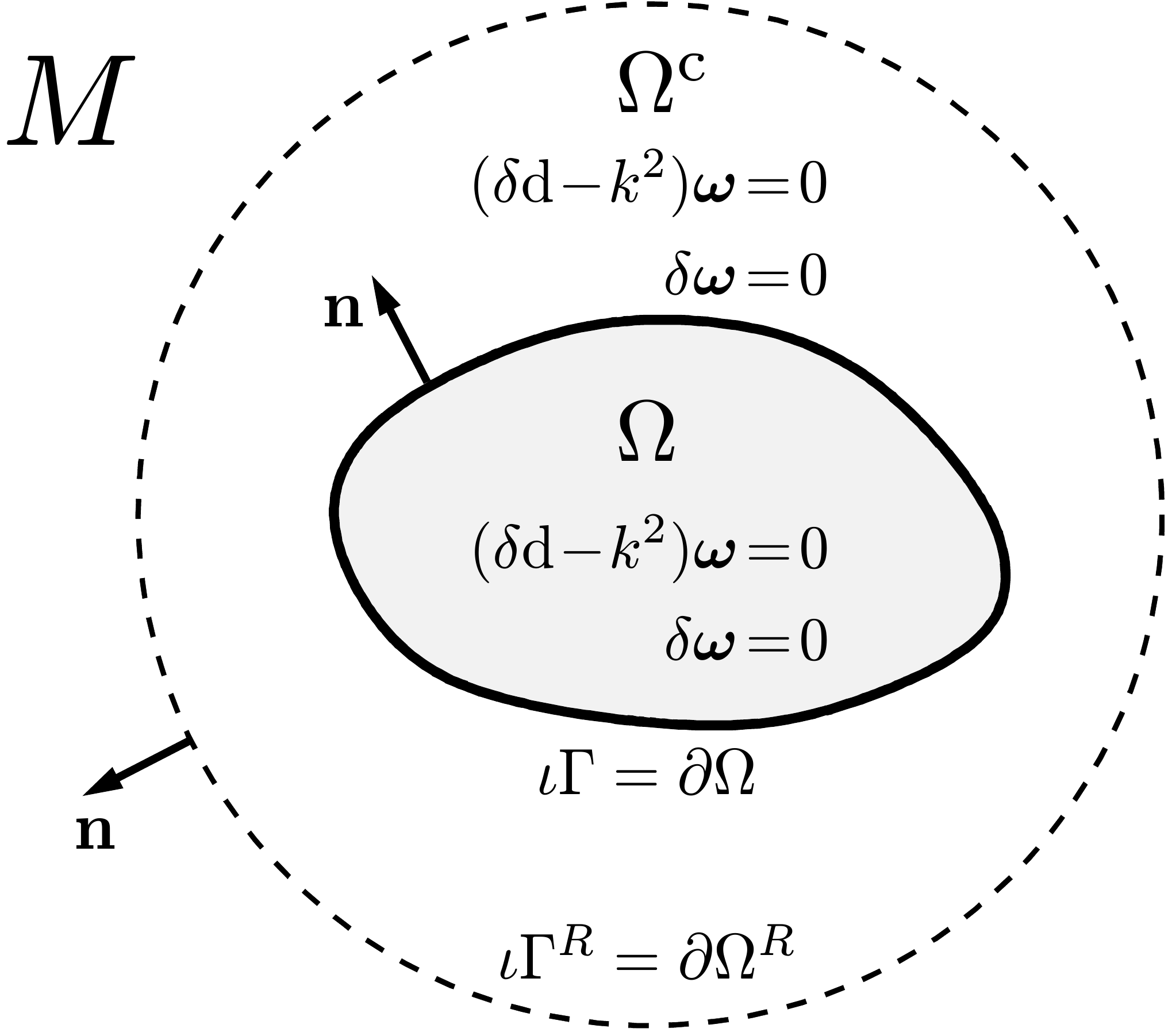}
\caption{Setting of the problem. $\Omega$ denotes a bounded subset of the Euclidean space $M$, homeomorphic to an open ball. Its Lipschitz boundary $\partial\Omega$ is represented by an embedded Lipschitz manifold $\Gamma$. $\Omega^\mathrm{c}$ is the complement of $\overline{\Omega}$ in $M$. $\Omega^R\subset M$, homeomorphic to an open ball, is chosen such that $\Omega^R\Supset\Omega$. Its smooth boundary $\partial\Omega^R$ ("far boundary") is represented by an embedded smooth manifold $\Gamma^R$. The normal-vector field $\vec{n}$ defines outer normals to both boundaries.}\label{fig:bemdf_setting}
\end{figure}
\marginpar{Dirichlet and Neumann traces}The interior Dirichlet, Neumann, and normal traces of Maxwell solutions are defined by
\begin{subequations}\label{defboundary}
\begin{empheq}[box=\fbox]{alignat=5}
&\gamma_{\mathrm{D}}:\;&&\hzero \pcomp{p}(\mathrm{d},\Omega)&&\rightarrow H_{\perp}^{-1/2}\pcomp{p}(\hhat{\mathrm{d}},\Gamma)&&:\;\omega\mapsto\mathrm{t}\,\omega,&&\quad\text{Dirichlet trace},\\
&\gamma_{\mathrm{N}}:\;&&\hzero \pcomp{p}(\delta\mathrm{d},\Omega)&&\rightarrow H_{\parallel}^{-1/2}\pcomp{p}(\hhat{\delta},\Gamma)&&:\;\omega\mapsto\mathrm{n}\,\mathrm{d}\,\omega,&&\quad\text{Neumann trace},\\
&\mathrm{n}:\;&&\hzero \pcomp{p}(\delta,\Omega)&&\rightarrow H_{\parallel}^{-1/2}\pcomp{p-1}(\hhat{\delta},\Gamma)&&:\;\omega\mapsto\mathrm{n}\,\omega,&&\quad\text{normal trace},
\end{empheq}
\end{subequations}
where the definition of the normal trace is given in \eqref{defnormaltrace}.
\begin{remark}
\begin{enumerate}
\item[]
\item The interior trace operators $\gamma_{\mathrm{D}}$ and $\gamma_{\mathrm{N}}$ are well defined, linear, continuous, and surjective, as a consequence of the properties of $\mathrm{t}$ and $\mathrm{n}$ laid out in Section~\ref{sec:dfboundary}.
\item The exterior Dirichlet, Neumann, and normal traces $\gamma_{\mathrm{D}}^\mathrm{c}$, $\gamma_{\mathrm{N}}^\mathrm{c}$, and $\mathrm{n}^\mathrm{c}$, as well as the trace operators on $\Gamma^R$, $\gamma_{\mathrm{D}}^R$, $\gamma_{\mathrm{N}}^R$, and $\mathrm{n}^R$, are defined in the same way, and the same mapping properties apply. 
\end{enumerate}
\end{remark}
\subsection{Asymptotic Conditions}\label{sec:radiation}
Any Maxwell solution of \eqref{defproblem} is required to fulfill an asymptotic condition at infinity. For $k\ne 0$ the condition is called a radiation condition, for $k=0$ it is called a decay condition. Let $\vec{n}$ denote an outward directed unit normal-vector field to a large sphere $\Gamma^R$ with radius $R$, compare Fig.\ \ref{fig:bemdf_setting}. $\ppoint$ denotes an arbitrary point on $\Gamma^R$.\noindent
\begin{enumerate}
\item\marginpar{Radiation condition} {\em Case $k\ne 0$}
The radiation condition reads 
\begin{equation}\label{condradiation3}
\boxed{\left.\begin{alignedat}{2}
\ltwonorm{(\gamma_\mathrm{N}^R-ik\gamma_\mathrm{D}^R)\boldsymbol{\omega}}{p}{\Gamma^R}&=o(1),\\
\ltwonorm{\mathrm{n}^R\boldsymbol{\omega}}{p-1}{\Gamma^R}&=o(1).
\end{alignedat}\quad\right\}}
\end{equation}
We present the radiation condition in an $L^2$ sense because it directly relates to the energy arguments that are used to establish uniqueness of solutions \cite[Thm.\ 2.4]{kress1998}. Lemma~\ref{lemma:radiation} provides strong versions of the radiation condition.
\begin{lemma}\label{lemma:radiation}
The following radiation conditions are equivalent to \eqref{condradiation3}, under the assumption of uniformity for $R\to\infty$:
\begin{equation}\label{condradiation2}
\left.\begin{alignedat}{2}
|(\gamma_\mathrm{N}^R-ik\gamma_\mathrm{D}^R)\boldsymbol{\omega}|_\ppoint&=o(R^{-(n-1)/2}),\\
|\mathrm{n}^R\boldsymbol{\omega}|_\ppoint&=o(R^{-(n-1)/2}),
\end{alignedat}\quad\right\}
\end{equation}
\begin{equation}\label{condradiation1}
|\mathrm{i}_{\vec{n}}\mathrm{d}\boldsymbol{\omega}-ik\boldsymbol{\omega}|_{\iota \ppoint}=o(R^{-(n-1)/2}).
\end{equation}
For the latter condition see also \cite[eq.\ (1)]{jawerth1995}.
\end{lemma}
\begin{proof}
The Euclidean measure $S_n(R)$ of a sphere of radius $R$ in $n$ dimensions is $\mathcal{O}(R^{n-1})$ for $R\to\infty$. This proves the validity of \eqref{condradiation2}.\par 
For a smooth $p$-form $\boldsymbol{\omega}$ it holds on a smooth submanifold that
\begin{align}
\mathrm{t}\,\mathrm{i}_{\vec{n}}\boldsymbol{\omega}&=(-1)^{pq}\mathrm{t}\,\mathrm{i}_{\vec{n}}**\,\boldsymbol{\omega}=(-1)^{pq}(-1)^q\hhat{*}\,\mathrm{t}*\boldsymbol{\omega}\nonumber\\
&=(-1)^{pq}(-1)^q(-1)^{(p-1)q}\hhat{*}^{-1}\,\mathrm{t}*\boldsymbol{\omega}\nonumber\\
&=\mathrm{n}\boldsymbol{\omega}.\label{condaux1}
\end{align}
For the second equality see \cite[eq.\ (32)]{kurz2009a}. With the decomposition techniques from \cite{kurz2009a} it can easily be shown that the Pythagoras' theorem
$
|\boldsymbol{\omega}|_{\iota \ppoint}^2=|\mathrm{t}\boldsymbol{\omega}|_\ppoint^2+|\mathrm{n}\boldsymbol{\omega}|_\ppoint^2
$
holds pointwise, where the moduli on the right hand side refer to the induced metric, and therefore $|\mathrm{t}\boldsymbol{\omega}|_\ppoint,|\mathrm{n}\boldsymbol{\omega}|_\ppoint\le|\boldsymbol{\omega}|_{\iota \ppoint}$. Condition \eqref{condradiation1} can be inferred from \eqref{condradiation2} with the help of \eqref{condaux1}, and vice versa.\par
\qed\end{proof}
\begin{framed}
Condition \eqref{condradiation1} encompasses the Sommerfeld radiation condition for $p=0$, 
\begin{equation*}\frac{\partial u}{\partial r}-iku=o(R^{-(n-1)/2}),\end{equation*}
and the Silver-M\"uller radiation condition for $p=1, n=3$,
\begin{equation*}
|\textbf{curl}\;\vec{u}\times\vec{n}-ik\vec{u}|=o(R^{-1}).
\end{equation*}
\end{framed}
The following corollary will be useful in Section~\ref{sec:symmsec}.
\begin{corollary}%
Let $\omega\in X^p(\Omega^\mathrm{c}\cap\Omega^R)$. Then for $k\ne 0$ the Dirichlet and Neumann traces fulfill
\begin{equation}\label{boundtraces}
\ltwonorm{\gamma_\mathrm{D}^R\boldsymbol{\omega}}{p}{\Gamma^R}=\mathcal{O}(1),\quad
\ltwonorm{\gamma_\mathrm{N}^R\boldsymbol{\omega}}{p}{\Gamma^R}=\mathcal{O}(1),
\end{equation}
i.e.\ the Dirichlet and Neumann traces on the far boundary are bounded as $R\to\infty$.
\end{corollary}
\begin{proof}
Substituting $\boldsymbol{\omega}$ by $\mathrm{d}\boldsymbol{\omega}$, $\boldsymbol{\omega}\in \hzero \pcomp{p}(\delta\mathrm{d},\Omega)$ in \eqref{partialint} on $\Omega^\mathrm{c}\cap\Omega^R$ and using the fact that $\omega$ is a Maxwell solution yields 
\[
-b(\gamma_\mathrm{N}^\mathrm{c}\boldsymbol{\omega},\gamma_\mathrm{D}^\mathrm{c}\boldsymbol{\eta})
+\ltwo{\gamma_\mathrm{N}^R\boldsymbol{\omega}}{\gamma_\mathrm{D}^R\boldsymbol{\eta}}{p}{\Gamma^R}
=\ltwo{\mathrm{d}\boldsymbol{\omega}}{\mathrm{d}\boldsymbol{\eta}}{p+1}{\Omega^\mathrm{c}\cap\Omega^R}-
k^2\ltwo{\boldsymbol{\omega}}{\boldsymbol{\eta}}{p}{\Omega^\mathrm{c}\cap\Omega^R}.
\]
Note that on the far boundary $\Gamma^R$ we encounter sufficient regularity so to work with the $L^2$ inner product rather than the sesquilinear form $b(\cdot,\cdot)$. We set $\boldsymbol{\eta}=\boldsymbol{\omega}$, multiply by $\overline{k}$ and take the imaginary part,
\begin{eqnarray*}
\lefteqn{-\Im\text{m}\bigl(\overline{k}b(\gamma_\mathrm{N}^\mathrm{c}\boldsymbol{\omega},\gamma_\mathrm{D}^\mathrm{c}\boldsymbol{\omega})\bigr)
+\Im\text{m}\bigl(\overline{k}\ltwo{\gamma_\mathrm{N}^R\boldsymbol{\omega}}{\gamma_\mathrm{D}^R\boldsymbol{\omega}}{p}{\Gamma^R}\bigr)}\\
&&=-(\Im\text{m}\,k)\bigl(\ltwo{\mathrm{d}\boldsymbol{\omega}}{\mathrm{d}\boldsymbol{\omega}}{p+1}{\Omega^\mathrm{c}\cap\Omega^R}+
|k|^2\ltwo{\boldsymbol{\omega}}{\boldsymbol{\omega}}{p}{\Omega^\mathrm{c}\cap\Omega^R}\bigr).
\end{eqnarray*}
From the radiation condition we obtain
\begin{align*}
o(1)&=\ltwonorm{\gamma_\mathrm{N}^R\boldsymbol{\omega}-ik\gamma_\mathrm{D}^R\boldsymbol{\omega}}{p}{\Gamma^R}^2\\
&=\ltwonorm{\gamma_\mathrm{N}^R\boldsymbol{\omega}}{p}{\Gamma^R}^2+|k|^2\ltwonorm{\gamma_\mathrm{D}^R\boldsymbol{\omega}}{p}{\Gamma^R}^2
-2\Im\text{m}\,\bigl(\overline{k}\ltwo{\gamma_\mathrm{N}^R\boldsymbol{\omega}}{\gamma_\mathrm{D}^R\boldsymbol{\omega}}{p}{\Gamma^R}\bigr).
\end{align*}
Combining both equations yields
\begin{gather}
\Im\text{m}\bigl(\overline{k}b(\gamma_\mathrm{N}^\mathrm{c}\boldsymbol{\omega},\gamma_\mathrm{D}^\mathrm{c}\boldsymbol{\omega})\bigr)=
(\Im\text{m}\,k)\bigl(\ltwonorm{\mathrm{d}\boldsymbol{\omega}}{p+1}{\Omega^\mathrm{c}\cap\Omega^R}^2+
|k|^2\ltwonorm{\boldsymbol{\omega}}{p}{\Omega^\mathrm{c}\cap\Omega^R}^2\bigr)\nonumber\\
+{\textstyle\frac{1}{2}}\bigl(\ltwonorm{\gamma_\mathrm{N}^R\boldsymbol{\omega}}{p}{\Gamma^R}^2+|k|^2\ltwonorm{\gamma_\mathrm{D}^R\boldsymbol{\omega}}{p}{\Gamma^R}^2\bigr)+o(1).\label{boundaux3}
\end{gather}
The left hand side is bounded, due to the continuity of the trace operators and of $b(\cdot,\cdot)$. Since we required $\Im\text{m}\,k\ge 0$, we know that the right hand side consists of four non-negative terms, each of which must be bounded for $R\to\infty$. This proves the claim \eqref{boundtraces}.
\qed\end{proof}
\begin{remark} This result generalizes the case $p=0$ that is discussed in \cite[eq.\ (9.17)-(9.19)]{mclean}.\end{remark}
\item
\marginpar{Decay condition}{\em Case $k=0$.} The decay condition reads
\begin{equation}\label{conddecay1}
\boxed{|\boldsymbol{\omega}|_{\iota \ppoint}\begin{cases}=\mathcal{O}(R^{2-n})&\quad n\ge 3,\\
\le \bigl|b\ln {\textstyle\frac{R}{r_0}}\bigr|+\mathcal{O}(R^{-1})&\quad n=2,
\end{cases}}
\end{equation}
uniformly for $R\to\infty$, where $b\in\mathbb{C}$ and $r_0>0$ are constants. For a motivation, see \cite[Thm.\ 8.9]{mclean}. Condition \eqref{conddecay1} implies the same asymptotic behaviour for the Dirichlet and normal traces $\gamma_\mathrm{D}^R\boldsymbol{\omega}$ and $\mathrm{n}^R\boldsymbol{\omega}$, respectively.\par
It can be shown that $\omega\in X(\Omega^\mathrm{c}\cap\Omega^R)$ subject to \eqref{conddecay1} also fulfills
\begin{equation}\label{conddecay2}
|\mathrm{d}\boldsymbol{\omega}|_{\iota \ppoint}\begin{cases}=\mathcal{O}(R^{1-n})&\quad n\ge 3,\\
\le |b|R^{-1}+\mathcal{O}(R^{-2})&\quad n=2,
\end{cases}
\end{equation}
uniformly for $R\to\infty$. The proof of this assertion is based on the theory of generalized spherical harmonics. For $p=0$ it can be found in \cite{mclean}, as part of the proof of Thm.\ 8.9. We skip the general case, while pointing out that the required background can be found in \cite{weck1994}. Condition \eqref{conddecay2} implies the same asymptotic behaviour for the Neumann trace $\gamma_\mathrm{N}^R\boldsymbol{\omega}$.\par
For $p=0$, the above conditions \eqref{conddecay1} and \eqref{conddecay2} are sharp. For $p>0$ it can even be shown that by condition $\delta\boldsymbol{\omega}=0$ we gain one order of decay, i.e.\ $|\boldsymbol{\omega}|_{\iota \ppoint}=\mathcal{O}(R^{1-n})$ and $|\mathrm{d}\boldsymbol{\omega}|_{\iota \ppoint}=\mathcal{O}(R^{-n})$, $n\ge 2$.\footnote{For $n=3,p=1$ see \cite{ammari2000}, proof of Proposition 3.1.} This implies $b=0$ for $p=1$.
\end{enumerate}
\begin{remark}
The asymptotic conditions guarantee that a representation formula for Maxwell solutions exists without extra terms due to the far boundary. It must be noted that the conditions do not generally ensure that the forms decay rapidly enough for $X^p(\Omega^\mathrm{c})\subset L^2\pcomp{p}(\Omega^\mathrm{c})$ to hold. For the scope of this paper, we circumvent the problem by using the spaces $L^2_\mathrm{loc}\pcomp{p}(\Omega^\mathrm{c})$. However, if one is aiming at variational formulations in the domain $\Omega^\mathrm{c}$, the issue generally has to be remedied by replacing $L^2\pcomp{p}(\Omega^\mathrm{c})$ with Beppo-Levi-type weighted spaces $L^2_s\pcomp{p}(\Omega^\mathrm{c})$ \cite{pauly2006}, and building the theory in terms of weighted Sobolev spaces.\,\footnote{For the vectorial case $n=3$, $p=1$, $s=-1$ see \cite{hiptmair2003,hiptmair2007}.}
\end{remark}
\subsection{Representation Formula for Maxwell Solutions}\label{sec:secrep}
We define the jump of some trace $\gamma_\mathrm{X}$ of a form $\omega$ as
\[
\jump{\gamma_\mathrm{X}}\boldsymbol{\omega}=\gamma_\mathrm{X}^\mathrm{c}\boldsymbol{\omega}-\gamma_\mathrm{X}\boldsymbol{\omega},
\]
and refer to the definitions of the single- and double-layer potentials in Section~\ref{sec:SLDL}.\par
\marginpar{Boundary data}The boundary data $(\beta,\gamma,\varphi)$ of forms $\boldsymbol{\omega}\in X^p(\Omega\cup\Omega^\mathrm{c})$ are defined as
\begin{subequations}\label{defboundary1}
\begin{empheq}[box=\fbox]{alignat=2}
\beta&=\jump{\gamma_\mathrm{D}}\,\omega&&\in H_\perp^{-1/2}\pcomp{p}(\hhat{\mathrm{d}},\Gamma),\\
\gamma&=\jump{\gamma_\mathrm{N}}\,\omega&&\in H_{\parallel}^{-1/2}\pcomp{p}(\hhat{\delta},\Gamma),\\
\varphi&=\jump{\mathrm{n}}\,\omega&&\in H_{\parallel}^{-1/2}\pcomp{p-1}(\hhat{\delta}0,\Gamma).
\end{empheq}
\end{subequations}
\begin{theorem}\label{thm:repthm}
\marginpar{Representation formula}
A Maxwell solution $\omega\in X^p(\Omega\cup\Omega^\mathrm{c})$ can be represented in terms of integral transformations of its boundary data by
\begin{equation}\label{represent}
\boxed{\boldsymbol{\omega}=-\psl{p}\,\boldsymbol{\gamma}
+\pdl{p}\,\boldsymbol{\beta}
-\mathrm{d}\,\psl{p-1}\,\boldsymbol{\varphi}.}
\end{equation}
\end{theorem}
\begin{remark} We note that the Neumann data obey the additional restriction
\begin{equation}\label{neumannsolution}
\hhat{\delta}\boldsymbol{\gamma}=-k^2\boldsymbol{\varphi},\quad\hhat{\delta}\boldsymbol{\varphi}=0.
\end{equation}\end{remark}
\begin{framed}
Consider the surface current density $\vec{k}\in\mathbf{H}^{-1/2}(\mathrm{div}_\Gamma,\partial\Omega)$ and the surface charge density $\sigma\in H^{-1/2}(\partial\Omega)$. For $p=1$, $(\vec{k},\sigma)$ can be defined as vector and scalar proxies of the differential forms $(\boldsymbol{\gamma},\boldsymbol{\varphi})$, respectively. The first equation in \eqref{neumannsolution} then translates into
\[
\mathrm{div}_\Gamma\vec{k}+i\omega\sigma=0,
\]
the continuity equation, where $\omega$ is the angular frequency. Therefore, \eqref{neumannsolution} is called generalized continuity equation.
\end{framed}\noindent
For the proof we proceed in three steps: First, we derive in Lemma~\ref{lemma:lemmarep} representation formulae for smooth forms on $\Omega$, $\Omega^\mathrm{c}\cap\Omega^R$, and $(\Omega\cup\Omega^\mathrm{c})\cap\Omega^R$, respectively. Second, we show in Lemma~\ref{lemma:lemmafarbound} that the terms related to the far boundary vanish for $R\rightarrow \infty$, provided that the smooth forms fulfill the asymptotic condition. Third, Lemma~\ref{lemma:repok} demonstrates that the layer potentials are well-defined for boundary data in the relevant Sobolev spaces.
\begin{lemma}\label{lemma:lemmarep}
Forms $\omega\in Y^p(\Omega)\cap\dfs{p}(\Omega)$ can be represented by
\begin{equation}\label{representcompactinterior}
\boldsymbol{\omega}=\psl{p}\,(\gamma_\mathrm{N}\boldsymbol{\omega})
-\pdl{p}\,(\gamma_\mathrm{D}\boldsymbol{\omega})
+\mathrm{d}\,\psl{p-1}\,(\mathrm{n}\boldsymbol{\omega}).
\end{equation}
Forms $\omega\in Y^p(\Omega^\mathrm{c}\cap\Omega^R)\cap\dfs{p}(\Omega^\mathrm{c}\cap\Omega^R)$ can be represented by 
\begin{equation}\label{representcompactexterior}
\boldsymbol{\omega}=-\psl{p}\,(\gamma_\mathrm{N}^\mathrm{c}\boldsymbol{\omega})
+\pdl{p}\,(\gamma_\mathrm{D}^\mathrm{c}\boldsymbol{\omega})
-\mathrm{d}\,\psl{p-1}\,(\mathrm{n}^\mathrm{c}\boldsymbol{\omega})+M\boldsymbol{\omega},
\end{equation}
where the term $M\omega$ contains integral transformations of Dirichlet and Neumann data on the far boundary $\Gamma^R$. Eventually, forms $\omega\in Y^p((\Omega\cup\Omega^\mathrm{c})\cap\Omega^R)\cap\dfs{p}((\Omega\cup\Omega^\mathrm{c})\cap\Omega^R)$ can be represented by
\begin{equation}\label{representcompact}
\boldsymbol{\omega}=-\psl{p}\,\boldsymbol{\gamma}
+\pdl{p}\,\boldsymbol{\beta}
-\mathrm{d}\psl{p-1}\,\boldsymbol{\varphi}+M\boldsymbol{\omega}.
\end{equation}
\end{lemma}
\begin{proof}
\subparagraph{Representation Formula on $\Omega$}
For smooth forms $\psi\in\dfs{p+1}(\Omega),\eta\in\dfs{p}(\Omega)$, we start from the integration by parts formula
\begin{equation}\label{partialintl2}
\ltwo{\delta\boldsymbol{\psi}}{\boldsymbol{\eta}}{p}{\Omega}-\ltwo{\boldsymbol{\psi}}{\mathrm{d}\boldsymbol{\eta}}{p+1}{\Omega}=-\ltwo{\mathrm{n}\boldsymbol{\psi}}{\mathrm{t}\boldsymbol{\eta}}{p}{\Gamma}.
\end{equation}
\marginpar{Green's first identity}Substituting $\boldsymbol{\psi}$ by $\mathrm{d}\boldsymbol{\omega}$, $\omega\in\dfs{p}(\Omega)$, yields Green's first identity
\begin{equation*}
\ltwo{\delta\mathrm{d}\boldsymbol{\omega}}{\boldsymbol{\eta}}{p}{\Omega}-\ltwo{\mathrm{d}\boldsymbol{\omega}}{\mathrm{d}\boldsymbol{\eta}}{p+1}{\Omega}=-\ltwo{\mathrm{n}\mathrm{d}\boldsymbol{\omega}}{\mathrm{t}\boldsymbol{\eta}}{p}{\Gamma}.
\end{equation*}
\marginpar{Green's second identities}We insert the $k^2$ term, rewrite the equality with arguments swapped, and then subtract both equations. We thus find Green's second identity of the so-called $\textbf{curl}\,\textbf{curl}$ type
\begin{gather}
\ltwo{(\delta\mathrm{d}-k^2)\boldsymbol{\omega}}{\boldsymbol{\eta}}{p}{\Omega}-
\ltwo{\boldsymbol{\omega}}{(\delta\mathrm{d}-\overline{k}^2)\boldsymbol{\eta}}{p}{\Omega}\nonumber\\[1mm]
=\ltwo{\gamma_{\mathrm{D}}\boldsymbol{\omega}}{\gamma_{\mathrm{N}}\boldsymbol{\eta}}{p}{\Gamma}
-\ltwo{\gamma_{\mathrm{N}}\boldsymbol{\omega}}{\gamma_{\mathrm{D}}\boldsymbol{\eta}}{p}{\Gamma}.
\label{greenid2}
\end{gather}
For completeness and further reference we also give Green's second identity of the so-called $\textbf{grad}\,\mathrm{div}$ type,
\begin{gather}
\ltwo{\mathrm{d}\delta\boldsymbol{\omega}}{\boldsymbol{\eta}}{p}{\Omega}-
\ltwo{\boldsymbol{\omega}}{\mathrm{d}\delta\boldsymbol{\eta}}{p}{\Omega}
=\ltwo{\mathrm{t}\delta\boldsymbol{\omega}}{\mathrm{n}\boldsymbol{\eta}}{p-1}{\Gamma}
-\ltwo{\mathrm{n}\boldsymbol{\omega}}{\mathrm{t}\delta\boldsymbol{\eta}}{p-1}{\Gamma}.
\label{greenid3}
\end{gather}
Now, replacing $\boldsymbol{\eta}$ by the fundamental solution \eqref{deffundamentalp}, equation \eqref{greenid2} becomes an integral transformation between $p$-forms. Moreover, we require that $\boldsymbol{\omega}$ is a solution of \eqref{defproblem} and obtain
\begin{equation}\label{representaux1}
\boldsymbol{\omega}-\ltwo{\boldsymbol{\omega}}{\mathrm{d}\delta\boldsymbol{G}_p}{p}{\Omega}
=\ltwo{\gamma_{\mathrm{N}}\boldsymbol{\omega}}{\gamma_{\mathrm{D}}\boldsymbol{G}_p}{p}{\Gamma}
-\ltwo{\gamma_{\mathrm{D}}\boldsymbol{\omega}}{\gamma_{\mathrm{N}}\boldsymbol{G}_p}{p}{\Gamma}.
\end{equation}
The second term can be decomposed via integration by parts \eqref{partialintl2},
\begin{equation}\label{representaux2}
\ltwo{\boldsymbol{\omega}}{\mathrm{d}\delta\boldsymbol{G}_p}{p}{\Omega}=
\ltwo{\delta\boldsymbol{\omega}}{\delta\boldsymbol{G}_p}{{p-1}}{\Omega}+
\ltwo{\mathrm{n}\boldsymbol{\omega}}{\mathrm{t}\delta\boldsymbol{G}_p}{{p-1}}{\Gamma}.
\end{equation}
The domain integral on the right-hand side vanishes due to the gauge condition \eqref{defgauge}.\par
With \eqref{representaux1} and \eqref{representaux2} we can write down a first version of the representation formula for $\boldsymbol{\omega}\in Y^p(\Omega)\cap \dfs{p}(\Omega)$,
\[
\boldsymbol{\omega}
=\ltwo{\gamma_\mathrm{N}\boldsymbol{\omega}}{\gamma_{\mathrm{D}}\boldsymbol{G}_p}{p}{\Gamma}
-\ltwo{\gamma_\mathrm{D}\boldsymbol{\omega}}{\gamma_{\mathrm{N}}\boldsymbol{G}_p}{p}{\Gamma}
+\ltwo{\mathrm{n}\boldsymbol{\omega}}{\mathrm{t}\delta\boldsymbol{G}_p}{{p-1}}{\Gamma}.
\]
The $p$-form $\omega$ is not defined for observation points $\ppoint^{\prime}$ on the boundary $\partial\Omega$.\par
With the help of (\ref{ipprop}c, d) we obtain a second version of the representation formula,
\begin{equation}\label{representaux3}
{\boldsymbol{\omega}
=\ltwo{\gamma_\mathrm{N}\boldsymbol{\omega}}{\mathrm{t}\boldsymbol{G}_p}{p}{\Gamma}
-\delta\ltwo{\gamma_\mathrm{D}\boldsymbol{\omega}}{\mathrm{n}\boldsymbol{G}_{p+1}}{p}{\Gamma}
+\mathrm{d}\,\ltwo{\mathrm{n}\boldsymbol{\omega}}{\mathrm{t}\boldsymbol{G}_{p-1}}{{p-1}}{\Gamma},}
\end{equation}
which effectively constitutes a Hodge decomposition of $\boldsymbol{\omega}$ for $k=0$. Finally, with the definitions \eqref{deflayerpotentials}, and with \eqref{dldef2}, the representation formula \eqref{representaux3} can be cast into the compact form \eqref{representcompactinterior},
\begin{equation*}
\boldsymbol{\omega}=\psl{p}\,(\gamma_\mathrm{N}\boldsymbol{\omega})
-\pdl{p}\,(\gamma_\mathrm{D}\boldsymbol{\omega})
+\mathrm{d}\,\psl{p-1}\,(\mathrm{n}\boldsymbol{\omega}).
\end{equation*}
The definition of the layer potentials can be extended to $\Omega\cup\Omega^\mathrm{c}$, admitting observation points $\ppoint^{\prime}\in\Omega^\mathrm{c}$ as well. By construction, $\boldsymbol{\omega}=0$ for $\ppoint^{\prime}\in\Omega^\mathrm{c}$.\par
\subparagraph{Representation Formula on $\Omega^\mathrm{c}\cap\Omega^R$}
The same line of reasoning can be applied to the bounded exterior domain $\Omega^\mathrm{c}\cap\Omega^R$, which yields \eqref{representcompactexterior} for $\ppoint^{\prime}\in\Omega^\mathrm{c}\cap\Omega^R$
\begin{equation*}
\boldsymbol{\omega}=-\psl{p}\,(\gamma_\mathrm{N}^\mathrm{c}\boldsymbol{\omega})
+\pdl{p}\,(\gamma_\mathrm{D}^\mathrm{c}\boldsymbol{\omega})
-\mathrm{d}\,\psl{p-1}\,(\mathrm{n}^\mathrm{c}\boldsymbol{\omega})+M\boldsymbol{\omega}.
\end{equation*}
The extra term $M\boldsymbol{\omega}$ collects the boundary integrals on the far boundary $\Gamma^R$,
\begin{equation}\label{defm}
M\boldsymbol{\omega}=\ltwo{\gamma_\mathrm{N}^R\boldsymbol{\omega}}{\gamma_{\mathrm{D}}^R\boldsymbol{G}_p}{p}{\Gamma^R}
-\ltwo{\gamma_\mathrm{D}^R\boldsymbol{\omega}}{\gamma_{\mathrm{N}}^R\boldsymbol{G}_p}{p}{\Gamma^R}
+\ltwo{\mathrm{n}^R\boldsymbol{\omega}}{\mathrm{t}^R\delta\boldsymbol{G}_p}{p-1}{\Gamma^R}.
\end{equation}
By construction, $\boldsymbol{\omega}=0$ for $\ppoint^{\prime}\in\Omega$. 
\subparagraph{Representation Formula on $(\Omega\cup\Omega^\mathrm{c})\cap\Omega^R$}
Adding up \eqref{representcompactinterior} and \eqref{representcompactexterior}, we obtain the representation formula \eqref{representcompact} which is valid for $\ppoint^{\prime}\in(\Omega\cup\Omega^\mathrm{c})\cap\Omega^R$,
\begin{equation*}
\boldsymbol{\omega}=-\psl{p}\,\boldsymbol{\gamma}
+\pdl{p}\,\boldsymbol{\beta}
-\mathrm{d}\,\psl{p-1}\,\boldsymbol{\varphi}+M\boldsymbol{\omega}.
\end{equation*}
\qed\end{proof}
\begin{remark}
If $\boldsymbol{\omega}\in Y^p(\Omega\cup\Omega^\mathrm{c})\cap \dfs{p}(\Omega\cup\Omega^\mathrm{c})$, then $M\boldsymbol{\omega}$ is independent of the exact location and shape of $\Gamma^R$. This can be seen by introducing another $\Omega^{\tilde{R}}\Supset\Omega^R$, and applying Green's identities on the annulus $\Omega^{\tilde{R}}\setminus\overline{\Omega}^R$.
\end{remark}
\begin{lemma}\label{lemma:lemmafarbound}
If $\boldsymbol{\omega}\in X^p(\Omega\cup\Omega^\mathrm{c})\cap\dfs{p}(\Omega\cup\Omega^\mathrm{c})$, i.e., $\omega$ satisifies the radiation condition \eqref{condradiation1} for $k\ne 0$ or the decay condition \eqref{conddecay1} for $k=0$, then $|M\omega|_{\ppoint^{\prime}}=o(1)$ for all $\ppoint^{\prime}\in\Omega^\mathrm{c}\cap\Omega^R$ as $R\rightarrow\infty$. Consequently, the far-boundary terms in \eqref{representcompactexterior} and \eqref{representcompact} vanish.
\end{lemma}
To prove Lemma~\ref{lemma:lemmafarbound} we need the following corollary about norms of the Green kernel. Consider some trace of the form $\gamma^R_\mathrm{X}\boldsymbol{G}_p$. For a fixed observation point $\ppoint^\prime$ this expression gives a covector-valued form on $\Gamma^R$. To define a norm, we first take the point-wise norm of the covector anchored in $\ppoint^{\prime}$. This leaves us with an ordinary form on $\Gamma^R$, whose $L^2$ norm we will then consider. We therefore define
\[
\bigltwonorm{\gamma^R_\mathrm{X}\boldsymbol{G}_p}{p}{\Gamma^R}^{\ppoint^\prime}=
\bigltwonorm{|\gamma^R_\mathrm{X}\boldsymbol{G}_p|_{\ppoint^\prime}}{p}{\Gamma^R}.
\]
Equivalently,
\[
\bigltwonorm{\gamma^R_\mathrm{X}\boldsymbol{G}_p}{p}{\Gamma^R}^{\ppoint^\prime}
=\sup
\bigltwonorm{\gamma^R_\mathrm{X}\metric{\boldsymbol{G}_p}{\boldsymbol{\eta}}_{\ppoint^{\prime}}}{p}{\Gamma^R}
=\sup
\bigltwonorm{\gamma^R_\mathrm{X}(\overline{g}_n\boldsymbol{\eta})}{p}{\Gamma^R},
\]
where the supremum is taken here and subsequently over all covariantly constant $p$-forms $\boldsymbol{\eta}\in\dfs{\text{const},p}(\Omega\cup\Omega^\mathrm{c})$ subject to $|\boldsymbol{\eta}|_{\ppoint^{\prime}}=1$. The last equation follows with \eqref{defip} and \eqref{deffundamentalp}.
\begin{corollary}\label{corollary:normgreenkernel}
The norms of traces of the Green kernel on the far boundary are bounded by the norms of traces of the scalar fundamental solution,
\begin{subequations}\label{greentraces}
\begin{align}
\ltwonorm{\gamma^R_\mathrm{D}\boldsymbol{G}_p}{p}{\Gamma^R}^{\ppoint^\prime}&\le
\ltwonorm{\gamma_\mathrm{D}^R\boldsymbol{G}_0}{0}{\Gamma^R}^{\ppoint^\prime},\label{greentraces1}\\
\ltwonorm{(\gamma_\mathrm{N}^R-\overline{ik}\gamma_\mathrm{D}^R)\boldsymbol{G}_p}{p}{\Gamma^R}^{\ppoint^\prime}&\le
\ltwonorm{
(\gamma_\mathrm{N}^R-\overline{ik}\gamma_\mathrm{D}^R)\boldsymbol{G}_0
}{0}{\Gamma^R}^{\ppoint^\prime}+
\ltwonorm{
\hhat{\mathrm{d}}\gamma_\mathrm{D}^R\boldsymbol{G}_0}{1}{\Gamma^R}^{\ppoint^\prime}.\label{greentraces2}
\end{align}
\end{subequations}
The inequality for the Neumann trace features an additional term, which involves tangential derivatives.
\end{corollary}
\begin{proof}
We denote the unit volume form on $\Gamma^R$ by $\hhat{\boldsymbol{\mu}}$ and consider
\begin{align}
\bigl(\ltwonorm{\gamma^R_\mathrm{D}\boldsymbol{G}_p}{p}{\Gamma^R}^{\ppoint^\prime}\bigr)^2&=\sup
\ltwonorm{\gamma^R_\mathrm{D}(\overline{g}_n\boldsymbol{\eta})}{p}{\Gamma^R}^2
=\sup
\ltwonorm{(\gamma^R_\mathrm{D}\overline{g}_n)(\gamma^R_\mathrm{D}\boldsymbol{\eta})}{p}{\Gamma^R}^2
\nonumber\\
&=\sup
\int_{\Gamma^R}|\gamma_\mathrm{D}^R\overline{g}_n|^2\,|\gamma_\mathrm{D}^R\boldsymbol{\eta}|^2\,\hhat{\boldsymbol{\mu}}\nonumber\\
&\le\int_{\Gamma^R}|\gamma_\mathrm{D}^R\overline{g}_n|^2\,\hhat{\boldsymbol{\mu}}=\bigr(\ltwonorm{\gamma_\mathrm{D}^R\boldsymbol{G}_0}{0}{\Gamma^R}^{\ppoint^\prime}\bigl)^2,\label{maux4}
\end{align}
where we used $|\gamma_\mathrm{D}^R\boldsymbol{\eta}|_\ppoint\le|\boldsymbol{\eta}|_{\iota \ppoint}=1$. This proves \eqref{greentraces1}.\par
To derive an estimate for the Neumann trace of the Green kernel, we first have to consider
\begin{align*}
\gamma_\mathrm{N}^R(\overline{g}_n\boldsymbol{\eta})&=\gamma_\mathrm{D}^R\mathrm{i}_{\vec{n}}\mathrm{d}(\overline{g}_n\boldsymbol{\eta})
=\gamma_\mathrm{D}^R\mathrm{i}_{\vec{n}}(\mathrm{d}\overline{g}_n\wedge\boldsymbol{\eta})\\
&=\gamma_\mathrm{D}^R(\mathrm{i}_{\vec{n}}\mathrm{d}\overline{g}_n\wedge\boldsymbol{\eta}-\mathrm{d}\overline{g}_n\wedge\mathrm{i}_{\vec{n}}\boldsymbol{\eta})\\
&=(\gamma_\mathrm{N}^R\overline{g}_n)(\gamma_\mathrm{D}^R\boldsymbol{\eta})-\hhat{\mathrm{d}}\gamma_\mathrm{D}^R\overline{g}_n\wedge\mathrm{n}^R\boldsymbol{\eta},
\end{align*}
where we used \eqref{condaux1}, and therefore
\[
(\gamma_\mathrm{N}^R-\overline{ik}\gamma_\mathrm{D}^R)(\overline{g}_n\boldsymbol{\eta})=\bigl((\gamma_\mathrm{N}^R-\overline{ik}\gamma_\mathrm{D}^R)\overline{g}_n\bigr)(\gamma_\mathrm{D}^R\boldsymbol{\eta})-\hhat{\mathrm{d}}\gamma_\mathrm{D}^R\overline{g}_n\wedge\mathrm{n}^R\boldsymbol{\eta}.
\]
We now tackle
\begin{align*}
\ltwonorm{(\gamma_\mathrm{N}^R-\overline{ik}\gamma_\mathrm{D}^R)\boldsymbol{G}_p}{p}{\Gamma^R}^{\ppoint^\prime}&=\sup
\ltwonorm{
(\gamma_\mathrm{N}^R-\overline{ik}\gamma_\mathrm{D}^R)(\overline{g}_n\boldsymbol{\eta})}{p}{\Gamma^R}\\
&=\sup
\ltwonorm{
\bigl((\gamma_\mathrm{N}^R-\overline{ik}\gamma_\mathrm{D}^R)\overline{g}_n\bigr)(\gamma_\mathrm{D}^R\boldsymbol{\eta})-\hhat{\mathrm{d}}\gamma_\mathrm{D}^R\overline{g}_n\wedge\mathrm{n}^R\boldsymbol{\eta}
}{p}{\Gamma^R}\\
&
\begin{aligned}
\le &\sup
\ltwonorm{
\bigl((\gamma_\mathrm{N}^R-\overline{ik}\gamma_\mathrm{D}^R)\overline{g}_n\bigr)(\gamma_\mathrm{D}^R\boldsymbol{\eta})
}{p}{\Gamma^R}\\
&+\sup
\ltwonorm{
\hhat{\mathrm{d}}\gamma_\mathrm{D}^R\overline{g}_n\wedge\mathrm{n}^R\boldsymbol{\eta}
}{p}{\Gamma^R}.
\end{aligned}
\end{align*}
Recall that the supremum is taken over all covariantly constant $p$-forms $\boldsymbol{\eta}\in\dfs{\text{const},p}(\Omega\cup\Omega^\mathrm{c})$ subject to $|\boldsymbol{\eta}|_{\ppoint^{\prime}}=1$. Along the same line as in \eqref{maux4} we obtain for the first term
\begin{align*}
\sup
\ltwonorm{
\bigl((\gamma_\mathrm{N}^R-\overline{ik}\gamma_\mathrm{D}^R)\overline{g}_n\bigr)(\gamma_\mathrm{D}^R\boldsymbol{\eta})
}{p}{\Gamma^R}^2\le
\bigr(\ltwonorm{
(\gamma_\mathrm{N}^R-\overline{ik}\gamma_\mathrm{D}^R)\boldsymbol{G}_0
}{0}{\Gamma^R}^{\ppoint^\prime}\bigl)^2.
\end{align*}
For the second term,
\begin{align*}
\sup
\ltwonorm{
\hhat{\mathrm{d}}\gamma_\mathrm{D}^R\overline{g}_n\wedge\mathrm{n}^R\boldsymbol{\eta}
}{p}{\Gamma^R}^2
&=\sup
\int_{\Gamma^R}|
\hhat{\mathrm{d}}\gamma_\mathrm{D}^R\overline{g}_n\wedge\mathrm{n}^R\boldsymbol{\eta}
|^2_\ppoint\,\hhat{\boldsymbol{\mu}}\\
&\le\sup
\int_{\Gamma^R}|
\hhat{\mathrm{d}}\gamma_\mathrm{D}^R\overline{g}_n|^2_\ppoint\,|\mathrm{n}^R\boldsymbol{\eta}
|^2_\ppoint\,\hhat{\boldsymbol{\mu}}\\
&\le\int_{\Gamma^R}|
\hhat{\mathrm{d}}\gamma_\mathrm{D}^R\overline{g}_n|^2_\ppoint\,\hhat{\boldsymbol{\mu}}
=\bigl(\ltwonorm{
\hhat{\mathrm{d}}\gamma_\mathrm{D}^R\boldsymbol{G}_0}{1}{\Gamma^R}^{\ppoint^\prime}\bigr)^2,
\end{align*}
where we used that for a 1-form $\boldsymbol{\nu}$ it holds that $|\boldsymbol{\nu}\wedge\boldsymbol{\omega}|_\ppoint\le|\boldsymbol{\nu}|_\ppoint|\boldsymbol{\omega}|_\ppoint$, and
$|\mathrm{n}^R\boldsymbol{\eta}|_\ppoint\le|\boldsymbol{\eta}|_{\iota \ppoint}=1$. This proves \eqref{greentraces2}.
\qed\end{proof}
\begin{proof}[Lemma~\ref{lemma:lemmafarbound}]
It has been shown in \cite[Lemma 7.11, Thms.\ 8.9 and 9.6]{mclean} for the case $p=0$ that the asymptotic conditions \eqref{condradiation1} and \eqref{conddecay1} imply $M\boldsymbol{\omega}=0$ for $R\rightarrow\infty$, and vice versa. In what follows we generalize the proof of the first implication to arbitrary $p$, i.e., we demonstrate that the representation formula works without extra terms due to the far boundary. For the converse implication see Remark~\ref{remark:decaylp}.\par
To this end, we write \eqref{defm} in a particular form. From \eqref{defcoderivative} and (\ref{defboundary}b) we infer $\mathrm{t}\,\delta=-\hhat{*}^{-1}\gamma_\mathrm{N}*$. Taking this into account yields
\begin{align}
M\boldsymbol{\omega}=&\bigltwo{(\gamma_\mathrm{N}^R-ik\gamma_\mathrm{D}^R)\boldsymbol{\omega}}{\gamma_{\mathrm{D}}^R\boldsymbol{G}_p}{p}{\Gamma^R}
-\bigltwo{\gamma_\mathrm{D}^R\boldsymbol{\omega}}{(\gamma_\mathrm{N}^R-\overline{ik}\gamma_\mathrm{D}^R)\boldsymbol{G}_p}{p}{\Gamma^R}\nonumber\\
&+(*^{-1})^{\prime}\bigltwo{(\gamma_\mathrm{N}^R-ik\gamma_\mathrm{D}^R)*\boldsymbol{\omega}}{\gamma_{\mathrm{D}}^R\boldsymbol{G}_q}{q}{\Gamma^R}\nonumber\\
&-(*^{-1})^{\prime}\bigltwo{\gamma_\mathrm{D}^R*\boldsymbol{\omega}}{(\gamma_\mathrm{N}^R-\overline{ik}\gamma_\mathrm{D}^R)\boldsymbol{G}_q}{q}{\Gamma^R},\label{maux1}
\end{align}
where we used \eqref{ipprop1}, \eqref{defnormaltrace} and \eqref{defgauge}, i.e.\ $\gamma_\mathrm{N}^R*\boldsymbol{\omega}=0$.\par
Consider an integral transformation of the form $\boldsymbol{\omega}=\ltwo{\hat{\boldsymbol{\omega}}}{\gamma^R_\mathrm{X}\boldsymbol{G}_p}{p}{\Gamma^R}$, and a covariantly constant $p$-form $\boldsymbol{\eta}\in\dfs{\text{const},p}(\Omega\cup\Omega^\mathrm{c})$, subject to $|\boldsymbol{\eta}|_{\ppoint^{\prime}}=1$. Then it holds that
\begin{align*}
|\boldsymbol{\omega}|_{\ppoint^{\prime}}&=\sup
\metric{\boldsymbol{\omega}}{\boldsymbol{\eta}}_{\ppoint^{\prime}}=\sup
\bigltwo{\hat{\boldsymbol{\omega}}}{\gamma^R_\mathrm{X}\metric{\boldsymbol{G}_p}{\boldsymbol{\eta}}_{\ppoint^{\prime}}}{p}{\Gamma^R}\\
&\le\ltwonorm{\hat{\boldsymbol{\omega}}}{p}{\Gamma^R}\sup
\bigltwonorm{\gamma^R_\mathrm{X}\metric{\boldsymbol{G}_p}{\boldsymbol{\eta}}_{\ppoint^{\prime}}}{p}{\Gamma^R}
=\ltwonorm{\hat{\boldsymbol{\omega}}}{p}{\Gamma^R}\,\bigltwonorm{\gamma^R_\mathrm{X}\boldsymbol{G}_p}{p}{\Gamma^R}^{\ppoint^\prime},
\end{align*}
where we used the Cauchy-Schwarz inequality. The supremum is taken over all admissible $\eta$, compare Corollary~\ref{corollary:normgreenkernel}. Then, from \eqref{maux1} with the triangle inequality it follows that
\begin{align}
|M\boldsymbol{\omega}|_{\ppoint^{\prime}}\le\, &\bigltwonorm{(\gamma_\mathrm{N}^R-ik\gamma_\mathrm{D}^R)\boldsymbol{\omega}}{p}{\Gamma^R}
\,\bigltwonorm{\gamma_{\mathrm{D}}^R\boldsymbol{G}_p}{p}{\Gamma^R}^{\ppoint^\prime}\nonumber\\
&+\bigltwonorm{\gamma_\mathrm{D}^R\boldsymbol{\omega}}{p}{\Gamma^R}
\,\bigltwonorm{(\gamma_\mathrm{N}^R-\overline{ik}\gamma_\mathrm{D}^R)\boldsymbol{G}_p}{p}{\Gamma^R}^{\ppoint^\prime}\nonumber\\
&+\bigltwonorm{ik\mathrm{n}^R\boldsymbol{\omega}}{p-1}{\Gamma^R}
\,\bigltwonorm{\gamma_{\mathrm{D}}^R\boldsymbol{G}_q}{p}{\Gamma^R}^{\ppoint^\prime}\nonumber\\
&+\bigltwonorm{\mathrm{n}^R\boldsymbol{\omega}}{p-1}{\Gamma^R}
\,\bigltwonorm{(\gamma_\mathrm{N}^R-\overline{ik}\gamma_\mathrm{D}^R)\boldsymbol{G}_q}{p}{\Gamma^R}^{\ppoint^\prime},\label{maux2}
\end{align}
where we used $\mathrm{t}\,*=\hhat{*}\mathrm{n}$, see \eqref{defnormaltrace}.\par
We will demonstrate that $|M\boldsymbol{\omega}|_{\ppoint^{\prime}}=o(1)$ for all $\ppoint^{\prime}\in\Omega^\mathrm{c}\cap\Omega^R$ as $R\to\infty$. Since $M\boldsymbol{\omega}$ is actually independent of the exact location and shape of $\Gamma^R$, this implies $M\boldsymbol{\omega}=0$. We leverage this freedom and pick a large ball with radius $R$ centered at a fixed point in $\Omega$. Three cases need to be distinguished:
\begin{enumerate}
\item {\em Case $k\ne 0$.} It holds that
\begin{equation}\label{maux3}
\left.
\begin{aligned}
\bigltwonorm{(\gamma_\mathrm{N}^R-ik\gamma_\mathrm{D}^R)\boldsymbol{\omega}}{p}{\Gamma^R}&=o(1),\\
\bigltwonorm{\gamma_\mathrm{D}^R\boldsymbol{\omega}}{p}{\Gamma^R}&=\mathcal{O}(1),\\
\bigltwonorm{\mathrm{n}^R\boldsymbol{\omega}}{p-1}{\Gamma^R}&=o(1),
\end{aligned}\quad\right\}
\end{equation}
for $R\to\infty$, see \eqref{condradiation3} and \eqref{boundtraces}, respectively. 
The scalar fundamental solution $g_n$ enjoys the properties
\begin{subequations}\label{greenradiation1}
\begin{equation}
\left.
\begin{aligned}
g_n&=\mathcal{O}(r^{-(n-1)/2}),\\
\bigl({\textstyle\frac{\partial}{\partial r}}-ik\bigr)g_n&=o(r^{-(n-1)/2}),\\
{\textstyle\frac{\partial}{\partial r}}g_n&=\mathcal{O}(r^{-(n-1)/2}),
\end{aligned}
\quad\right\}
\end{equation}
see \cite[eq.\ (9.10), (9.13)]{mclean}. Note that for $\ppoint\in\Gamma^R$ and fixed $\ppoint^{\prime}\in\Omega^\mathrm{c}\cap\Omega^R$
\begin{align}
\bigl|\hhat{\mathrm{d}}\gamma_\mathrm{D}^R\boldsymbol{G}_0\bigr|_\ppoint&=\bigl|\mathrm{d}R\wedge\mathrm{d}g_n\bigr|_{\iota \ppoint}
=\bigl|\mathrm{d}R\wedge\mathrm{d}r\bigr|_{\iota \ppoint}\bigl|{\textstyle\frac{\partial}{\partial r}}g_n\bigr|\nonumber\\
&=\mathcal{O}(r^{-1})\mathcal{O}(r^{-(n-1)/2})=o(r^{-(n-1)/2}).\label{greenradiation1b}
\end{align}
\end{subequations}
We win one order of decay for the tangential derivatives by the spherical symmetry of the fundamental solution. In fact, if the ball $\Omega^R$ happened to be centered in $\ppoint^{\prime}$ we would have $r=R$ and $\mathrm{d}R\wedge\mathrm{d}r=0$.\,\footnote{
For $n=3$, $p=1$, the expression $\mathrm{d}R\wedge\mathrm{d}g_n$ corresponds to $\vec{n}\times\textbf{grad}\,g_3$, see \cite[p.\ 161]{kress1998}.}\par
By standard asymptotics for Hankel functions it can be confirmed that the asymptotic orders in \eqref{greenradiation1} stay the same if we replace $r$ by $R$. It follows then from \eqref{greentraces} and \eqref{greenradiation1} that
\begin{equation}\label{greenradiation2}
\left.\begin{aligned}
\bigltwonorm{\gamma^R_\mathrm{D}\boldsymbol{G}_p}{p}{\Gamma^R}^{\ppoint^\prime}&=\mathcal{O}(1),\\
\bigltwonorm{(\gamma_\mathrm{N}^R-\overline{ik}\gamma_\mathrm{D}^R)\boldsymbol{G}_p}{p}{\Gamma^R}^{\ppoint^\prime}&=o(1),
\end{aligned}\quad\right\}
\end{equation}
for $R\to\infty$. Finally, \eqref{maux3} and \eqref{greenradiation2} in connection with \eqref{maux2} yield the assertion.
\item {\em Case $k=0$, $n\ge 3$.} The asymptotic conditions for the traces of $\boldsymbol{\omega}$ are given by \eqref{conddecay1} and \eqref{conddecay2}, respectively. The decay behaviour of the fundamental solution with $k=0$ is obvious from \eqref{defgreen}. While the third term on the right hand side of \eqref{maux2} vanishes, it is easy to see that the three remaining terms are of order $\mathcal{O}(R^{2-n})$ for $R\to\infty$.
\item {\em Case $k=0$, $n=2$.} For $p=0$ see \cite[Thm.\ 8.9]{mclean}. For $p=1$, there holds $b=0$ in \eqref{conddecay1} and \eqref{conddecay2}, respectively. We can argue like in the previous case. The first term in \eqref{maux2} is of order $\mathcal{O}(R^{-1}\ln R)$, the other two non-zero terms are of order $\mathcal{O}(R^{-1})$, for $R\to\infty$.
\end{enumerate}
\qed\end{proof}
\begin{lemma}\label{lemma:repok}
The representation formula 
\[
\boldsymbol{\omega}=-\psl{p}\,\boldsymbol{\gamma}
+\pdl{p}\,\boldsymbol{\beta}
-\mathrm{d}\,\psl{p-1}\,\boldsymbol{\varphi}
\]
is well defined for boundary data according to
\begin{alignat*}{2}
\beta&=\jump{\gamma_\mathrm{D}}\,\omega&&\in H_\perp^{-1/2}\pcomp{p}(\hhat{\mathrm{d}},\Gamma),\\
\gamma&=\jump{\gamma_\mathrm{N}}\,\omega&&\in H_{\parallel}^{-1/2}\pcomp{p}(\hhat{\delta},\Gamma),\\
\varphi&=\jump{\mathrm{n}}\,\omega&&\in H_{\parallel}^{-1/2}\pcomp{p-1}(\hhat{\delta}0,\Gamma),
\end{alignat*}
where $\boldsymbol{\omega}\in X^p(\Omega\cup\Omega^\mathrm{c})$ is a Maxwell solution.
\end{lemma}
\begin{proof}
Taking into account the regularity of the Green operator and the mapping properties of the trace operator, it can be shown along the same lines as in \cite[eq.\ 5.5, Thm. 5.1]{hiptmair2003} that $\psl{p}$ can be extended to a linear continuous operator
\begin{equation}\label{defpsislsob1}
\psl{p}:H_{\parallel}^{-1/2}\pcomp{p}(\Gamma)\rightarrow H^1_{\mathrm{loc}}\pcomp{p}(M).
\end{equation}
If we restrict its domain to $H_{\parallel}^{-1/2}\pcomp{p}(\hhat{\delta},\Gamma)$ then \eqref{lpotprop1} is well defined and yields that $\delta\psl{p}$ is in $H^1_{\mathrm{loc}}\pcomp{p-1}(M)\subset \hzero _{\mathrm{loc}}\pcomp{p-1}(\mathrm{d},M)$. Therefore we know that $\psl{p}$ is then even in $\hzero _{\mathrm{loc}}\pcomp{p}(\mathrm{d}\delta,M)$. From \eqref{lpotprop2}, which holds in $\Omega$ and $\Omega^\mathrm{c}$ separately, we see that
\begin{equation}\label{defpsislsob2}
\psl{p}:H_{\parallel}^{-1/2}\pcomp{p}(\hhat{\delta},\Gamma)\rightarrow H^1_{\mathrm{loc}}\pcomp{p}(M)\cap \hzero _{\mathrm{loc}}\pcomp{p}(\delta\mathrm{d},\Omega\cup\Omega^\mathrm{c})
\end{equation}
is continuous. Moreover, from \eqref{defpsislsob2},
\begin{equation}\label{defpsislsob4}
\mathrm{d}\,\psl{p}:H_{\parallel}^{-1/2}\pcomp{p}(\hhat{\delta}0,\Gamma)\rightarrow \hzero _{\mathrm{loc}}\pcomp{p+1}(\mathrm{d}0,M) \cap \hzero _{\mathrm{loc}}\pcomp{p+1}(\delta,\Omega\cup\Omega^\mathrm{c}).
\end{equation}
The definition (\ref{deflayerpotentials}b) of the double layer potential, or, equivalently, \eqref{lpotprop3}, extends to a linear and continuous operator
\begin{equation*}
\pdl{p}:H_{\bot}^{-1/2}\pcomp{p}(\Gamma)\rightarrow \hzero _{\mathrm{loc}}\pcomp{p}(\delta,M).
\end{equation*}
If we restrict its domain to $H_{\bot}^{-1/2}\pcomp{p}(\hhat{\mathrm{d}},\Gamma)$, then \eqref{lpotprop4} is well defined in $\Omega$ and $\Omega^\mathrm{c}$ separately and yields that $\mathrm{d}\,\pdl{p}$ is in $\hzero _{\mathrm{loc}}\pcomp{p+1}(\delta,M)$. From this we find that
\begin{equation}\label{defpsidlsob2}
\pdl{p}:H_{\bot}^{-1/2}\pcomp{p}(\hhat{\mathrm{d}},\Gamma)\rightarrow \hzero _{\mathrm{loc}}\pcomp{p}(\delta,M)\cap \hzero _{\mathrm{loc}}\pcomp{p}(\delta\mathrm{d},\Omega\cup\Omega^\mathrm{c})
\end{equation}
is continuous. We conclude that all boundary-integral operators in the representation formula are well defined for the respective boundary data.
\qed\end{proof}
\begin{proof}[Theorem~\ref{thm:repthm}] The proof of Theorem~\ref{thm:repthm} follows directly from the Lemmata~\ref{lemma:lemmarep}, \ref{lemma:lemmafarbound}, and \ref{lemma:repok}.
\qed\end{proof}
\begin{remark}
If we dispense with the requirements $(\delta\mathrm{d}-k^2)\boldsymbol{\omega}=0$ and $\delta\boldsymbol{\omega}=0$, and the asymptotic condition, the most complete statement of the representation formula for forms in $\hzero _{\mathrm{loc}}\pcomp{p}(\delta\mathrm{d},\Omega\cup\Omega^\mathrm{c})\cap \hzero _{\mathrm{loc}}\pcomp{p}(\delta,\Omega\cup\Omega^\mathrm{c})$ reads
\begin{eqnarray*}
\boldsymbol{\omega}&=&-\psl{p}\,\boldsymbol{\gamma}
+\pdl{p}\,\boldsymbol{\beta}
-\mathrm{d}\,\psl{p-1}\,\boldsymbol{\varphi}+M\boldsymbol{\omega}\\
&&+\ltwo{(\delta\mathrm{d}-k^2)\boldsymbol{\omega}}{\boldsymbol{G}_p}{p}{\Omega}
+\mathrm{d}\,\ltwo{\delta\boldsymbol{\omega}}{\boldsymbol{G}_{p-1}}{p-1}{\Omega}.
\end{eqnarray*}
Both $(\delta\mathrm{d}-k^2)\boldsymbol{\omega}$ and $\delta\boldsymbol{\omega}$ have to be compactly supported. Then $\Omega$ can always be enlarged such that it contains their supports, which has been assumed in the above representation formula.
\end{remark}
We proceed to show that individual terms of the representation formula lie in $X^p(\Omega\cup\Omega^\mathrm{c})$. This is a prerequisite for indirect methods, where some Maxwell solution is represented in terms of either single or double layer potential rather than combining them into the representation formula. Lemma~\ref{lemma:repiny} shows that the relevant terms lie in $Y^p(\Omega\cup\Omega^\mathrm{c})$, and Lemma~\ref{lemma:repinx} proves the asymptotic properties.
\begin{lemma}\label{lemma:repiny}
Let $(\boldsymbol{\beta},\boldsymbol{\gamma},\boldsymbol{\varphi})$ be boundary data according to \eqref{defboundary1} and fulfilling \eqref{neumannsolution}. For $k\ne 0$ it holds that $(\psl{p}\boldsymbol{\gamma}
+\mathrm{d}\,\psl{p-1}\boldsymbol{\varphi})$ and $\pdl{p}\boldsymbol{\beta}$ are each elements of $Y^p(\Omega\cup\Omega^\mathrm{c})$. For $k=0$ it holds that $\psl{p}\boldsymbol{\gamma}$, $\mathrm{d}\,\psl{p-1}\boldsymbol{\varphi}$, and $\pdl{p}\boldsymbol{\beta}$ are each elements of $Y^p(\Omega\cup\Omega^\mathrm{c})$. 
\end{lemma}
\begin{proof}
From \eqref{defpsidlsob2} it follows that (\ref{lpotprop}d, g) are well defined, which shows that $\pdl{p}\beta\in Y^p(\Omega\cup\Omega^\mathrm{c})$. We note in passing that also \eqref{lpotprop7} is well-defined provided that its domain is chosen as $H_{\bot}^{-1/2}\pcomp{p}(\hhat{\mathrm{d}},\Gamma)$.
Consider now
\[
\boldsymbol{\omega}=\psl{p}\boldsymbol{\gamma}
+\mathrm{d}\,\psl{p-1}\boldsymbol{\varphi},
\]
with $(\boldsymbol{\gamma},\boldsymbol{\varphi})\in
H_{\parallel}^{-1/2}\pcomp{p}(\hhat{\delta},\Gamma)\times H_{\parallel}^{-1/2}\pcomp{p-1}(\hhat{\delta}0,\Gamma)$ obeying to \eqref{neumannsolution}. From \eqref{defpsislsob2} and \eqref{defpsislsob4} we know that $\boldsymbol{\omega}\in \hzero _{\mathrm{loc}}\pcomp{p}(\delta,\Omega\cup\Omega^\mathrm{c})\cap \hzero _{\mathrm{loc}}\pcomp{p}(\delta\mathrm{d},\Omega\cup\Omega^\mathrm{c})$. We can therefore calculate
\begin{align*}
(\delta\mathrm{d}-k^2)\boldsymbol{\omega}&=
-\mathrm{d}\,\psl{p-1}\hhat{\delta}\boldsymbol{\gamma}-\mathrm{d}\,\psl{p-1}k^2\boldsymbol{\varphi}\\
&=-\mathrm{d}\,\psl{p-1}(\hhat{\delta}\boldsymbol{\gamma}+k^2\boldsymbol{\varphi})=0,\\
\delta\boldsymbol{\omega}&=\psl{p-1}\hhat{\delta}\boldsymbol{\gamma}+k^2\psl{p-1}\boldsymbol{\varphi}-\mathrm{d}\,\psl{p-1}\hhat{\delta}\boldsymbol{\varphi}\\
&=\psl{p-1}(\hhat{\delta}\boldsymbol{\gamma}+k^2\boldsymbol{\varphi})-\mathrm{d}\,\psl{p-1}\hhat{\delta}\boldsymbol{\varphi}=0,
\end{align*}
where we used (\ref{lpotprop}a, b). We have thus shown that $\psl{p}\boldsymbol{\gamma}
+\mathrm{d}\,\psl{p-1}\boldsymbol{\varphi}\in Y^p(\Omega\cup\Omega^\mathrm{c})$. The case $k=0$ follows from the above formulae and \eqref{lpotprop1}.
\qed\end{proof}
\begin{lemma}\label{lemma:repinx}
Let $(\boldsymbol{\beta},\boldsymbol{\gamma},\boldsymbol{\varphi})$ be boundary data according to \eqref{defboundary1} and fulfilling \eqref{neumannsolution}. For $k\ne 0$ it holds that $(\psl{p}\boldsymbol{\gamma}
+\mathrm{d}\,\psl{p-1}\boldsymbol{\varphi})$ and $\pdl{p}\boldsymbol{\beta}$ each satisfy the radiation condition \eqref{condradiation1}. For $k=0$ it holds that $\psl{p}\boldsymbol{\gamma}$, $\mathrm{d}\,\psl{p-1}\boldsymbol{\varphi}$, and $\pdl{p}\boldsymbol{\beta}$ each satisfy the decay condition \eqref{conddecay1}.
\end{lemma}
\begin{proof}
\begin{enumerate}
\item[]
\item {\em Case $k\ne 0$.} We rely on \cite[Lemma 3.2]{jawerth1995}, which tells us that $\boldsymbol{\omega}=\psl{p}\boldsymbol{\gamma}$ as well as its derivatives $\mathrm{d}$ and $\delta$ satisfy the following radiation condition
\begin{equation}\label{condradiation6}
|\mathrm{i}_{\vec{n}}\mathrm{d}\boldsymbol{\omega}+\mathrm{j}_{\mathrm{d}R}\delta\boldsymbol{\omega}-ik\boldsymbol{\omega}|_{\iota \ppoint^\prime}=o(R^{-(n-1)/2}),
\end{equation}
uniformly for $R\to\infty$. Note the extra term $\mathrm{j}_{\mathrm{d}R}\delta\boldsymbol{\omega}$ that controls the coderivative, which is non-zero, in general. Let us define the linear maps that correspond to the conditions \eqref{condradiation1} and \eqref{condradiation6}, respectively,
\begin{align*}
f(\boldsymbol{\omega})&=\mathrm{i}_{\vec{n}}\mathrm{d}\boldsymbol{\omega}-ik\boldsymbol{\omega},\\
g(\boldsymbol{\omega})&=\mathrm{i}_{\vec{n}}\mathrm{d}\boldsymbol{\omega}+\mathrm{j}_{\mathrm{d}R}\delta\boldsymbol{\omega}-ik\boldsymbol{\omega}.
\end{align*}
Both maps coincide on $Y^p(\Omega\cup\Omega^\mathrm{c})$. From Lemma~\ref{lemma:repiny} we know that $\psl{p}\boldsymbol{\gamma}
+\mathrm{d}\,\psl{p-1}\boldsymbol{\varphi}$ and $\pdl{p}\boldsymbol{\beta}$ lie in $Y^p(\Omega\cup\Omega^\mathrm{c})$. Moreover, it is easy to see that $g(*\boldsymbol{\omega})=*g(\boldsymbol{\omega})$. We can therefore argue
\begin{align*}
|f(\psl{p}\boldsymbol{\gamma}
+\mathrm{d}\,\psl{p-1}\boldsymbol{\varphi})|_{\iota \ppoint^\prime}
&=|g(\psl{p}\boldsymbol{\gamma}
+\mathrm{d}\,\psl{p-1}\boldsymbol{\varphi})|_{\iota \ppoint^\prime}\\
&\le|g(\psl{p}\boldsymbol{\gamma})|_{\iota \ppoint^\prime}+|g(\mathrm{d}\,\psl{p-1}\boldsymbol{\varphi})|_{\iota \ppoint^\prime}\\
&=o(R^{-(n-1)/2}),\\
|f(\pdl{p}\boldsymbol{\beta})|_{\iota \ppoint^\prime}
&=|g(\pdl{p}\boldsymbol{\beta})|_{\iota \ppoint^\prime}
=|*g(\mathrm{d}\,\psl{p}\hhat{*}^{-1}\boldsymbol{\beta})|_{\iota \ppoint^\prime}\\
&=|g(\mathrm{d}\,\psl{p}\hhat{*}^{-1}\boldsymbol{\beta})|_{\iota \ppoint^\prime}=o(R^{-(n-1)/2}),
\end{align*}
where we used \eqref{defdl}. This completes the proof for $k\ne 0$.
\item {\em Case $k=0$, except for $n=2$, $p=1$.} The boundary integral operators decay at least as fast as the fundamental solution, compare \eqref{defgreen} with \eqref{conddecay1}, for $k=0$.
\item {\em Case $k=0$, $n=2$, $p=1$.} Define the abbreviations $A=\psl{p}\boldsymbol{\gamma}$, $B=\mathrm{d}\,\psl{p-1}\boldsymbol{\varphi}$ and $C=\pdl{p}\boldsymbol{\beta}$. All terms obey at least condition \eqref{conddecay1}, with $b\ne 0$, according to the previous case. Terms $B$ and $C$ can be expressed each in terms of the exterior derivative of a scalar single layer potential, which yields $\mathcal{O}(R^{-1})$ decay for $R\to \infty$, i.e.\ $b=0$. It remains to deal with term $A$.\par
As a corollary, we consider the more general case $k=0$, $p=n-1$, and let $\boldsymbol{\omega}=\psl{n-1}\boldsymbol{\gamma}$. $\boldsymbol{\gamma}$ is required to be coclosed, $\delta\boldsymbol{\gamma}=0$, therefore $\hhat{*}\boldsymbol{\gamma}$ is a closed 0-form, and must be constant on each connected component of $\Gamma$. Since we assumed a trivial topology, we can for linearity reasons simply set $\boldsymbol{\gamma}=\hhat{*}^{-1}1$. Now it is a well-known result from potential theory that $\pdl{0}1=0$ on the exterior domain $\Omega^\mathrm{c}$, taking into account \eqref{conddecay1}, which prohibits a constant at infinity. But then with \eqref{defdl}
\begin{align}
0&=-*^{-1}\pdl{0}1
=\mathrm{d}\,\psl{n-1}(\hhat{*}^{-1}1)\nonumber\\
&=\mathrm{d}\,\psl{n-1}\boldsymbol{\gamma}=\mathrm{d}\boldsymbol{\omega}.\label{decayaux3}
\end{align}
Thus, with \eqref{lpotprop1} we have $\mathrm{d}\boldsymbol{\omega}=\delta\boldsymbol{\omega}=0$ in $\Omega^\mathrm{c}$.\par
Now returning to the case $n=2$ we apply Stokes' theorem in the annulus $\Omega^\mathrm{c}\cap\Omega^R$ to $\mathrm{d}\boldsymbol{\omega}$ and $\mathrm{d}*\boldsymbol{\omega}$, respectively. We infer that the contour integrals of the tangential and radial components of $\boldsymbol{\omega}$ along the circle $\Gamma^R$ must be constant, as $R\to\infty$. This proves that neither of the components can involve a logarithmic singularity, nor can the modulus, thus $b=0$ in \eqref{conddecay1} for term $A$ as well.
\end{enumerate}
\qed\end{proof}
\begin{remark}
In 2-D electromagnetics, \eqref{decayaux3} has an intuitive interpretation. If $\boldsymbol{\gamma}$ is seen to model a surface current that flows in circumferential direction along $\Gamma$, then $\boldsymbol{\omega}$ would be its vector potential. Equation \eqref{decayaux3} says that the magnetic flux density in the exterior domain vanishes. In other words, \eqref{decayaux3} states that the magnetostatic field of a densely wound cylindrical coil is entirely confined to its interior.
\end{remark}
\begin{remark}\label{remark:decaylp}%
The converse of Lemma~\ref{lemma:lemmafarbound} states that $M\omega=0$ for $R\rightarrow\infty$ implies that $\omega$ fulfills the asymptotic condition. This can be seen easily, since from $M\boldsymbol{\omega}=0$ it follows from \eqref{representcompact} that $\boldsymbol{\omega}$ can be represented solely in terms of the layer potentials. Then from Lemma~\ref{lemma:repinx} we know that $\boldsymbol{\omega}$ satisfies the asymptotic condition.
\end{remark}
\begin{framed}
Equation \eqref{defproblem} encompasses the Helmholtz equation for $p=0$ and the $\textbf{curl}\,\textbf{curl}$ equation for $p=1$. The following table lists the translation into classical calculus for $n=3$, $p=0,1$.\par
\begin{center}
\begin{tabular}{p{1cm}p{2cm}p{3.5cm}p{1.5cm}p{1.5cm}}
\hline\noalign{\smallskip}
$p$ & Translation & $(\delta\mathrm{d}-k^2)\,\boldsymbol{\omega}=0$ & $\gamma_{\mathrm{D}}\,\boldsymbol{\omega}$ & $\gamma_{\mathrm{N}}\,\boldsymbol{\omega}$\\
\noalign{\smallskip}\svhline\noalign{\smallskip}
0 & $u=\Upsilon^0\boldsymbol{\omega}$ & $(\Delta_\mathrm{s}+k^2)\,u=0$ & $\gamma\,u$ & $\gamma_n\,\textbf{grad}\,u$\\
1 & $\vec{u}=\Upsilon^1\boldsymbol{\omega}$ & $(\textbf{curl}\,\textbf{curl}-k^2)\,\vec{u}=\vec{0}$ & $\pi_{\tau}\,\vec{u}$ & $\gamma_{\tau}\,\textbf{curl}\,\vec{u}$\\
\noalign{\smallskip}\hline\noalign{\smallskip}
\end{tabular}
\end{center}
The representation formula \eqref{representcompactinterior} for $n=3$ encompasses and generalizes the classical Kirchhoff ($p=0$) and Stratton-Chu ($p=1$, \cite[Sec.\ 4.15]{stratton}) representation formulae in three dimensions, which read in vector notation
\begin{align*}
u&=\int_{\partial\Omega}\Bigl(g_3\,\frac{\partial u}{\partial n}-\frac{\partial g_3}{\partial n}\,u\Bigr)\mathrm{d}\Gamma,\\
\vec{u}&=\int_{\partial\Omega}g_3\,\textbf{curl}\,\vec{u}\times\vec{n}\,\mathrm{d}\Gamma+\textbf{curl}^{\prime}\int_{\partial\Omega}g_3\,\vec{u}\times\vec{n}\,\mathrm{d}\Gamma+\textbf{grad}^{\prime}\int_{\partial\Omega}g_3\,\vec{u}\cdot\vec{n}\,\mathrm{d}\Gamma.
\end{align*}
Note that for $p=0$ the last term in \eqref{representcompactinterior} vanishes.
\end{framed}
\marginpar{Maxwell single-layer potential}Equation \eqref{neumannsolution} motivates the definition of the Maxwell single layer potential for $k\ne 0$
\begin{alignat}{2}\label{defmodslpot}
\tpsl{p}:H_{\parallel}^{-1/2}\pcomp{p}(\hhat{\delta},\Gamma)\rightarrow X^p(\Omega\cup\Omega^\mathrm{c}):\gamma\mapsto(\psl{p}-\frac{1}{k^2}\mathrm{d}\,\psl{p-1}\hhat{\delta})\gamma,
\end{alignat}
so that the representation formula \eqref{represent} can be written equivalently
\begin{equation}\label{modrepresent}
{\boldsymbol{\omega}=-\tpsl{p}\,\boldsymbol{\gamma}
+\pdl{p}\,\boldsymbol{\beta}.}
\end{equation}
From \eqref{lpotprop1} and \eqref{lpotprop2} it follows that the Maxwell single layer potential can also be written as $\tpsl{p}=k^{-2}\delta\mathrm{d}\,\psl{p}$.\par
We proceed to show the mapping properties of $\tpsl{p}$. From \eqref{defpsislsob2} and \eqref{defpsislsob4} it can be inferred that 
\begin{equation}\label{defpsislsob5}
\tpsl{p}:H_{\parallel}^{-1/2}\pcomp{p}(\hhat{\delta},\Gamma)\rightarrow \hzero _{\mathrm{loc}}\pcomp{p}(\delta,\Omega\cup\Omega^\mathrm{c})\cap \hzero _{\mathrm{loc}}\pcomp{p}(\delta\mathrm{d},\Omega\cup\Omega^\mathrm{c})
\end{equation}
is continuous. 
The Maxwell single layer potential has the following properties, which are obvious from the proof of Lemma~\ref{lemma:repok},
\begin{subequations}\label{lpotpropmod}
\begin{alignat}{2}
\delta\tpsl{}&=0,\label{lpotprop8}\\
(\delta\mathrm{d}-k^2)\tpsl{}&=0.\label{lpotprop9}
\end{alignat}
\end{subequations}
Moreover, from Lemma~\ref{lemma:repinx} it follows immediately that the Maxwell single layer potential satisfies radiation condition \eqref{condradiation1}. This confirms in connection with \eqref{defpsislsob5} and \eqref{lpotpropmod} the mapping property that is stated in definition \eqref{defmodslpot}.
\begin{remark} Representation formula \eqref{modrepresent} without the extra Neumann data $\boldsymbol{\varphi}$ can also be derived following a different route. For $k\ne 0$, the Maxwell-type operator in \eqref{defproblem} admits a fundamental solution, which is defined by the Maxwell Green kernel double $p$-form
\[
\widetilde{\boldsymbol{G}}_p=\bigl(1-\frac{1}{\overline{k}^2}\mathrm{d}\delta\bigr)\boldsymbol{G}_p.
\] 
Plugging this kernel into Green's second identity \eqref{greenid2} yields, after a few manipulations, directly the modified representation formula \eqref{modrepresent}. The disadvantage of this derivation is that the case $k=0$ is lost.\end{remark}
\subsection{Jump Relations of the Layer Potentials}\label{sec:interfac}
\begin{lemma}\label{lemma:lemmajump}
\marginpar{Jump relations}
All relevant jump relations on the interface between $\Omega$ and $\Omega^\mathrm{c}$ are collected in the subsequent table for data $\gamma\in H_{\parallel}^{-1/2}\pcomp{p}(\hhat{\delta},\Gamma)$, $\beta\in H_{\bot}^{-1/2}\pcomp{p}(\hhat{\mathrm{d}},\Gamma)$, and $\varphi\in H_{\parallel}^{-1/2}\pcomp{p}(\hhat{\delta},\Gamma)$.
\begin{center}
\begin{tabular}{p{17mm}*{3}{@{\hspace{3mm}}c@{\hspace{3mm}}}@{\hspace{3mm}}p{22mm}}
\hline\noalign{\smallskip}
Potential & $\jump{\gamma_\mathrm{N}} \;\cdot $ & $\jump{\gamma_\mathrm{D}}\; \cdot $ & $\jump{\mathrm{n}}\; \cdot $ \\
\noalign{\smallskip}\svhline\noalign{\smallskip}
$\pslnot\gamma$ & $-\gamma$ & $0$ & $0$ \\[0.5ex]
$\pdlnot\beta$ & $0$ & $\beta$ & $0$ \\[0.5ex]
$\mathrm{d}\pslnot\varphi$ & $0$ & $0$ & $-\varphi$ \\[0.5ex]
\noalign{\smallskip}\hline\noalign{\smallskip}
\end{tabular}
\end{center}
\end{lemma}
\begin{remark}
\begin{enumerate}
\item[]
\item The mapping properties \eqref{defpsislsob2} and \eqref{defpsidlsob2} of the layer potentials together with those of the trace operators \eqref{defboundary} ensure that all combinations displayed in the table are well defined.
\item The top left 2x2 block of the table reveals that these jump relations coincide with those of the standard single and double layer potentials in the scalar case.
\item In the sequel we call $\mathrm{d}\,\pslnot$ the exact potential, to distinguish it from single- and double-layer potentials.
\end{enumerate}\end{remark}
\begin{proof} The nine proofs are given below:
\subparagraph{Dirichlet and normal traces of the single layer potential}
Obvious, due to the $H^1$-regularity of the potential, see \eqref{defpsislsob2}.
\subparagraph{Neumann trace of the single layer potential}
Write Green's second identity \eqref{greenid2} for the domains $\Omega$ and $\Omega^\mathrm{c}$ separately. Recall the extension of the $L^2$ inner product $\ltwo{\cdot}{\cdot}{p}{\Gamma}$ to the sesquilinear form $b(\cdot,\cdot)$ according to \eqref{partialint}. Be aware that all boundary terms carry an additional minus sign in the case of $\Omega^\mathrm{c}$, because the boundary $\Gamma$ was chosen to be consistently oriented with $\Omega$. Assume $\boldsymbol{\eta}=\boldsymbol{\Phi}\in\mathcal{D}\pcomp{p}(M)$, which renders all traces of $\boldsymbol{\Phi}$ continuous. Adding both equations yields
\begin{eqnarray}
\lefteqn{\ltwo{(\delta\mathrm{d}-k^2)\boldsymbol{\omega}}{\boldsymbol{\Phi}}{p}{\Omega\cup\Omega^\mathrm{c}}-
\ltwo{\boldsymbol{\omega}}{(\delta\mathrm{d}-\overline{k}^2)\boldsymbol{\Phi}}{p}{\Omega\cup\Omega^\mathrm{c}}}\nonumber\\
&\qquad=&-\overline{b(\gamma_{\mathrm{N}}\boldsymbol{\Phi},\jump{\gamma_{\mathrm{D}}}\boldsymbol{\omega})}
+b(\jump{\gamma_{\mathrm{N}}}\boldsymbol{\omega},\gamma_{\mathrm{D}}\boldsymbol{\Phi}).
\label{greenid2jump}
\end{eqnarray}
In a similar way we obtain from Green's second identity \eqref{greenid3}
\begin{equation}\label{greenid3jump}
\ltwo{\mathrm{d}\delta\boldsymbol{\omega}}{\boldsymbol{\Phi}}{p}{\Omega\cup\Omega^\mathrm{c}}-
\ltwo{\boldsymbol{\omega}}{\mathrm{d}\delta\boldsymbol{\Phi}}{p}{\Omega\cup\Omega^\mathrm{c}}=-\overline{b(\mathrm{n}\boldsymbol{\Phi},\jump{\mathrm{t}}\delta\boldsymbol{\omega})}
+b(\jump{\mathrm{n}}\boldsymbol{\omega},\mathrm{t}\delta\boldsymbol{\Phi}).
\end{equation}
Add up \eqref{greenid2jump} and \eqref{greenid3jump},
\begin{eqnarray*}
\lefteqn{\ltwo{(\Delta-k^2)\boldsymbol{\omega}}{\boldsymbol{\Phi}}{p}{\Omega\cup\Omega^\mathrm{c}}-
\ltwo{\boldsymbol{\omega}}{(\Delta-\overline{k}^2)\boldsymbol{\Phi}}{p}{\Omega\cup\Omega^\mathrm{c}}}\nonumber\\
&\qquad=&-\overline{b(\gamma_{\mathrm{N}}\boldsymbol{\Phi},\jump{\gamma_{\mathrm{D}}}\boldsymbol{\omega})}
+b(\jump{\gamma_{\mathrm{N}}}\boldsymbol{\omega},\gamma_{\mathrm{D}}\boldsymbol{\Phi})
-\overline{b(\mathrm{n}\boldsymbol{\Phi},\jump{\mathrm{t}}\delta\boldsymbol{\omega})}
+b(\jump{\mathrm{n}}\boldsymbol{\omega},\mathrm{t}\delta\boldsymbol{\Phi}).
\end{eqnarray*}
There are no contributions from the far boundary $\Gamma^R$, since $\Phi$ is compactly supported. Pick $\boldsymbol{\omega}=\pslnot\bdryform$, $\bdryform\in H_{\parallel}^{-1/2}\pcomp{p}(\hhat{\delta},\Gamma)$. We already know that $\jump{\gamma_{\mathrm{D}} }\boldsymbol{\omega}=0$ and $\jump{\mathrm{n} }\boldsymbol{\omega}=0$. From \eqref{lpotprop1} in connection with \eqref{defpsislsob1} it follows that $\delta\boldsymbol{\omega}\in H^1_{\mathrm{loc}}\pcomp{p-1}(M)$. The $H^1$-regularity implies $\jump{\mathrm{t} }\delta\boldsymbol{\omega}=0$. Eventually, \eqref{lpotprop2} eliminates the first domain integral. We are left with
\begin{equation}\label{jumpaux1}
-\ltwo{\boldsymbol{\omega}}{(\Delta-\overline{k}^2)\boldsymbol{\Phi}}{p}{\Omega\cup\Omega^\mathrm{c}}=b(\jump{\gamma_{\mathrm{N}} }\boldsymbol{\omega},\gamma_{\mathrm{D}}\boldsymbol{\Phi}).
\end{equation}
Check that
\[
\ltwo{(\Delta-\overline{k}^2)\boldsymbol{\Phi}}{\overline{\boldsymbol{G}}_p}{p}{\Omega\cup\Omega^\mathrm{c}}=
\ltwo{\boldsymbol{\Phi}}{\overline{(\Delta-\overline{k}^2)\boldsymbol{G}_p}}{p}{\Omega\cup\Omega^\mathrm{c}}=\widetilde{\boldsymbol{\Phi}}.
\]
The first equation is justified by the self adjointness of the Laplace-Beltrami operator, the second by \eqref{greenprop1}. Note that $\widetilde{\boldsymbol{\Phi}}$ is not defined on the boundary $\partial\Omega$, but can be continuously extended to yield $\boldsymbol{\Phi}$. Then
\begin{align*}
\bigltwo{\pslnot\bdryform}{(\Delta-\overline{k}^2)\boldsymbol{\Phi}}{p}{\Omega\cup\Omega^\mathrm{c}}
&=\bigltwo{b(\bdryform,\mathrm{t}\boldsymbol{G}_p)}
{(\Delta-\overline{k}^2)\boldsymbol{\Phi}}{p}{\Omega\cup\Omega^\mathrm{c}}^{\prime}\\
&=b\bigl(\bdryform,\mathrm{t}\ltwo{(\Delta-\overline{k}^2)\boldsymbol{\Phi}}{\overline{\boldsymbol{G}}_p}{p}{\Omega\cup\Omega^\mathrm{c}}^{\prime}\bigr)
\\
&=b(\bdryform,\gamma_{\mathrm{D}}\boldsymbol{\Phi}),
\end{align*}
and with \eqref{jumpaux1}
\[
b(\jump{\gamma_{\mathrm{N}}}\boldsymbol{\omega}+\bdryform,\gamma_{\mathrm{D}}\boldsymbol{\Phi})=0.
\]
Density of $\mathrm{t}\mathcal{D}\pcomp{p}(M)$ in $H_{\parallel}^{-1/2}\pcomp{p}(\hhat{\delta},\Gamma)$ completes the proof,
\[
\jump{\gamma_{\mathrm{N}} }\pslnot\bdryform=-\bdryform.
\]
\subparagraph{Dirichlet trace of the double layer potential}
This jump can be related to the Neumann trace of the single layer potential. Consider
\[
\gamma_{\mathrm{D}}*\mathrm{d}=\mathrm{t}*\mathrm{d}=*\,\mathrm{n}\,\mathrm{d}=\hhat{*}\gamma_{\mathrm{N}},
\]
therefore
\begin{equation}\label{jumpaux4}
\gamma_{\mathrm{D}}\pdlnot\bdryform=-\hhat{*}\gamma_{\mathrm{N}}\pslnot(\hhat{*}^{-1}\bdryform),
\end{equation}
which is well defined for $\bdryform\in H_{\bot}^{-1/2}\pcomp{p}(\hhat{\mathrm{d}},\Gamma)$. The jump relation for the Neumann trace of the single layer potential gives us immediately the desired result,
\[
\jump{\gamma_{\mathrm{D}} }\pdlnot\bdryform=\bdryform.
\]
\subparagraph{Neumann trace of the double layer potential}
Consider
\begin{align}
\gamma_{\mathrm{N}}\pdlnot\bdryform
&=\hhat{*}^{-1}\mathrm{t}*\mathrm{d}\,\pdlnot\bdryform\nonumber\\
&=\hhat{*}^{-1}\mathrm{t}\bigl((-1)^p\mathrm{d}\,\pslnot(\hhat{*}\,\hhat{\mathrm{d}}\bdryform)
+k^2\pslnot(\hhat{*}\bdryform)\bigr)\nonumber\\
&=-\hhat{\delta}\,\hhat{*}^{-1}\,\mathrm{t}\,\pslnot(\hhat{*}\,\hhat{\mathrm{d}}\bdryform)+k^2\hhat{*}^{-1}\mathrm{t}\,\pslnot(\hhat{*}\bdryform),\label{jumpaux2}
\end{align}
where \eqref{lpotprop7} has been used. For $\bdryform\in H_{\bot}^{-1/2}\pcomp{p}(\hhat{\mathrm{d}},\Gamma)$, we know that the tangential traces of the above single layer potentials are continuous. Thus we have shown
\[
\jump{\gamma_{\mathrm{N}} }\pdlnot\bdryform=0.
\]
\subparagraph{Dirichlet trace of the exact potential}
Because of $\mathrm{t}\,\mathrm{d}=\hhat{\mathrm{d}}\,\mathrm{t}$ the jump of the Dirichlet trace of the exact potential is related to the jump of the Dirichlet trace of the single layer potential, which yields
\[
\jump{\gamma_\mathrm{D} }\mathrm{d}\pslnot\bdryform=0
\]
for $\bdryform\in H_{\parallel}^{-1/2}\pcomp{p}(\hhat{\delta},\Gamma)$.
\subparagraph{Normal trace of the double layer potential}
Equation \eqref{lpotprop3} shows that the normal trace of the double layer potential is related to the Dirichlet trace of the exact potential as follows
\begin{align*}
\mathrm{n}\pdlnot\bdryform
&=\hhat{*}^{-1}\mathrm{t}\,*\pdlnot\bdryform\nonumber\\
&=(-1)^{p+1}\,\hhat{*}^{-1}\mathrm{t}\,\mathrm{d}\,\pslnot(\hhat{*}\bdryform).
\end{align*}
We therefore conclude
\[
\jump{\mathrm{n} }\pdlnot\bdryform=0
\]
for $\bdryform\in H_{\bot}^{-1/2}\pcomp{p}(\hhat{\mathrm{d}},\Gamma)$.
\subparagraph{Neumann trace of the exact potential}
Trivial, since this trace vanishes.
\subparagraph{Normal trace of the exact potential}
The normal trace of the exact potential coincides with the Neumann trace of the single layer potential. Therefore
\[
\jump{\mathrm{n} }\,\mathrm{d}\pslnot\bdryform=-\bdryform,
\]
for $\bdryform\in H_{\parallel}^{-1/2}\pcomp{p}(\hhat{\delta},\Gamma)$. This completes the proof of Lemma~\ref{lemma:lemmajump}.
\qed\end{proof}
\section{Boundary Integral Operators}\label{sec:bdryintop}
In this section, we introduce boundary-integral operators, and discuss their properties with respect to the sesquilinear form \eqref{partialint}.\par
\marginpar{Boundary integral operators}We denote the average of some trace $\gamma_\mathrm{X}$ of a form $\omega$ across $\Gamma$ by
\[
\mean{\gamma_\mathrm{X}}\boldsymbol{\omega}={\scriptstyle\frac{1}{2}}(\gamma_\mathrm{X}^\mathrm{c}\boldsymbol{\omega}+\gamma_\mathrm{X}\boldsymbol{\omega}),
\]
and define the following boundary integral operators
\begin{subequations}\label{defbio}
\begin{empheq}[box=\fbox]{alignat=4}
&V&&=\gamma_\mathrm{D}\pslnot:&&\quad H_{\parallel}^{-1/2}\pcomp{p}(\hhat{\delta},\Gamma)&&\rightarrow H_{\bot}^{-1/2}\pcomp{p}(\hhat{\mathrm{d}},\Gamma),\label{defbio1}\\
&\tilde{V}&&=\gamma_\mathrm{D}\tpsl{}:&&\quad H_{\parallel}^{-1/2}\pcomp{p}(\hhat{\delta},\Gamma)&&\rightarrow H_{\bot}^{-1/2}\pcomp{p}(\hhat{\mathrm{d}},\Gamma),\quad k\ne 0,\label{defbio1a}\\
&K^\dagger&&=\mean{\gamma_\mathrm{N}}\pslnot:&&\quad H_{\parallel}^{-1/2}\pcomp{p}(\hhat{\delta},\Gamma)&&\rightarrow H_{\parallel}^{-1/2}\pcomp{p}(\hhat{\delta},\Gamma),\label{defbio2}\\
&K&&=\mean{\gamma_\mathrm{D}}\pdlnot:&&\quad H_{\bot}^{-1/2}\pcomp{p}(\hhat{\mathrm{d}},\Gamma)&&\rightarrow H_{\bot}^{-1/2}\pcomp{p}(\hhat{\mathrm{d}},\Gamma),\label{defbio3}\\
&D&&=-\gamma_\mathrm{N}\pdlnot:&&\quad H_{\bot}^{-1/2}\pcomp{p}(\hhat{\mathrm{d}},\Gamma)&&\rightarrow H_{\parallel}^{-1/2}\pcomp{p}(\hhat{\delta},\Gamma),\label{defbio4}\\
&W&&=\mathrm{n}\pslnot:&&\quad H_{\parallel}^{-1/2}\pcomp{p}(\hhat{\delta},\Gamma)&&\rightarrow H_{\parallel}^{-1/2}\pcomp{p-1}(\hhat{\delta},\Gamma),\label{defbio5}
\end{empheq}
\end{subequations}
where $V$ denotes the single layer operator, $\tilde{V}$ the Maxwell single layer operator, $K$ the double layer operator, $K^\dagger$ the conjugate adjoint double layer operator, and $D$ the hypersingular operator. Inspecting the mapping properties of the trace operators and of the layer potentials shows that the boundary integral operators are well defined and continuous.
\pagebreak % <--- !!!
\begin{lemma}\label{lemma:bioprop}The boundary integral operators have the following properties:
\begin{subequations}\label{bioprop}
\begin{alignat}{3}
\text{(i)}\quad&&\hhat{*}K^\dagger&=-K\hhat{*},\label{bioprop2}\\
\text{(ii)}\quad&&\hhat{\delta}K^\dagger-K^\dagger\hhat{\delta}&=-k^2W,\label{bioprop3}\\
\text{(iii)}\quad&&\hhat{\mathrm{d}}\,K-K\hhat{\mathrm{d}}&=(-1)^{p+1}k^2\hhat{*}^{-1}W\hhat{*},\label{bioprop4}\\
\text{(iv)}\quad&&D&=\hhat{*}^{-1}\hhat{\mathrm{d}}\,V\hhat{\delta}\hhat{*}-k^2\hhat{*}^{-1}V\hhat{*}.\label{bioprop5a}
\intertext{For $k\ne 0$ it also holds that}
\text{(v)}\quad&&\tilde{V}&=V-\frac{1}{k^2}\hhat{\mathrm{d}}\,V\hhat{\delta},\label{bioprop1}\\
\text{(vi)}\quad&&D&=-k^2\hhat{*}^{-1}\tilde{V}\hhat{*}.\label{bioprop5b}
\end{alignat}
\end{subequations}
\end{lemma}
\begin{remark}
\begin{enumerate}
\item[]
\item Properties (ii) and (iii) show that the double-layer operator and its adjoint respect the respective de Rham sequence for $k=0$.
\item Properties (iv) and (vi) connect the hypersingular operator $D$ to the weakly singular operator $V$, by using integration by parts. For the Laplace case this goes back to \textsc{Maue} \cite[p.\ 604]{maue1949}, and there is a paper by \textsc{N\'ed\'elec} \cite{nedelec1982} on related representations for different second order partial differential equations. This provides a way of avoiding the evaluation of hypersingular integrals in Galerkin boundary element methods. For example, from (iv) we find
\[
b(D\boldsymbol{\omega},\boldsymbol{\eta})=\overline{b(\hhat{\delta}\hhat{*}\boldsymbol{\eta},V\hhat{\delta}\hhat{*}\boldsymbol{\omega})}-k^2\overline{b(\hhat{*}\boldsymbol{\eta},V\hhat{*}\boldsymbol{\omega})},
\]
$\boldsymbol{\omega},\boldsymbol{\eta}\in H_{\bot}^{-1/2}\pcomp{p}(\hhat{\mathrm{d}},\Gamma)$.
\end{enumerate}
\end{remark}
\begin{proof}
Property (i) is an immediate consequence of \eqref{jumpaux4}.
Property (ii) can be seen as follows:
\begin{align*}
\hhat{\delta}K^\dagger
&=\hhat{\delta}\mathrm{n}\mathrm{d}\,\pslnot
=-\mathrm{n}\delta\mathrm{d}\pslnot\\
&=-\mathrm{n}(k^2-\mathrm{d}\delta)\pslnot
=\mathrm{n}\mathrm{d}\pslnot\hhat{\delta}-k^2\mathrm{n}\pslnot\\
&=K^\dagger\hhat{\delta}-k^2W,
\end{align*}
where we used \eqref{lpotprop1}, \eqref{lpotprop2}, \eqref{commutendelta}, \eqref{defbio2}, and \eqref{defbio5}. Equation (iii) can be seen as follows:
\begin{align*}
\hhat{\mathrm{d}}\,K
&=-\hhat{\mathrm{d}}\,\hhat{*}K^\dagger\hhat{*}^{-1}
=(-1)^q\hhat{*}\hhat{\delta}K^\dagger\hhat{*}^{-1}\\
&=(-1)^q\hhat{*}(K^\dagger\hhat{\delta}-k^2W)\hhat{*}^{-1}
=-\hhat{*}K^\dagger\hhat{*}^{-1}\hhat{\mathrm{d}}-(-1)^qk^2\hhat{*}W\hhat{*}^{-1}\\
&=K\hhat{\mathrm{d}}-(-1)^pk^2\hhat{*}^{-1}W\hhat{*},
\end{align*}
where we used \eqref{defcoderivative}, \eqref{bioprop2}, and \eqref{bioprop3}. Eventually, (iv) follows from \eqref{jumpaux2} with \eqref{defcoderivative}, and \eqref{defbio4}. Properties (v) and (vi) follow from \eqref{defmodslpot} with \eqref{commutetd}.
\qed\end{proof}
\begin{framed}
The ordinary scalar and vectorial boundary integral operators in three dimensions follow from \eqref{defbio} by means of the translation isomorphisms. For the vectorial operators we adopt the conventions of \cite[Thm.\ 9]{hiptmair2007}.\par\medskip
\renewcommand{\arraystretch}{1.5}
{\centering\begin{tabular}{|l|lll|rll|}\hline
\eqref{defbio} & \multicolumn{3}{|c}{$p=0$, scalar operators} & \multicolumn{3}{|c|}{$p=1$, vectorial operators}\\ \hline
\,$V$ & \,$\mathsf{V}$: & $H^{-\frac{1}{2}\!}(\partial\Omega)$ & $\rightarrow H^{\frac{1}{2}\!}(\partial\Omega)$
& $\mathsf{A}$: & $\mathbf{H}^{-\frac{1}{2}\!}(\mathrm{div}_\Gamma,\partial\Omega)$ & $\rightarrow\mathbf{H}^{-\frac{1}{2}\!}(\mathrm{curl}_\Gamma,\partial\Omega)$\\
\,$K^\dagger$ & \,$\mathsf{K}^\dagger$: & $H^{-\frac{1}{2}\!}(\partial\Omega)$ & $\rightarrow H^{-\frac{1}{2}\!}(\partial\Omega)$
& $\mathsf{B}$: & $\mathbf{H}^{-\frac{1}{2}\!}(\mathrm{div}_\Gamma,\partial\Omega)$ & $\rightarrow\mathbf{H}^{-\frac{1}{2}\!}(\mathrm{div}_\Gamma,\partial\Omega)$\\
\,$K$ & \,$\mathsf{K}$: & $H^{\frac{1}{2}\!}(\partial\Omega)$ & $\rightarrow H^{\frac{1}{2}\!}(\partial\Omega)$
& $-\mathsf{C}$: & $\mathbf{H}^{-\frac{1}{2}\!}(\mathrm{curl}_\Gamma,\partial\Omega)$ & $\rightarrow\mathbf{H}^{-\frac{1}{2}\!}(\mathrm{curl}_\Gamma,\partial\Omega)$\\
\,$D$ & \,$\mathsf{D}$: & $H^{\frac{1}{2}\!}(\partial\Omega)$ & $\rightarrow H^{-\frac{1}{2}\!}(\partial\Omega)$
& $\mathsf{N}$: & $\mathbf{H}^{-\frac{1}{2}\!}(\mathrm{curl}_\Gamma,\partial\Omega)$ & $\rightarrow\mathbf{H}^{-\frac{1}{2}\!}(\mathrm{div}_\Gamma,\partial\Omega)$\\ \hline
\end{tabular}}\par\bigskip\noindent
Equation \eqref{bioprop5a} yields for $p=0$ (compare \cite[Thm.\ 9.15]{mclean})
\[
\mathsf{D}\boldsymbol{\beta}=\mathrm{curl}_\Gamma \,\mathsf{A}\,\textbf{curl}_\Gamma\boldsymbol{\beta}-k^2\int_{\partial\Omega}\boldsymbol{\beta}(\ppoint)g_3(\ppoint,\ppoint^\prime)\vec{n}(\ppoint)\cdot\vec{n}(\ppoint^\prime)\,\mathrm{d}\Gamma(\ppoint),
\]
where the integral extends to $\boldsymbol{\beta}\in H^{1/2}(\partial\Omega)$. For $p=1$, the same equation yields
\[
\mathsf{N}=\textbf{curl}_\Gamma\,\mathsf{V}\,\mathrm{curl}_\Gamma+k^2\mathsf{R}\,\mathsf{A}\,\mathsf{R},
\]
where the rotation operator $\mathsf{R}$ was defined in \eqref{defrotation}. This generalizes Lemma 6.3 in \cite{hiptmair2003}\footnotemark.\par
We state some useful relations that stem from (i) -- (iii) for $k=0$:
\begin{alignat*}{3}
\text{(i)};p=1:\quad&&\vec{n}\times\mathsf{B}&=\mathsf{C}\,\vec{n}\times\cdot,\\
\text{(i), (ii)};p=1:\quad&&\mathrm{curl}_\Gamma\mathsf{C}&=\mathsf{K}^\dagger\mathrm{curl}_\Gamma,\\
\text{(iii)};p=0:\quad&&\mathrm{grad}_\Gamma\mathsf{K}&=-\mathsf{C}\,\mathrm{grad}_\Gamma,\\
\text{(i), (iii)};p=0:\quad&&\textbf{curl}_\Gamma\mathsf{K}&=-\mathsf{B}\,\textbf{curl}_\Gamma.
\end{alignat*}
\end{framed}
\footnotetext{Be aware that the operators $\mathsf{B}$, $\mathsf{C}$, and $\mathsf{N}$ are defined slightly differently in \cite{hiptmair2003} compared to \cite{hiptmair2007}. In particular, $\mathsf{N}$ bears an additional minus sign.}
%
%%%666
\subsection{Symmetry Properties}\label{sec:symmsec}
\begin{lemma}\label{lemma:symmetrylemma}\marginpar{Symmetry properties}If $k\ne0$, the boundary integral operators exhibit the following symmetry properties with respect to the sesquilinear form $b(\cdot,\cdot)$:
\begin{subequations}\label{biosym}
\begin{alignat}{3}
\text{(i)}\quad&&\overline{b(\boldsymbol{\gamma},V\boldsymbol{\gamma}^{\,\prime})}&=b(\boldsymbol{\gamma}^{\,\prime},\overline{V}\boldsymbol{\gamma}),\label{biosym1}\\
\text{(ii)}\quad&&\overline{b(\boldsymbol{\gamma},\tilde{V}\boldsymbol{\gamma}^{\,\prime)}}&=b(\boldsymbol{\gamma}^{\,\prime},\overline{\tilde{V}}\boldsymbol{\gamma}),\label{biosym1a}\\
\text{(iii)}\quad&&\overline{b(D\boldsymbol{\beta}^\prime,\boldsymbol{\beta})}&=b(\overline{D}\boldsymbol{\beta},\boldsymbol{\beta}^\prime),\label{biosym2}\\
\text{(iv)}\quad&&b(K^\dagger\boldsymbol{\gamma},\boldsymbol{\beta})&=b(\boldsymbol{\gamma},\overline{K}\boldsymbol{\beta}),\label{biosym3}
\end{alignat}
\end{subequations}
for all $\boldsymbol{\beta},\boldsymbol{\beta}^\prime\in H_{\bot}^{-1/2}\pcomp{p}(\hhat{\mathrm{d}},\Gamma)$ and $\boldsymbol{\gamma},\boldsymbol{\gamma}^{\,\prime}\in H_{\parallel}^{-1/2}\pcomp{p}(\hhat{\delta},\Gamma)$.
The single layer operators $V$, $\tilde{V}$ and the hypersingular operator $D$ are complex symmetric, while $K^\dagger$ is the conjugate adjoint of $K$.\,\footnote{Compare with \cite[Thm.\ 10]{hiptmair2007}. Note that {\sc Hiptmair} defines the rotation operator as $\mathsf{R}=-\vec{n}\times\cdot$ rather than $\mathsf{R}=\vec{n}\times\cdot$, which yields an additional minus sign in the definition of the double layer potential and consequently in {\sc Hiptmair}'s Thm.\ 10.}\par
Relationships (i), (iii), and (iv) hold in the case $k=0$ as well, provided $\boldsymbol{\gamma},\boldsymbol{\gamma}^{\,\prime}$ are restricted to the space $H_{\parallel}^{-1/2}\pcomp{p}(\hhat{\delta}0,\Gamma)$. The conjugations are just irrelevant, so that $V$ and $D$ are symmetric, and $K^\dagger$ is the adjoint of $K$.
\end{lemma}
To prove Lemma~\ref{lemma:symmetrylemma}, we need the following corollary.
\begin{corollary}
We exclude the case $k=0$, $n=2$, $p=0$. In all other cases, for $\boldsymbol{\omega},\boldsymbol{\eta}\in X^p(\Omega\cup\Omega^\mathrm{c})$, we can use the following equality:
\begin{equation}\label{auxbdrjump}
b(\jump{\gamma_\mathrm{N}}\boldsymbol{\omega},\mean{\gamma_\mathrm{D}}\overline{\boldsymbol{\eta}})
+b(\mean{\gamma_\mathrm{N}}\boldsymbol{\omega}, \jump{\gamma_\mathrm{D} }\overline{\boldsymbol{\eta}})
=b(\jump{\gamma_\mathrm{N} }\boldsymbol{\eta},\mean{\gamma_\mathrm{D}}\overline{\boldsymbol{\omega}})
+b(\mean{\gamma_\mathrm{N}}\boldsymbol{\eta}, \jump{\gamma_\mathrm{D} }\overline{\boldsymbol{\omega}}).
\end{equation}
\end{corollary}
\begin{proof}
Substituting $\boldsymbol{\omega}$ by $\mathrm{d}\boldsymbol{\omega}$, $\boldsymbol{\omega}\in \hzero \pcomp{p}(\delta\mathrm{d},\Omega)$, in \eqref{partialint} yields
\begin{equation*}
b(\gamma_\mathrm{N}\boldsymbol{\omega},\gamma_\mathrm{D}\boldsymbol{\eta})=
\ltwo{\mathrm{d}\boldsymbol{\omega}}{\mathrm{d}\boldsymbol{\eta}}{p+1}{\Omega}-
\ltwo{\delta\mathrm{d}\boldsymbol{\omega}}{\boldsymbol{\eta}}{p}{\Omega}.
\end{equation*}
We may assume that $\boldsymbol{\omega}\in X^p(\Omega)$ is a Maxwell solution, to obtain
\begin{equation}\label{boundaux1}
b(\gamma_\mathrm{N}\boldsymbol{\omega},\gamma_\mathrm{D}\boldsymbol{\eta})=
\ltwo{\mathrm{d}\boldsymbol{\omega}}{\mathrm{d}\boldsymbol{\eta}}{p+1}{\Omega}-
k^2\ltwo{\boldsymbol{\omega}}{\boldsymbol{\eta}}{p}{\Omega}.
\end{equation}
The same arguments can be applied to the exterior domain $\Omega^\mathrm{c}\cap\Omega^R$, where $\boldsymbol{\omega}\in X^p(\Omega^\mathrm{c}\cap\Omega^R)$, $\eta\in \hzero \pcomp{p}(\mathrm{d},\Omega^\mathrm{c}\cap\Omega^R)$, from which we conclude
\begin{equation}\label{boundaux4}
-b(\gamma_\mathrm{N}^\mathrm{c}\boldsymbol{\omega},\gamma_\mathrm{D}^\mathrm{c}\boldsymbol{\eta})
+\ltwo{\gamma_\mathrm{N}^R\boldsymbol{\omega}}{\gamma_\mathrm{D}^R\boldsymbol{\eta}}{p}{\Gamma^R}
=\ltwo{\mathrm{d}\boldsymbol{\omega}}{\mathrm{d}\boldsymbol{\eta}}{p+1}{\Omega^\mathrm{c}\cap\Omega^R}-
k^2\ltwo{\boldsymbol{\omega}}{\boldsymbol{\eta}}{p}{\Omega^\mathrm{c}\cap\Omega^R}.
\end{equation}
Note that on the far boundary $\Gamma^R$ we encounter sufficient regularity so to work with the $L^2$ inner product rather than the sesquilinear form $b(\cdot,\cdot)$. By combining \eqref{boundaux1} and \eqref{boundaux4} we find for $\boldsymbol{\omega}\in X^p\bigl((\Omega\cup\Omega^\mathrm{c})\cap\Omega^R)\bigr)$
\begin{gather}
b(\gamma_\mathrm{N}\boldsymbol{\omega},\gamma_\mathrm{D}\boldsymbol{\eta})
-b(\gamma_\mathrm{N}^\mathrm{c}\boldsymbol{\omega},\gamma_\mathrm{D}^\mathrm{c}\boldsymbol{\eta})
+\ltwo{\gamma_\mathrm{N}^R\boldsymbol{\omega}}{\gamma_\mathrm{D}^R\boldsymbol{\eta}}{p}{\Gamma^R}\nonumber\\
=\ltwo{\mathrm{d}\boldsymbol{\omega}}{\mathrm{d}\boldsymbol{\eta}}{p+1}{(\Omega\cup\Omega^\mathrm{c})\cap\Omega^R}-
k^2\ltwo{\boldsymbol{\omega}}{\boldsymbol{\eta}}{p}{(\Omega\cup\Omega^\mathrm{c})\cap\Omega^R}.\label{symmaux3}
\end{gather}
Rearranging terms and substituting $\boldsymbol{\eta}$ by $\overline{\boldsymbol{\eta}}$ yields
\begin{gather}
-b(\jump{\gamma_\mathrm{N}}\boldsymbol{\omega},\mean{\gamma_\mathrm{D}}\overline{\boldsymbol{\eta}})
-b(\mean{\gamma_\mathrm{N}}\boldsymbol{\omega}, \jump{\gamma_\mathrm{D} }\overline{\boldsymbol{\eta}})
+\ltwo{\gamma_\mathrm{N}^R\boldsymbol{\omega}}{\gamma_\mathrm{D}^R\overline{\boldsymbol{\eta}}}{p}{\Gamma^R}\nonumber\\
=\ltwo{\mathrm{d}\boldsymbol{\omega}}{\mathrm{d}\overline{\boldsymbol{\eta}}}{p+1}{(\Omega\cup\Omega^\mathrm{c})\cap\Omega^R}-
k^2\ltwo{\boldsymbol{\omega}}{\overline{\boldsymbol{\eta}}}{p}{(\Omega\cup\Omega^\mathrm{c})\cap\Omega^R}.\label{symmaux4}
\end{gather}
We may assume that $\boldsymbol{\eta}\in X^p(\Omega\cup\Omega^\mathrm{c})$ is a Maxwell solution as well, exchange the roles of $\boldsymbol{\omega}$ and $\boldsymbol{\eta}$, and subtract both equations. This eliminates the sesquilinear $L^2$ inner products over the domain. The sesquilinear $L^2$ inner products
\begin{align*}
x&=\ltwo{\gamma_\mathrm{N}^R\boldsymbol{\omega}}{\gamma_\mathrm{D}^R\overline{\boldsymbol{\eta}}}{p}{\Gamma^R}-\ltwo{\gamma_\mathrm{N}^R\boldsymbol{\eta}}{\gamma_\mathrm{D}^R\overline{\boldsymbol{\omega}}}{p}{\Gamma^R}
\end{align*}
over the far boundary $\Gamma^R$ vanish for $R\to\infty$, which we show in the sequel. We need to distinguish three cases:
\begin{enumerate}
\item {\em Case $k\ne 0$.} Rewrite
\begin{align*}
x&=\ltwo{\gamma_\mathrm{N}^R\boldsymbol{\omega}-ik\gamma_\mathrm{D}^R\boldsymbol{\omega}}{\gamma_\mathrm{D}^R\overline{\boldsymbol{\eta}}}{p}{\Gamma^R}-\ltwo{\gamma_\mathrm{N}^R\boldsymbol{\eta}-ik\gamma_\mathrm{D}^R\boldsymbol{\eta}}{\gamma_\mathrm{D}^R\overline{\boldsymbol{\omega}}}{p}{\Gamma^R}.
\end{align*}
From the triangle inequality we obtain
\[
|x|\le\bigl|\ltwo{\gamma_\mathrm{N}^R\boldsymbol{\omega}-ik\gamma_\mathrm{D}^R\boldsymbol{\omega}}{\gamma_\mathrm{D}^R\overline{\boldsymbol{\eta}}}{p}{\Gamma^R}\bigl|+\bigr|\ltwo{\gamma_\mathrm{N}^R\boldsymbol{\eta}-ik\gamma_\mathrm{D}^R\boldsymbol{\eta}}{\gamma_\mathrm{D}^R\overline{\boldsymbol{\omega}}}{p}{\Gamma^R}\bigr|,
\]
and from the Cauchy-Schwarz inequality
\begin{align*}
\bigl|\ltwo{\gamma_\mathrm{N}^R\boldsymbol{\omega}-ik\gamma_\mathrm{D}^R\boldsymbol{\omega}}{\gamma_\mathrm{D}^R\overline{\boldsymbol{\eta}}}{p}{\Gamma^R}\bigr|
&\le
\bigltwonorm{\gamma_\mathrm{N}^R\boldsymbol{\omega}-ik\gamma_\mathrm{D}^R\boldsymbol{\omega}}{p}{\Gamma^R}\bigltwonorm{\gamma_\mathrm{D}^R\overline{\boldsymbol{\eta}}}{p}{\Gamma^R}
=o(1),
\end{align*}
where we used \eqref{condradiation3} and \eqref{boundtraces}. Therefore, $|x|=o(1)$ for $R\to\infty$.
\item {\em Case $k=0$, $n\ge 3$.} From \eqref{conddecay1} and \eqref{conddecay2} we obtain
\begin{equation}\label{decayaux1}
\bigl|\ltwo{\gamma_\mathrm{N}^R\boldsymbol{\omega}}{\gamma_\mathrm{D}^R\overline{\boldsymbol{\eta}}}{p}{\Gamma^R}\bigr|
=\mathcal{O}(R^{2-n}).
\end{equation}
Together with the triangle inequality this yields $|x|=o(1)$ for $R\to\infty$.\par\noindent
\item {\em Case $k=0$, $n=2$, $p=1$.} Taking into account $b=0$, \eqref{conddecay1} and \eqref{conddecay2} yield
\begin{equation}\label{decayaux2}
\bigl|\ltwo{\gamma_\mathrm{N}^R\boldsymbol{\omega}}{\gamma_\mathrm{D}^R\overline{\boldsymbol{\eta}}}{p}{\Gamma^R}\bigr|
=\mathcal{O}(R^{-2}),
\end{equation}
for $R\to\infty$, and we can proceed like in the previous case.
\end{enumerate}
We conclude that for $\boldsymbol{\omega},\boldsymbol{\eta}\in X^p(\Omega\cup\Omega^\mathrm{c})$
\begin{gather*}
b(\jump{\gamma_\mathrm{N}}\boldsymbol{\omega},\mean{\gamma_\mathrm{D}}\overline{\boldsymbol{\eta}})
+b(\mean{\gamma_\mathrm{N}}\boldsymbol{\omega}, \jump{\gamma_\mathrm{D} }\overline{\boldsymbol{\eta}})
=b(\jump{\gamma_\mathrm{N} }\boldsymbol{\eta},\mean{\gamma_\mathrm{D}}\overline{\boldsymbol{\omega}})
+b(\mean{\gamma_\mathrm{N}}\boldsymbol{\eta}, \jump{\gamma_\mathrm{D} }\overline{\boldsymbol{\omega}}).
\end{gather*}
\qed\end{proof}
\begin{proof}[Lemma~\ref{lemma:symmetrylemma}]
The case $k=0$, $n=2$, $p=0$ is covered by the formally self-adjoint case in \cite[eq.\ (7.3), (7.4)]{mclean}. In all other cases we choose $\boldsymbol{\omega},\boldsymbol{\eta}$ in \eqref{auxbdrjump} according to the representation formula \eqref{represent} as
\begin{alignat*}{4}
\boldsymbol{\omega}&=-\psl{p}\,\boldsymbol{\gamma}
&&+\pdl{p}\,\boldsymbol{\beta}
&&-\mathrm{d}\,\psl{p-1}\,\boldsymbol{\varphi},\\
\boldsymbol{\eta}&=-\psl{p}\,\boldsymbol{\gamma}^{\,\prime}
&&+\pdl{p}\,\boldsymbol{\beta}^\prime
&&-\mathrm{d}\,\psl{p-1}\,\boldsymbol{\varphi}^\prime.
\end{alignat*}
To prove \eqref{biosym1}, we pick $\boldsymbol{\beta}=\boldsymbol{\beta}^\prime=0$. With Lemma~\ref{lemma:lemmajump}, \eqref{neumannsolution}, and definition \eqref{defbio1} we obtain
\begin{gather*}
b(\boldsymbol{\gamma},\overline{V\boldsymbol{\gamma}^{\,\prime}+\hhat{\mathrm{d}}V\boldsymbol{\varphi}^\prime})=b(\boldsymbol{\gamma}^{\,\prime},\overline{V\boldsymbol{\gamma}+\hhat{\mathrm{d}}V\boldsymbol{\varphi}}),\\
b(\boldsymbol{\gamma},\overline{V\boldsymbol{\gamma}^{\,\prime}})+
b(\hhat{\delta}\boldsymbol{\gamma},\overline{V\boldsymbol{\varphi}^\prime})
=
b(\boldsymbol{\gamma}^{\,\prime},\overline{V\boldsymbol{\gamma}})+
b(\hhat{\delta}\boldsymbol{\gamma}^{\,\prime},\overline{V\boldsymbol{\varphi}}),\\
b(\boldsymbol{\gamma},\overline{V\boldsymbol{\gamma}^{\,\prime}})
-k^2b(\boldsymbol{\varphi},\overline{V\boldsymbol{\varphi}^\prime})
=
b(\boldsymbol{\gamma}^{\,\prime},\overline{V\boldsymbol{\gamma}})
-k^2b(\boldsymbol{\varphi}^\prime,\overline{V\boldsymbol{\varphi}}),
\end{gather*}
which proves \eqref{biosym1} for $k=0$. Otherwise restrict $\boldsymbol{\gamma},\boldsymbol{\gamma}^{\prime}$ to $H_{\parallel}^{-1/2}\pcomp{p}(\hhat{\delta}0,\Gamma)$. This eliminates the second and forth terms, since \eqref{neumannsolution} then implies $\boldsymbol{\varphi}=\boldsymbol{\varphi}^\prime=0$. From the remaining terms we can conclude that \eqref{biosym1} holds at least on $H_{\parallel}^{-1/2}\pcomp{p}(\hhat{\delta}0,\Gamma)$. From this we infer with $\boldsymbol{\varphi},\boldsymbol{\varphi}^\prime\in H_{\parallel}^{-1/2}\pcomp{p-1}(\hhat{\delta}0,\Gamma)$ that the second and forth terms always cancel, which means that the first and third terms imply \eqref{biosym1} even for $\boldsymbol{\gamma},\boldsymbol{\gamma}^{\,\prime}\in H_{\parallel}^{-1/2}\pcomp{p}(\hhat{\delta},\Gamma)$.\par
Property \eqref{biosym2} is seen setting $\boldsymbol{\gamma}=\boldsymbol{\gamma}^{\,\prime}=0$, $\boldsymbol{\varphi}=\boldsymbol{\varphi}^\prime=0$. With Lemma~\ref{lemma:lemmajump} and definition \eqref{defbio4} we obtain
\[
b(D\boldsymbol{\beta},\overline{\boldsymbol{\beta}^\prime})=b(D\boldsymbol{\beta}^\prime,\overline{\boldsymbol{\beta}}),
\]
which proves \eqref{biosym2}.\par
For \eqref{biosym3}, we pick $\boldsymbol{\beta}=0$, $\boldsymbol{\gamma}^{\,\prime}=0$, $\boldsymbol{\varphi}^\prime=0$. With Lemma~\ref{lemma:lemmajump} and the definitions \eqref{defbio2}, \eqref{defbio3} we obtain
\[
b(\boldsymbol{\gamma},\overline{K\boldsymbol{\beta}^\prime})-b(K^\dagger\boldsymbol{\gamma},\overline{\boldsymbol{\beta}^\prime})=0,
\]
which proves \eqref{biosym3}.\par
Finally, we have
\begin{align*}
\overline{b(\boldsymbol{\gamma},{\textstyle\frac{1}{k^2}}\hhat{\mathrm{d}}V\hhat{\delta}\boldsymbol{\gamma}^{\,\prime})}
&={\textstyle\frac{1}{k^2}}\overline{b(\hhat{\delta}\boldsymbol{\gamma},V\hhat{\delta}\boldsymbol{\gamma}^{\,\prime})}\\
&={\textstyle\frac{1}{k^2}}b(\hhat{\delta}\boldsymbol{\gamma}^{\,\prime},\overline{V}\hhat{\delta}\boldsymbol{\gamma})
=b(\boldsymbol{\gamma}^{\,\prime},\overline{{\textstyle\frac{1}{k^2}}\hhat{\mathrm{d}}V\hhat{\delta}}\boldsymbol{\gamma}),
\end{align*}
where we used \eqref{biosym1}. Together with \eqref{bioprop1} this proves \eqref{biosym1a}.
\qed\end{proof}
\subsection{Ellipticity Properties}
\begin{lemma}\marginpar{Ellipticity properties}
The boundary integral operators in connection with the sesquilinear form $b(\cdot,\cdot)$ exhibit the following ellipticity properties for $\Im\text{m}\,k>0$:
\begin{subequations}\label{bioellip}
\begin{alignat}{4}
&|b(\boldsymbol{\gamma},V\boldsymbol{\gamma})|&&\ge c\,\frac{\Im\text{m}\,k}{|k|}\min\bigl(1,\frac{1}{|k|^2}\bigr)
&&|\boldsymbol{\gamma}|_{H_{\parallel}^{-1/2}\pcomp{p}(\hhat{\delta},\Gamma)}^2\quad&&\forall\boldsymbol{\gamma}\in H_{\parallel}^{-1/2}\pcomp{p}(\hhat{\delta}0,\Gamma),\label{bioellip1}\\
&|b(\boldsymbol{\gamma},\tilde{V}\boldsymbol{\gamma})|&&\ge c\,\frac{\Im\text{m}\,k}{|k|}\min\bigl(1,\frac{1}{|k|^2}\bigr)
&&|\boldsymbol{\gamma}|_{H_{\parallel}^{-1/2}\pcomp{p}(\hhat{\delta},\Gamma)}^2\quad&&\forall\boldsymbol{\gamma}\in H_{\parallel}^{-1/2}\pcomp{p}(\hhat{\delta},\Gamma),\label{bioellip1a}\\
&|b(D\boldsymbol{\beta},\boldsymbol{\beta})|&&\ge c\,\frac{\Im\text{m}\,k}{|k|}\min(1,|k|^2)
&&|\boldsymbol{\beta}|_{H_{\bot}^{-1/2}\pcomp{p}(\hhat{\mathrm{d}},\Gamma)}^2\quad&&\forall\boldsymbol{\beta}\in H_{\bot}^{-1/2}\pcomp{p}(\hhat{\mathrm{d}},\Gamma),\label{bioellip2}
\end{alignat}
with positive generic constants $c$, that depend only on the boundary.\footnote{Estimates for the real part instead of the modulus can be derived as well \cite{hiptmair2007}, under the prerequisite $\Re\text{e}\,k^2\le 0$. When comparing with \cite{hiptmair2007}, $\kappa^2=-k^2$ has to be taken into account.}	
For $k=0$ it holds that
\begin{alignat}{4}
&b(\boldsymbol{\gamma},V\boldsymbol{\gamma})&&\ge c\,
&&|\boldsymbol{\gamma}|_{H_{\parallel}^{-1/2}\pcomp{p}(\hhat{\delta},\Gamma)}^2\quad&&\forall\boldsymbol{\gamma}\in H_{\parallel}^{-1/2}\pcomp{p}(\hhat{\delta}0,\Gamma),\label{bioellip1b}
\end{alignat}
\end{subequations}
while $b(D\boldsymbol{\beta},\boldsymbol{\beta})\ge0$ on $H_{\bot}^{-1/2}\pcomp{p}(\hhat{\mathrm{d}},\Gamma)$. In the case $n=2$, $p=0$ the parameter $r_0$ in the fundamental solution \eqref{defgreen} has to be large enough, see \cite[Thm.\ 8.16]{mclean}.
\end{lemma}
\begin{remark}
\begin{enumerate}
\item[]
\item The estimates \eqref{bioellip} are useful to establish well-posedness of boundary integral equations based on the above operators, by leveraging the Lax-Milgram theorem. This carries over to conforming Galerkin discretizations, and quasi-optimal error estimates in the norms of the trace spaces are available from Cea's lemma \cite{hiptmair2003,hiptmair2007}.
\item For $k\in\mathbb{R}^+$, the estimates \eqref{bioellip} no longer hold true. Still, the well-posedness of boundary integral equations based on the above operators can be established, based on a generalized G{\aa}rding inequality and Hodge decompositions \cite{buffa2005,buffa2002a,buffa2003,christiansen2002}.
\end{enumerate} \end{remark}
\begin{proof}
To show the ellipticity properties \eqref{bioellip}, we need to distinguish three cases:
\begin{enumerate}
\item {\em Case $\Im\text{m}\,k>0$.} We proceed like in the derivation of \eqref{boundaux3}, but consider the domain $(\Omega\cup\Omega^\mathrm{c})\cap\Omega^R$ rather than $\Omega^\mathrm{c}\cap\Omega^R$. We find for $\boldsymbol{\omega}\in X^p\bigl((\Omega\cup\Omega^\mathrm{c})\cap\Omega^R)\bigr)$
\begin{gather*}
\Im\text{m}\bigl(\overline{k}
b(\jump{\gamma_\mathrm{N}}\boldsymbol{\omega},\mean{\gamma_\mathrm{D}}\boldsymbol{\omega})+\overline{k}
b(\mean{\gamma_\mathrm{N}}\boldsymbol{\omega},\jump{\gamma_\mathrm{D}}\boldsymbol{\omega})
\bigr)\\[1mm]
\begin{aligned}
=&
(\Im\text{m}\,k)\bigl(\ltwonorm{\mathrm{d}\boldsymbol{\omega}}{p+1}{(\Omega\cup\Omega^\mathrm{c})\cap\Omega^R}^2+
|k|^2\ltwonorm{\boldsymbol{\omega}}{p}{(\Omega\cup\Omega^\mathrm{c})\cap\Omega^R}^2\bigr)\\[1mm]
&+{\textstyle\frac{1}{2}}\bigl(\ltwonorm{\gamma_\mathrm{N}^R\boldsymbol{\omega}}{p}{\Gamma^R}^2+|k|^2\ltwonorm{\gamma_\mathrm{D}^R\boldsymbol{\omega}}{p}{\Gamma^R}^2\bigr)+o(1).\end{aligned}
\end{gather*}
Pick $\boldsymbol{\omega}=\pdlnot\boldsymbol{\beta}$ and let $R\to\infty$, which yields
\[
-\Im\text{m}\bigl(\overline{k}b(D\boldsymbol{\beta},\boldsymbol{\beta})\bigr)\ge
(\Im\text{m}\,k)\bigl(\ltwonorm{\mathrm{d}\boldsymbol{\omega}}{p+1}{\Omega\cup\Omega^\mathrm{c}}^2+
|k|^2\ltwonorm{\boldsymbol{\omega}}{p}{\Omega\cup\Omega^\mathrm{c}}^2\bigr)\ge0,
\]
where we used Lemma~\ref{lemma:lemmajump} and definition \eqref{defbio4}. Next, consider
\[
|k|^2|b(D\boldsymbol{\beta},\boldsymbol{\beta})|^2=|\overline{k}b(D\boldsymbol{\beta},\boldsymbol{\beta})|^2\ge\Im\text{m}\,\bigl(\overline{k}b(D\boldsymbol{\beta},\boldsymbol{\beta})\bigr)^2,
\]
and therefore
\begin{align*}
|k||b(D\boldsymbol{\beta},\boldsymbol{\beta})|&\ge-\Im\text{m}\bigl(\overline{k}b(D\boldsymbol{\beta},\boldsymbol{\beta})\bigr)\\
&\ge
(\Im\text{m}\,k)\bigl(\ltwonorm{\mathrm{d}\boldsymbol{\omega}}{p+1}{\Omega\cup\Omega^\mathrm{c}}^2+
|k|^2\ltwonorm{\boldsymbol{\omega}}{p}{\Omega\cup\Omega^\mathrm{c}}^2\bigr).
\end{align*}
Finally,
\begin{align}
|b(D\boldsymbol{\beta},\boldsymbol{\beta})|&\ge\frac{\Im\text{m}\,k}{|k|}\min(1,|k|^2)\bigl(\ltwonorm{\mathrm{d}\boldsymbol{\omega}}{p+1}{\Omega\cup\Omega^\mathrm{c}}^2+
\ltwonorm{\boldsymbol{\omega}}{p}{\Omega\cup\Omega^\mathrm{c}}^2\bigr)\nonumber\\
&\ge \frac{1}{2}\frac{\Im\text{m}\,k}{|k|}\min(1,|k|^2)|\boldsymbol{\omega}|^2_{\hzero \pcomp{p}(\mathrm{d},\Omega\cup\Omega^\mathrm{c})},\label{ellipaux2}
\end{align}
where we used $a^2+b^2\ge (a+b)^2/2$.\par
On the other hand, from the continuity of the tangential trace operator we have
\begin{equation}\label{ellipaux3}
|\boldsymbol{\beta}|_{H_{\bot}^{-1/2}\pcomp{p}(\hhat{\mathrm{d}},\Gamma)}=
|\jump{\mathrm{t}}\boldsymbol{\omega}|_{H_{\bot}^{-1/2}\pcomp{p}(\hhat{\mathrm{d}},\Gamma)}
\le c|\boldsymbol{\omega}|_{\hzero \pcomp{p}(\mathrm{d},\Omega\cup\Omega^\mathrm{c})},
\end{equation}
with $c>0$ that depends only on $\Gamma$. Combining \eqref{ellipaux2} with \eqref{ellipaux3} proves \eqref{bioellip2}.\par
Equation \eqref{bioellip1a} follows immediately with \eqref{bioprop5b}, if we let $\boldsymbol{\gamma}=\hhat{*}\boldsymbol{\beta}$. For $k\ne 0$ \eqref{bioellip1} is obvious, since $V$ coincides with $\tilde{V}$ on $H_{\parallel}^{-1/2}\pcomp{p}(\hhat{\delta}0,\Gamma)$.
\item {\em Case $k=0$, $n=2$, $p=0$.} See \cite[Thm.\ 8.16]{mclean} for $V$ and \cite[Thm.\ 8.21]{mclean} for $D$.
\item {\em Case $k=0$, all other $n$ and $p$.} Consider the same domain as in the case $k\ne 0$, but with \eqref{symmaux3} as starting point. We rearrange as in \eqref{symmaux4}, with $\boldsymbol{\eta}=\boldsymbol{\omega}$, $k=0$, let $R\to\infty$ and take into account \eqref{decayaux1} or \eqref{decayaux2}, respectively,
\begin{equation}\label{ellipaux4}
-b(\jump{\gamma_\mathrm{N}}\boldsymbol{\omega},\mean{\gamma_\mathrm{D}}\boldsymbol{\omega})
-b(\mean{\gamma_\mathrm{N}}\boldsymbol{\omega}, \jump{\gamma_\mathrm{D}}\boldsymbol{\omega})
=\ltwo{\mathrm{d}\boldsymbol{\omega}}{\mathrm{d}\boldsymbol{\omega}}{p+1}{\Omega\cup\Omega^\mathrm{c}}.
\end{equation}
Pick $\boldsymbol{\omega}=\psl{p}\,\boldsymbol{\gamma}$, $\boldsymbol{\gamma}\in H_{\parallel}^{-1/2}\pcomp{p}(\hhat{\delta}0,\Gamma)$, which yields
\begin{equation}\label{ellipaux5}
b(\boldsymbol{\gamma},V\boldsymbol{\gamma})
=\ltwonorm{\mathrm{d}\boldsymbol{\omega}}{p+1}{\Omega\cup\Omega^\mathrm{c}}^2,
\end{equation}
where we used Lemma~\ref{lemma:lemmajump} and definition \eqref{defbio1}. On the other hand, from the continuity of the normal trace operator we have
\begin{equation}\label{ellipaux6}
|\boldsymbol{\gamma}|_{H_{\parallel}^{-1/2}\pcomp{p}(\hhat{\delta},\Gamma)}=
|[\mathrm{n}]\mathrm{d}\boldsymbol{\omega}|_{H_{\parallel}^{-1/2}\pcomp{p}(\hhat{\delta},\Gamma)}
\le c|\mathrm{d}\boldsymbol{\omega}|_{\hzero \pcomp{p+1}(\delta,\Omega\cup\Omega^\mathrm{c})},
\end{equation}
with $c>0$ that depends only on $\Gamma$. Since $\delta\mathrm{d}\boldsymbol{\omega}=0$ in $\Omega\cup\Omega^\mathrm{c}$ it holds that $|\mathrm{d}\boldsymbol{\omega}|_{\hzero \pcomp{p+1}(\delta,\Omega\cup\Omega^\mathrm{c})}=\ltwonorm{\mathrm{d}\boldsymbol{\omega}}{p+1}{\Omega\cup\Omega^\mathrm{c}}$, and combining \eqref{ellipaux5} and \eqref{ellipaux6} yields \eqref{bioellip1b}.\par
Finally, if we pick $\boldsymbol{\omega}=\pdl{p}\,\boldsymbol{\beta}$, \eqref{ellipaux4} yields the positive semidefiniteness
\[
b(D\boldsymbol{\beta},\boldsymbol{\beta})
=\ltwonorm{\mathrm{d}\boldsymbol{\omega}}{p+1}{\Omega\cup\Omega^\mathrm{c}}^2\ge 0,
\]
where we used Lemma~\ref{lemma:lemmajump} and definition \eqref{defbio4}.
\begin{remark}
The proof of \eqref{bioellip1b} critically hinges on the fact that $\delta\mathrm{d}\boldsymbol{\omega}=0$. In the light of \eqref{lpotprop1} and \eqref{lpotprop2}, this requires $\boldsymbol{\gamma}\in H_{\parallel}^{-1/2}\pcomp{p}(\hhat{\delta}0,\Gamma)$ rather than $\boldsymbol{\gamma}\in H_{\parallel}^{-1/2}\pcomp{p}(\hhat{\delta},\Gamma)$.
\end{remark}
\end{enumerate}
\qed\end{proof}
\section{Boundary Integral Equations}\label{sec:BIE}
In this section we discuss Calder\'on projectors and their properties under dual transformations. To study the case $k=0$ we introduce quotient spaces with respect to exact forms on the space of Maxwell solutions and the space of Dirichlet data, respectively. In a slight abuse of notation, we work with representatives of the equivalence classes as if they were actual elements of the quotient spaces.
\subsection{Calder\'on Projector for Interior and Exterior Problems}
Up to now, we have considered the representation formula for the domain $\Omega\cup\Omega^\mathrm{c}$. The boundary data have been defined in terms of the jumps on $\Gamma$ of Dirichlet, Neumann, and normal traces, respectively. In this section, we treat the interior domain $\Omega$ and the exterior domain $\Omega^\mathrm{c}$ separately.\par
\marginpar{Interior and exterior problems}An interior problem is a Maxwell-type problem with the additional requirement that $\omega = 0$ in $\Omega^\mathrm{c}$. The solution space effectively reduces to $X(\Omega)$. For an exterior problem we require that $\omega=0$ in $\Omega$, and its solutions therefore lie in $X(\Omega^\mathrm{c})$. Hence the boundary data acquire status of Cauchy data for the respective problem.\par
\marginpar{Calder\'on projector}Applying the Dirichlet trace $\gamma_\mathrm{D}$ and the Neumann trace $\gamma_\mathrm{N}$, respectively, to the representation formula \eqref{represent} yields the Calder\'on equations
\begin{equation}\label{calderon1}
\begin{pmatrix}\gamma_\mathrm{D}\\ \gamma_\mathrm{N}\end{pmatrix}\boldsymbol{\omega}
=P\begin{pmatrix}\gamma_\mathrm{D}\\ \gamma_\mathrm{N}\\ \mathrm{n}\end{pmatrix}\boldsymbol{\omega}
\end{equation}
for the interior problem, where
\[
P={\scriptstyle\frac{1}{2}}I+A.
\]
For the exterior problem, all traces in \eqref{calderon1} have to be replaced by the exterior traces, and $P$ reads
\[
P={\scriptstyle\frac{1}{2}}I-A.
\]
We have
\begin{align*}
I&=\begin{pmatrix}\mathrm{Id} & 0 & 0\\ 0 & \mathrm{Id} & 0\end{pmatrix},\\
A&=\begin{pmatrix}\mean{\gamma_\mathrm{D}}\\ \mean{\gamma_\mathrm{N}}\end{pmatrix}
\begin{pmatrix} -\pdlnot & \pslnot & \mathrm{d}\,\pslnot\end{pmatrix}
=\begin{pmatrix}-K & V & \hhat{\mathrm{d}}V\\
D & K^\dagger & 0
\end{pmatrix}.
\end{align*}
To arrive at the usual square form of $P$, the extra Neumann data $\mathrm{n}\boldsymbol{\omega}$ needs to be eliminated. $P$ is then called the Calder\'on projector. We distinguish two cases:
\begin{enumerate}
\item {\em Case $k\ne 0$.} The goal is achieved by leveraging \eqref{neumannsolution} and using the Maxwell single layer potential. With $Z=H_{\perp}^{-1/2}\pcomp{p}(\hhat{\mathrm{d}},\Gamma)\times H_{\parallel}^{-1/2}\pcomp{p}(\hhat{\delta},\Gamma)$ we obtain
\begin{equation}\label{calderon1a}
A=\begin{pmatrix}-K & \tilde{V}\\
D & K^\dagger
\end{pmatrix}:
Z\rightarrow Z.
\end{equation}
\item {\em Case $k=0$.} We eliminate the extra Neumann data by projecting the first Calder\'on equation onto the quotient space
\begin{equation}\label{defquotienth}
[H_{\perp}^{-1/2}\pcomp{p}(\hhat{\mathrm{d}},\Gamma)]=H_{\perp}^{-1/2}\pcomp{p}(\hhat{\mathrm{d}},\Gamma)/\hhat{\mathrm{d}}\,H_{\perp}^{-1/2}\pcomp{p-1}(\hhat{\mathrm{d}},\Gamma).
\end{equation}
Exact forms are mapped to zero, therefore the term containing the extra Neumann data vanishes. Inspection of \eqref{bioprop4} and \eqref{bioprop5a} reveals that $K$ and $D$ are well defined on the quotient space. Moreover, from \eqref{neumannsolution} we know that $\boldsymbol{\gamma}\in H_{\parallel}^{-1/2}\pcomp{p}(\hhat{\delta}0,\Gamma)$. Let $Z^0=[H_{\perp}^{-1/2}\pcomp{p}(\hhat{\mathrm{d}},\Gamma)]\times H_{\parallel}^{-1/2}\pcomp{p}(\hhat{\delta}0,\Gamma)$, to obtain
\begin{equation}\label{calderon1b}
A^0=\begin{pmatrix}-K & V\\
D & K^\dagger
\end{pmatrix}:
Z^0\rightarrow Z^0.
\end{equation}
The mapping property of the second equation can be seen from \eqref{bioprop3} and \eqref{bioprop5a}.\par
\end{enumerate}
Galerkin boundary element methods are based on a variational formulation of the Calder\'on projector. Again, we distinguish the two cases:
\begin{enumerate}
\item {\em Case $k\ne 0$.} Following \cite[Sec.\ 3.4]{buffa2003}\footnote{When comparing with \cite{buffa2003}, the following identifications hold: $\gamma_\mathrm{D}\to *\gamma_\mathrm{D}$, $k\gamma_\mathrm{N}\to -\gamma_\mathrm{N}$, $b(\vec{v},\vec{w})\to-b(\boldsymbol{\nu},*\overline{\boldsymbol{\omega}})$, $M\to K^\dagger$, $C\to-k*\tilde{V}$, $\vec{m}\to*\boldsymbol{\beta}$, $k\vec{j}\to-\gamma$.}, we introduce the {\em antisymmetric bilinear} form 
\[
B:Z\times Z\rightarrow\mathbb{C}:\left(\begin{pmatrix}\boldsymbol{\beta}\\\boldsymbol{\gamma}\end{pmatrix},\begin{pmatrix}\boldsymbol{\beta}^\prime\\\boldsymbol{\gamma}^{\,\prime}\end{pmatrix}\right)\mapsto\frac{1}{k}\bigl(
b(\boldsymbol{\gamma},\overline{\boldsymbol{\beta}}^\prime)
-b(\boldsymbol{\gamma}^{\,\prime},\overline{\boldsymbol{\beta}})
\bigr),
\]
where the factor $1/k$ is conventional. The related variational form of the boundary integral operator $A$ is
\[
\mathcal{A}:Z\times Z\rightarrow\mathbb{C}:(\boldsymbol{\alpha},\boldsymbol{\alpha}^\prime)\mapsto B(A\boldsymbol{\alpha},\boldsymbol{\alpha}^\prime).
\]
As a consequence of \eqref{biosym} we note that
\[
\mathcal{A}(\boldsymbol{\alpha},\boldsymbol{\alpha}^\prime)=\mathcal{A}(\boldsymbol{\alpha}^\prime,\boldsymbol{\alpha})
\]
holds, which generalizes \cite[Thm.\ 3.9]{buffa2003}.
\item {\em Case $k=0$.} We introduce the symmetric form 
\[
B^0:Z^0\times Z^0\rightarrow\mathbb{C}:\left(\begin{pmatrix}\boldsymbol{\beta}\\\boldsymbol{\gamma}\end{pmatrix},\begin{pmatrix}\boldsymbol{\beta}^\prime\\\boldsymbol{\gamma}^{\,\prime}\end{pmatrix}\right)\mapsto
b(\boldsymbol{\gamma},\boldsymbol{\beta}^\prime)
+b(\boldsymbol{\gamma}^{\,\prime},\boldsymbol{\beta}),
\]
which is non degenerate, since $[H_{\perp}^{-1/2}\pcomp{p}(\hhat{\mathrm{d}},\Gamma)]$ and $H_{\parallel}^{-1/2}\pcomp{p}(\hhat{\delta}0,\Gamma)$ are dual spaces with respect to the sesqui\-linear form $b(\cdot,\cdot)$. The related variational form of the boundary integral operator $A^0$ is
\[
\mathcal{A}^0:Z^0\times Z^0\rightarrow\mathbb{C}:(\boldsymbol{\alpha},\boldsymbol{\alpha}^\prime)\mapsto B^0(A^0\boldsymbol{\alpha},\boldsymbol{\alpha}^\prime).
\]
It holds that
\[
\mathcal{A}^0(\boldsymbol{\alpha},\boldsymbol{\alpha})\ge c\bigl(
|\hhat{\mathrm{d}}\boldsymbol{\beta}|_{H_{\bot}^{-1/2}\pcomp{p+1}(\hhat{\mathrm{d}},\Gamma)}^2+
|\boldsymbol{\gamma}|_{H_{\parallel}^{-1/2}\pcomp{p}(\hhat{\delta},\Gamma)}^2
\bigr)\quad\forall\boldsymbol{\alpha}=\begin{pmatrix}\boldsymbol{\beta}\\ \boldsymbol{\gamma}\end{pmatrix}\in Z^0,
\]
as a consequence of \eqref{biosym} and \eqref{bioellip}, which is useful for establishing coercivity.\par
The boundary integral operator in its variational form $\mathcal{A}^0$ is a major building block for the Galerkin boundary element methods presented in \cite{hiptmair2003,hiptmair2007}.
\end{enumerate}
\subsection{Equivalent Maxwell-Type Problems, Dual Transformations}
There is an inherent symmetry in the solution set $X^p(\Omega\cup\Omega^\mathrm{c})$ that can be exploited by so-called dual transformations. To show this, we need to distinguish the cases $k\ne 0$ and $k=0$:
\begin{enumerate}
\item {\em Case $k\ne 0$.}\marginpar{Dual transformation, $k\ne 0$}The dual transformation is defined by
\begin{equation}\label{defdualtrafo}
\Phi:\boldsymbol{\omega}\mapsto{\textstyle\frac{1}{ik}}*\mathrm{d}\boldsymbol{\omega}.
\end{equation}
\begin{lemma}
The dual transformation \eqref{defdualtrafo} constitutes an isomorphism $X^p(\Omega\cup\Omega^\mathrm{c})\rightarrow X^{q-1}(\Omega\cup\Omega^\mathrm{c})$.
\end{lemma}
\begin{proof}
We first show that $\Phi:X^p(\Omega\cup\Omega^\mathrm{c})\rightarrow Y^{q-1}(\Omega\cup\Omega^\mathrm{c})$ is injective, and then, that the image of $\Phi$ fulfills the radiation condition and hence lies in $X^{q-1}(\Omega\cup\Omega^\mathrm{c})$.\par
Let $\tilde{\boldsymbol{\omega}}=\Phi\boldsymbol{\omega}$. Since $\boldsymbol{\omega}\in X^p(\Omega\cup\Omega^\mathrm{c})$ we know at least that $\tilde{\boldsymbol{\omega}}\in \hzero _\mathrm{loc}\pcomp{q-1}(\mathrm{d},\Omega\cup\Omega^\mathrm{c})$ and can compute
\begin{equation}\label{dualaux1}
(-1)^{p+1}*^{-1}\mathrm{d}\tilde{\boldsymbol{\omega}}={\textstyle\frac{1}{ik}}\delta\mathrm{d}\boldsymbol{\omega}=-ik\boldsymbol{\omega}\in \hzero _\mathrm{loc}\pcomp{p}(\delta\mathrm{d},\Omega\cup\Omega^\mathrm{c}).
\end{equation}
We are therefore entitled to apply $*\mathrm{d}$ to find
\[
\delta\mathrm{d}\tilde{\boldsymbol{\omega}}=k^2\tilde{\boldsymbol{\omega}},
\]
which implies $\tilde{\boldsymbol{\omega}}\in Y^{q-1}(\Omega\cup\Omega^\mathrm{c})$. From \eqref{dualaux1} we read off
\[
\boldsymbol{\omega}=(-1)^{p(q-1)}{\textstyle\frac{1}{ik}}*\mathrm{d}\tilde{\boldsymbol{\omega}},
\]
which shows that 
\[
\Phi^2\boldsymbol{\omega}=(-1)^{p(q-1)}\boldsymbol{\omega},\quad\boldsymbol{\omega}\in X^p(\Omega\cup\Omega^\mathrm{c}),
\]
that is, $\Phi$ is injective.\par
It is of advantage to introduce a slightly different form of the radiation condition in connection with the above dual transformation. The condition is stated in terms of both, $\boldsymbol{\omega}$ and $\tilde{\boldsymbol{\omega}}$, and reads \cite[Def.\ 3.2 (iii)]{pauly2006}\footnote{When comparing with \cite{pauly2006}, the correct identifications are $(\boldsymbol{\omega},\tilde{\boldsymbol{\omega}})=(E,-*H)$. Moreover, $\boldsymbol{\omega}\in L^{2,p}_{>-\frac{1}{2}}(\Omega)$ is a weak version of $|\boldsymbol{\omega}|_{\iota \ppoint}=o(R^{-(n-1)/2})$.}
\begin{equation}\label{condradiation4}
\left.
\begin{aligned}
|\mathrm{i}_{\vec{n}}*^{-1}\tilde{\boldsymbol{\omega}}-\boldsymbol{\omega}|_{\iota \ppoint}&=o(R^{-(n-1)/2}),\\
|\mathrm{i}_{\vec{n}}*\boldsymbol{\omega}-(-1)^p\tilde{\boldsymbol{\omega}}|_{\iota \ppoint}&=o(R^{-(n-1)/2}).
\end{aligned}
\quad\right\}
\end{equation}
$\ppoint$ denotes an arbitrary point on $\Gamma^R$. The condition is invariant under dual transformations. Indeed, if we replace $(\boldsymbol{\omega},\tilde{\boldsymbol{\omega}})$ by $\Phi(\boldsymbol{\omega},\tilde{\boldsymbol{\omega}})=\bigl(\tilde{\boldsymbol{\omega}},(-1)^{p(q-1)}\boldsymbol{\omega}\bigr)$, and $p$ by $q-1$, then the equations \eqref{condradiation4} just exchange their roles.\par
The first equation in connection with \eqref{defdualtrafo} is identical to \eqref{condradiation1}. To investigate the second equation let us study its tangential and normal traces on $\Gamma^R$, more specifically
\begin{alignat*}{5}
(-1)^{p+1}& ik&&\hhat{*}^{-1}\mathrm{t}^R && \bigl(\mathrm{i}_{\vec{n}}*\boldsymbol{\omega}-(-1)^p\tilde{\boldsymbol{\omega}}\bigr)
&&= (\gamma_\mathrm{N}^R-ik\gamma_\mathrm{D}^R)\boldsymbol{\omega},\quad\text{and}\quad\\
&ik&&\hhat{*}^{-1}\mathrm{n}^R && \bigl(\mathrm{i}_{\vec{n}}*\boldsymbol{\omega}-(-1)^p\tilde{\boldsymbol{\omega}}\bigr)
&&= \hhat{\mathrm{d}}\gamma_\mathrm{D}^R\boldsymbol{\omega}.
\end{alignat*}
While the tangential part is equivalent to the first equation of \eqref{condradiation2}, the normal part gives an additional condition. We have thus shown that \eqref{condradiation4} is equivalent to \eqref{condradiation2} plus
\begin{equation}\label{condradiation5}
|\hhat{\mathrm{d}}\gamma_\mathrm{D}^R\boldsymbol{\omega}|_\ppoint=o(R^{-(n-1)/2}),
\end{equation}
which fixes the asymptotic behaviour of the tangential derivatives. However, it can be shown that solutions of the Maxwell-type equation subject to radiation condition \eqref{condradiation1} necessarily fulfill \eqref{condradiation5}, so that we see that no additional constraint has been introduced. This is related to the asympotic behavior \eqref{greenradiation1b} of the fundamental solution.\,\footnote{Compare also \cite[p.\ 163]{kress1998}, which says that both equations \eqref{condradiation4} are equivalent for $n=3$, $p=1$.}\par
It follows from the equivalence of \eqref{condradiation4} and \eqref{condradiation1} and from the invariance of \eqref{condradiation4} under dual transformations, that $\Phi$ maps to $X^{q-1}(\Omega\cup\Omega^\mathrm{c})$.
\qed\end{proof}
\begin{remark}
For $n=3$, $p=1$, the dual transformation \eqref{defdualtrafo} describes the symmetry between electric and magnetic fields encountered in the propagation of harmonic waves in simple media. 
\end{remark}
\begin{framed}
Let $(\vec{E},\vec{H})$ denote the vector proxies of the 1-forms $(\boldsymbol{\omega},\tilde{\boldsymbol{\omega}})$. Then \eqref{defdualtrafo} and \eqref{condradiation4} read
\begin{gather*}
\vec{H}=\frac{1}{ik}\textbf{curl}\,\vec{E}\\[-\baselineskip]
\intertext{and}\\[-2\baselineskip]
\begin{aligned}
|\vec{H}\times\vec{n}-\vec{E}|_{\iota \ppoint}&=o(R^{-1}),\\
|\vec{E}\times\vec{n}+\vec{H}|_{\iota \ppoint}&=o(R^{-1}),
\end{aligned}
\end{gather*}
respectively, compare \cite[Thm.\ 6.4, Def.\ 6.5]{kress1998}. This is yet another way to state the Silver-M{\"u}ller radiation conditions.
\end{framed}\noindent
We introduce an adapted set of trace operators
\begin{alignat*}{3}
&\traceadapt_{\mathrm{D}}=ik\gamma_{\mathrm{D}}:\;&&\hzero \pcomp{p}(\mathrm{d},\Omega)&&\rightarrow H_{\perp}^{-1/2}\pcomp{p}(\hhat{\mathrm{d}},\Gamma),\\
&\traceadapt_{\mathrm{N}}=\hhat{*}\gamma_{\mathrm{N}}:\;&&\hzero \pcomp{p}(\delta\mathrm{d},\Omega)&&\rightarrow H_{\perp}^{-1/2}\pcomp{q-1}(\hhat{\mathrm{d}},\Gamma).
\end{alignat*}
It is easy to see from the definitions that these trace operators fulfill
\begin{subequations}\label{dualtraceprop}
\begin{align}
\traceadapt_\mathrm{D}&=(-1)^{p(q-1)}\traceadapt_\mathrm{N}\Phi,\\
\traceadapt_\mathrm{N}&=\traceadapt_\mathrm{D}\Phi.
\end{align}
\end{subequations}
In sum, $\Phi$ maps Maxwell solutions onto Maxwell solutions and interchanges the roles of the adapted boundary data, up to sign.\par
This symmetry can be exploited to render \eqref{calderon1a} in the following symmetric form
\begin{equation}\label{calderon2}
{\textstyle\frac{1}{2}}\begin{pmatrix}\traceadapt_\mathrm{D}\boldsymbol{\omega}\\ \traceadapt_\mathrm{N}\boldsymbol{\omega}\end{pmatrix}
=\begin{pmatrix} -K & (-1)^{p(q-1)}C\\
C & -K
\end{pmatrix}
\begin{pmatrix}\traceadapt_\mathrm{D}\boldsymbol{\omega}\\ \traceadapt_\mathrm{N}\boldsymbol{\omega}\end{pmatrix},
\end{equation}
where we introduced a boundary integral operator $C$ according to
\[
C=ik\tilde{V}\hhat{*}:H_{\perp}^{-1/2}\pcomp{p}(\hhat{\mathrm{d}},\Gamma)\rightarrow H_{\perp}^{-1/2}\pcomp{q-1}(\hhat{\mathrm{d}},\Gamma),\quad k\ne 0,
\]
and used \eqref{bioprop2} and \eqref{bioprop5b}.\par
From \eqref{dualtraceprop} it can be inferred that the Calder\'on equation \eqref{calderon2} is invariant under dual transformations \eqref{defdualtrafo}, up to sign.\par\noindent
\begin{remark}
\begin{enumerate}
\item[]
\item[1.] This result implies that electric field and magnetic field based formulations share the same Calder\'on equation \eqref{calderon2}, up to sign of the off-diagonal elements.
\item[2.] The analysis of Maxwell transmission problems in Lipschitz domains in \cite[p.\ 468]{buffa2003} is essentially based on the Calder\'on projector in \eqref{calderon2}.\,\footnote{The formulation in \cite{buffa2003} can be stated in $H_{\parallel}^{-1/2}\pcomp{1}(\hhat{\delta},\Gamma)$ with the help of the translation isomorphisms. The following identifications hold: $ik\gamma_\mathrm{D}\to\hhat{*}\traceadapt_\mathrm{D}$, $k\gamma_\mathrm{N}\to\hhat{*}\traceadapt_\mathrm{N}$, $C\to i\hhat{*}C\hhat{*}^{-1}$, $M\to-\hhat{*}^{-1}K\hhat{*}$.}
\end{enumerate}
\end{remark}
\item {\em Case $k=0$.} We first define a quotient space on $X^p(\Omega\cup\Omega^\mathrm{c})$ by
\begin{align}
X^p_0(\Omega\cup\Omega^\mathrm{c})&=X^p(\Omega\cup\Omega^\mathrm{c})\cap\mathrm{d}\hzero \pcomp{p-1}(\mathrm{d},\Omega\cup\Omega^\mathrm{c}),\nonumber\\[0.2\baselineskip]
[X^p(\Omega\cup\Omega^\mathrm{c})]&=X^p(\Omega\cup\Omega^\mathrm{c})/X^p_0(\Omega\cup\Omega^\mathrm{c}).\label{defquotientx}
\end{align}
\marginpar{Dual transformation, $k=0$}
The dual transformation 
\begin{equation*}
\Phi^0:[X^p(\Omega\cup\Omega^\mathrm{c})]\rightarrow [X^{q-2}(\Omega\cup\Omega^\mathrm{c})]:\boldsymbol{\omega}\mapsto\tilde{\boldsymbol{\omega}}
\end{equation*}
can be defined implicitly, in two equivalent ways:
\begin{enumerate} 
\item[1.] The dual transform fulfills
\begin{equation}\label{dt0}
\mathrm{d}\tilde{\boldsymbol{\omega}}=*\mathrm{d}\boldsymbol{\omega}.
\end{equation}
Extend the definition \eqref{defboundary1} of the boundary data $(\boldsymbol{\beta},\boldsymbol{\gamma})$ to the quotient spaces \eqref{defquotienth} and \eqref{defquotientx}. Applying $\jump{\mathrm{t}}$ and $\jump{\mathrm{n}}$, and using \eqref{commutetd} and \eqref{defnormaltrace}, it can be seen that \eqref{dt0} implies a transformation
\begin{gather*}
[H_{\perp}^{-1/2}\pcomp{p}(\mathrm{d},\Gamma)]\times H_{\parallel}^{-1/2}\pcomp{p}(\delta 0,\Gamma)\\
\rightarrow [H_{\perp}^{-1/2}\pcomp{q-2}(\mathrm{d},\Gamma)]\times H_{\parallel}^{-1/2}\pcomp{q-2}(\delta 0,\Gamma):
\\[1mm]
(\boldsymbol{\beta},\boldsymbol{\gamma})\mapsto
(\tilde{\boldsymbol{\beta}},\tilde{\boldsymbol{\gamma}})
\end{gather*}
so that
\begin{equation}\label{dt0ab}
\left.\begin{aligned}
\hhat{\mathrm{d}}\tilde{\boldsymbol{\beta}}&=\hhat{*}\boldsymbol{\gamma},\\
\tilde{\boldsymbol{\gamma}}&=(-1)^{p+1}\hhat{*}\hhat{\mathrm{d}}\beta.
\end{aligned}\quad\right\}
\end{equation}
\item[2.]Building on \eqref{dt0ab}, we may also define the dual transformation by
\begin{equation}
\tilde{\boldsymbol{\omega}}=\psl{q-2}\tilde{\boldsymbol{\gamma}}-\pdl{q-2}\tilde{\boldsymbol{\beta}}.\label{dualaux3}
\end{equation}
\end{enumerate}
\begin{lemma}The two definitions \eqref{dt0} and \eqref{dualaux3} of the dual transformation are well defined; they are equivalent and constitute an isomorphism.
\end{lemma}
\begin{proof}
We first show that \eqref{dualaux3} is well defined and implies \eqref{dt0}. Check that for $\boldsymbol{\omega}\in[X^p(\Omega\cup\Omega^\mathrm{c})]$ the boundary data $(\boldsymbol{\beta},\boldsymbol{\gamma})$ in \eqref{dt0ab} are well defined. Since $*\boldsymbol{\gamma}$ is closed, it can be represented by a potential $\tilde{\boldsymbol{\beta}}$. In general, there might be a contribution from cohomology as well, which can be discarded, since we restricted ourselves to trivial topology. Existence of the potential is guaranteed from the exact sequence property \eqref{exactsequence3}. The potential is defined up to exact forms, so it corresponds to exactly one element of the quotient space $[H_{\perp}^{-1/2}\pcomp{q-2}(\mathrm{d},\Gamma)]$.\footnote{We consider \eqref{dt0ab} as a mere theoretical device, but it has also profound practical implications. For $n=3$, $p=1$ $\tilde{\boldsymbol{\beta}}$ is a scalar potential. In \cite[Ch.\ 8]{hiptmair2003}, a scalar stream function $\tilde{\boldsymbol{\beta}}_h$ is employed on the discrete level to span solenoidal surface Raviart-Thomas covector fields $\boldsymbol{\gamma}_h$. This provides a discrete subspace of $H_{\parallel}^{-1/2}\pcomp{1}(\delta 0,\Gamma)$.} The last equation in \eqref{dt0ab} poses no further difficulties.\par
Equation \eqref{dualaux3} results from the representation formula \eqref{represent}, when applied to $(\tilde{\boldsymbol{\beta}},\tilde{\boldsymbol{\gamma}})$ and projected onto the quotient spaces. The exact potential $\mathrm{d}\,\psl{p-1}\,\boldsymbol{\varphi}$ lies in the zero class and therefore does not appear in \eqref{dualaux3}. In the light of \eqref{lpotprop4}, a different choice of the representative $\tilde{\boldsymbol{\beta}}$ in \eqref{dt0ab} leads to a different representative $\tilde{\boldsymbol{\omega}}$ in \eqref{dualaux3}. However, the map is well defined in terms of the quotient spaces. The definition via the representation formula ensures that $\tilde{\boldsymbol{\omega}}\in[X^{q-2}(\Omega\cup\Omega^\mathrm{c})]$. We have thus shown that the dual transformation $\Phi^0$ based on \eqref{dt0ab} and \eqref{dualaux3} is well defined.\par
To show that $\tilde{\boldsymbol{\omega}}$ fulfills \eqref{dt0} we first represent $\boldsymbol{\omega}$ in terms of its boundary data $(\boldsymbol{\beta},\boldsymbol{\gamma})$, in analogy to \eqref{dualaux3}, and apply $*\mathrm{d}$. We obtain
\begin{align*}
*\mathrm{d}\omega&=*\mathrm{d}\,\psl{}\gamma-*\mathrm{d}\,\pdl{}\beta\\
&=-\pdl{}\hhat{*}\gamma+(-1)^{p+1}\mathrm{d}\,\psl{}\hhat{*}\hhat{\mathrm{d}}\beta,
\end{align*}
where we used \eqref{defdl} and \eqref{lpotprop7}. When we plug in \eqref{dt0ab} this yields
\[
*\mathrm{d}\omega=\mathrm{d}(\psl{}\tilde{\boldsymbol{\gamma}}-\pdl{}\tilde{\boldsymbol{\beta}})
=\mathrm{d}\tilde{\boldsymbol{\omega}},
\]
which proves the assertion.\par
Next we show that \eqref{dt0} is well defined and implies \eqref{dualaux3}. Existence of $\tilde{\boldsymbol{\omega}}$ is guaranteed from the first part of the proof. Uniqueness can be seen as follows: We know that \eqref{dt0} implies \eqref{dt0ab}. If there were several solutions for $\tilde{\boldsymbol{\omega}}$, their boundary data $(\tilde{\boldsymbol{\beta}},\tilde{\boldsymbol{\gamma}})$ had to coincide, because we know from the first part of the proof that \eqref{dt0ab} fixes $(\tilde{\boldsymbol{\beta}},\tilde{\boldsymbol{\gamma}})$ for given $(\boldsymbol{\beta},\boldsymbol{\gamma})$. But from Theorem~\ref{thm:repthm} we know that elements of $[X^{q-2}(\Omega\cup\Omega^\mathrm{c})]$ can be uniquely represented in terms of their boundary data. Therefore $\tilde{\boldsymbol{\omega}}$ is unique and coincides with the solution obtained from \eqref{dualaux3}.\par
It remains to show that $\Phi^0$ constitutes an isomorphism. To that end we apply the map \eqref{dt0ab} to the boundary data twice, which yields 
\[
(\tilde{\tilde{\boldsymbol{\beta}}},\tilde{\tilde{\boldsymbol{\gamma}}})=(-1)^{(p+1)(q-1)}(\boldsymbol{\beta},\boldsymbol{\gamma}).
\]
Arguing like in the second part reveals
\begin{eqnarray}\label{phi0iso}
(\Phi^0)^2\boldsymbol{\omega}=(-1)^{(p+1)(q-1)}\boldsymbol{\omega},\quad\boldsymbol{\omega}\in [X^p(\Omega\cup\Omega^\mathrm{c})],
\end{eqnarray}
which concludes the proof.
\qed\end{proof}
We introduce yet another set of trace operators
\begin{alignat*}{3}
&\traceadapt_{\mathrm{D}}^{\,0}=\hhat{\mathrm{d}}\gamma_{\mathrm{D}}:\;&&\hzero \pcomp{p}(\mathrm{d},\Omega)&&\rightarrow H_{\perp}^{-1/2}\pcomp{p+1}(\hhat{\mathrm{d}},\Gamma),\\
&\traceadapt_{\mathrm{N}}^{\,0}=\traceadapt_{\mathrm{N}}=\hhat{*}\gamma_{\mathrm{N}}:\;&&\hzero \pcomp{p}(\delta\mathrm{d},\Omega)&&\rightarrow H_{\perp}^{-1/2}\pcomp{q-1}(\hhat{\mathrm{d}},\Gamma).
\end{alignat*}
Note that $\traceadapt_{\mathrm{D}}^{\,0}$ and $\traceadapt_{\mathrm{N}}^{\,0}$ are well defined on $[X^p(\Omega]$. From \eqref{dt0} and \eqref{phi0iso} it is immediate that these trace operators fulfill
\begin{subequations}\label{dualtraceprop0}
\begin{align}
\traceadapt_\mathrm{D}^{\,0}&=(-1)^{(p+1)(q-1)}\traceadapt_\mathrm{N}^{\,0}\Phi^0,\\
\traceadapt_\mathrm{N}^{\,0}&=\traceadapt_\mathrm{D}^{\,0}\Phi^0.
\end{align}
\end{subequations}
Equation \eqref{calderon1b} can be rearranged in the following symmetric form
\begin{equation}\label{calderon3}
{\textstyle\frac{1}{2}}\begin{pmatrix}\traceadapt_\mathrm{D}^{\,0}\boldsymbol{\omega}\\ \traceadapt_\mathrm{N}^{\,0}\boldsymbol{\omega}\end{pmatrix}
=\begin{pmatrix} -K & (-1)^{(p+1)(q-1)}B\\
B & -K
\end{pmatrix}
\begin{pmatrix}\traceadapt_\mathrm{D}^{\,0}\boldsymbol{\omega}\\ \traceadapt_\mathrm{N}^{\,0}\boldsymbol{\omega}\end{pmatrix},
\end{equation}
$p=\mathrm{deg}\,\boldsymbol{\omega}$, where we introduced a boundary integral operator $B$ according to
\[
B=(-1)^p\hhat{\mathrm{d}}\,V\hhat{*}:H_{\perp}^{-1/2}\pcomp{p}(\hhat{\mathrm{d}},\Gamma)\rightarrow H_{\perp}^{-1/2}\pcomp{q}(\hhat{\mathrm{d}},\Gamma),
\]
and used \eqref{bioprop2}, \eqref{bioprop4} and \eqref{bioprop5a}. From \eqref{dualtraceprop0} it can be inferred that Calder\'on equation \eqref{calderon3} is invariant under dual transformations \eqref{dt0}, up to sign.\par\noindent
\begin{framed}
For $n=3$, $p=1$ this invariance means that magnetostatic vector-potential and scalar-potential formulations share basically the same Calder\'on equation \eqref{calderon3}. This confirms the observation in \cite[Sec.\ 9]{hiptmair2003} which states that the boundary integral operators associated with the $\textbf{curl}\,\textbf{curl}$ equation are structurally equal to those of scalar second order elliptic problems.\par
Let $\gamma_n\vec{B}\in H^{-1/2}(\partial\Omega)$ be the normal trace of the magnetic flux density, and $\gamma_\tau\vec{H}\in\mathbf{H}^{-1/2}(\mathrm{div}_\Gamma0,\partial\Omega)$ the tangential trace of the magnetic field. Equation \eqref{calderon3} translates for $n=3$, $p=1$ into
\begin{equation}\label{calderon3vec}
{\textstyle\frac{1}{2}}\begin{pmatrix}\gamma_n\vec{B}\\ \gamma_\tau\vec{H}\end{pmatrix}
=\begin{pmatrix} \mathsf{K}^\dagger & \mathrm{curl}_\Gamma\mathsf{A}\\
\textbf{curl}_\Gamma\mathsf{V} & \mathsf{B}
\end{pmatrix}
\begin{pmatrix}\gamma_n\vec{B}\\ \gamma_\tau\vec{H}\end{pmatrix}.
\end{equation}
The flux density can be represented by a magnetic vector potential $\vec{A}$, where $\pi_\tau\vec{A}\in\mathbf{H}^{-1/2}(\mathrm{curl}_\Gamma,\partial\Omega)$, $\gamma_n\vec{B}=\gamma_n\textbf{curl}\vec{A}=\mathrm{curl}_\Gamma\pi_\tau\vec{A}$. Moreover, for a trivial topology, the magnetic field can be represented by a magnetic scalar potential $\psi$, where $\gamma\psi\in H^{1/2}(\partial\Omega)$, $\gamma_\tau\vec{H}=\gamma_\tau\textbf{grad}\psi=\textbf{curl}_\Gamma\gamma\psi$. Equation \eqref{calderon3vec} may be expressed in terms of these potentials in different ways. The operators $\mathrm{curl}_\Gamma$ and $\textbf{curl}_\Gamma$ inherit $L^2$ adjointness from $\mathrm{d}$ and $\delta$, which is useful to proceed to variational forms of the equations.\par
All the formulations have in common that they are built upon the scalar and vectorial single and double layer operators, plus straightforward surface curl operators. Hence only this set of boundary-integral operators needs to be implemented to cover both, scalar and vector potential formulations.
\end{framed}
\end{enumerate}
\section{Conclusions}
We have shown that the theory of boundary-integral equations for Maxwell-type problems can be formulated entirely in the framework of differential-form calculus. The basic toolkit is given by Sobolev spaces on the problem domain and their traces on the domain's Lipschitz boundary, as well as by integral transformations and fundamental solutions of Maxwell-type equations. The presented proofs concerning the representation formula, layer potentials, and boundary-integral operators follow the strategies of classical vector-analysis literature. Their scope, however, extends beyond scalar- or vector-potential formulations in dimensions two and three. Nonetheless, it is the authors' opinion that the main advantage of employing differential-form calculus in this field of applied mathematics does not lie in the prospect of extending the theory beyond three dimensions; it lies in the emphasis that is put on structural considerations and abstract viewpoints. For instance consider the properties of layer potentials and boundary integral operators that are presented in Lemmata~\ref{lemma:lpotprop} and \ref{lemma:bioprop}, respectively. As a further example, we have presented the invariance of the Calder\'on projectors under dual transformations, which not only strikes us by its elegance, but also has a direct impact on our implementation of the boundary-element method for magnetostatics: in three dimensions, a single set of boundary integral operators covers both, the scalar- and the vector-potential case. 
%
%  For working with BiBTeX
%%%\bibliographystyle{spbasic}
%%%\bibliography{lit1000,papers,local}
%
%%%\end{document}

\end{document}